\documentclass[a4paper,10pt, reqno]{amsart}

\usepackage[main=english]{babel}

\usepackage{amsmath,amssymb,latexsym, amsfonts}
\usepackage{verbatim}
\usepackage{newcent}
\usepackage{mathbbol}
\usepackage[colorlinks=true, allcolors=black]{hyperref}
\usepackage{mathabx}
\usepackage{bbm}
\usepackage{pifont}
\usepackage{manfnt}

\usepackage[dvipsnames]{xcolor}

\usepackage{tikz-cd}
\usepackage{mathtools}

\usepackage{upgreek}  

\usepackage{tikz}

\usepackage[all,ps]{xy}
\usepackage{color}

\usepackage{stmaryrd}

\usepackage{capt-of}

\makeatletter
\@namedef{subjclassname@2020}{%
  \textup{2020} Mathematics Subject Classification}
\makeatother

\usepackage[OT2, T1]{fontenc} 

\usepackage[textsize=scriptsize]{todonotes}

\usepackage{scalerel,stackengine}  

\newcommand\pig[1]{\scalerel*[5pt]{\big#1}{%
  \ensurestackMath{\addstackgap[1.5pt]{\big#1}}}}

\newcommand{\compactlist}[1]{\setlength{\itemsep}{0pt} \setlength{\parskip}{0pt} \setlength{\leftskip}{-0.#1em}}

\numberwithin{equation}{section}

\DeclareRobustCommand{\SkipTocEntry}[5]{}

\theoremstyle{plain}

\newtheorem{theorem}{Theorem}[section]

\newtheorem{prop}[theorem]{Proposition}

\newtheorem{lem}[theorem]{Lemma}

\newtheorem{cor}[theorem]{Corollary}

\theoremstyle{definition}

\newtheorem{dfn}[theorem]{Definition}

\newtheorem{rem}[theorem]{Remark}

\newtheorem{notation}[theorem]{Notation}

\newcommand{\ahha}{{\scriptscriptstyle{A}}}

\newcommand{\emme}{{\scriptscriptstyle{M}}}
\newcommand{\enne}{{\scriptscriptstyle{N}}}

\newcommand{\uhhu}{{\scriptscriptstyle{U}}}
\newcommand{\pehhe}{{\scriptscriptstyle{P}}}

\newcommand{\ikks}{{\scriptscriptstyle{X}}}

 \newcommand{\N}{{\mathbb{N}}}

\newcommand{\ga}{\alpha}

\newcommand{\gd}{\delta} 
\newcommand{\gD}{\Delta} 
\newcommand{\gve}{\varepsilon} 
  
\newcommand{\gvf}{\varphi}  

\newcommand{\gl}{\lambda}

\newcommand{\gs}{\sigma}

\newcommand{\gt}{\theta} 
\newcommand{\gvt}{\vartheta} 

\newcommand{\cC}{{\mathcal C}}
\newcommand{\cD}{{\mathcal D}}
\newcommand{\cE}{{\mathcal E}}

\newcommand{\cM}{{\mathcal M}}

\newcommand{\cZ}{{\mathcal Z}}

\newcommand{\Hom}{\operatorname{Hom}}
\newcommand{\HOM}{\operatorname{HOM}}

\newcommand{\Tor}{\operatorname{Tor}}
\newcommand{\Ext}{\operatorname{Ext}}
\newcommand{\Coext}{\operatorname{Coext}}
\newcommand{\Cotor}{\operatorname{Cotor}}

\newcommand{\id}{{\rm id}}

\newcommand{\tr}{{\rm tr}}

\newcommand{\due}[3]{{}_{{#2 }} {#1}_{{ #3}}\,}    

\newcommand{\pl}{\partial}

\newcommand{{\Hl}}{{H^{\ell}}} 
\newcommand{{\mHop}}{{m_{H^{\rm op}}}} 
\newcommand{{\Hop}}{{H^{\rm op}}} 
\newcommand{{\mUop}}{{m_{U^{\rm op}}}} 
\newcommand{{\mUopp}}{{m_{\scriptscriptstyle{U^{\rm op}}}}} 
\newcommand{{\Uop}}{{U^{\rm op}}}
\newcommand{{\mVop}}{{m_{V^{\rm op}}}} 
\newcommand{{\Vop}}{{V^{\rm op}}}  
\newcommand{{\Ae}}{{A^{\rm e}}}
\newcommand{{\Be}}{{B^{\rm e}}}
\newcommand{{\Ue}}{{U^{\rm e}}}
\newcommand{{\He}}{{H^{\rm e}}}
\newcommand{{\Aop}}{{A^{\rm op}}}
\newcommand{{\Aope}}{({A^{\rm op}})^{\rm e}}
\newcommand{{\Aopl}}{{A^{\rm op}_\pl}}

\newcommand{{\Bop}}{{B^{\rm op}}}
\newcommand{{\Bopp}}{{\scriptscriptstyle{{B^{\rm op}}}}}
\newcommand{{\Bope}}{({B^{\rm op}})^{\rm e}}
\newcommand{{\Bpl}}{{B_\pl}}

\newcommand{{\op}}{{{\rm op}}}
\newcommand{{\coop}}{{{\rm coop}}}
\newcommand{{\sop}}{{*^{\rm op}}}
\newcommand{{\co}}{{{\rm co}}}

\newcommand{\kmod}{k\mbox{-}\mathbf{Mod}}                     %
\newcommand{\amoda}{A^{\rm e}\mbox{-}\mathbf{Mod}}                  %
\newcommand{\umod}{U\mbox{-}\mathbf{Mod}}                     
\newcommand{\modu}{\mathbf{Mod}\mbox{-}U}         %
\newcommand{\yd}{{}^\uhhu_\uhhu\mathbf{YD}}                     
\newcommand{\ayd}{{}^\uhhu\!\mathbf{aYD}_\uhhu}                     
\newcommand{\sayd}{{}^\uhhu\!\mathbf{saYD}_\uhhu}                     %
\newcommand{\cayd}{ {}_\uhhu \mathbf{aYD}^{\scriptscriptstyle{\rm contra-}\uhhu}}                     %
\newcommand{\csayd}{ {}_\uhhu \mathbf{saYD}^{\scriptscriptstyle{\rm contra-}\uhhu}}                     %

\newcommand{\contramodu}{\mathbf{Contramod}\mbox{-}U}

\newcommand{\comodu}{\mathbf{Comod}\mbox{-}U}         
\newcommand{\ucomod}{U\mbox{-}\mathbf{Comod}}

\newcommand{\lact}{\smalltriangleright}                  
\newcommand{\ract}{\smalltriangleleft}
\newcommand{\blact}{\blacktriangleright}  
\newcommand{\bract}{\blacktriangleleft}

\newcommand{{\gog}}{{G \rightrightarrows G_0}}

\newcommand{{\rra}}{\rightrightarrows}

\newcommand{{\lra}}{\ \longrightarrow \ }
\newcommand{{\lla}}{\ \longleftarrow \ }
\newcommand{{\lma}}{\ \longmapsto \ }

\newcommand{{\bull}}{{\scriptscriptstyle{\bullet}}}
\newcommand{{\qqquad}}{{\quad\quad\quad}}
\newcommand{\Aopp}{{\scriptscriptstyle{\Aop}}}

\newsavebox{\foobox}

\newcommand{\pmact}{\mbox{ \raisebox{-1pt}{\ding{226}} }}
\newcommand{\mpact}{\mbox{ \raisebox{-1pt}{\ding{227}} }}
\newcommand{\umact}{\mbox{\hspace*{-1.5pt} \scalebox{0.75}{\rotatebox{137}{\raisebox{-2pt}{ \manfilledquartercircle}} }\!}}

\newcommand{\copact}{\mbox{ \raisebox{6pt}{\rotatebox{180}{{\ding{226}}}} }}

\newcommand{\mpsqact}{\mbox{ \scalebox{0.85}{\rotatebox{180}{\raisebox{-6.0pt}{\ding{238}}} }}}

\newcommand{\sma}[1]{\raisebox{1pt}{${\scriptstyle ({#1})}$}}
\newcommand{\smap}{\raisebox{0.3pt}{${{\scriptscriptstyle {[+]}}}$}}
\newcommand{\smam}{\raisebox{0.3pt}{${{\scriptscriptstyle {[-]}}}$}}

\newcommand{\mancino}{{\,\scalebox{0.7}{\rotatebox{90}{\mancone}}\,}}

\newcommand{\fren}{\mbox{\hspace*{1pt}\scalebox{0.75}{\ding{228}}}}
\newcommand{\frenop}{\mbox{\raisebox{5pt}{\rotatebox{180}{{\scalebox{0.75}{\ding{228}}}}}\hspace*{1pt}}}

\begin{document}

\title{Centres, trace functors, and cyclic cohomology}

\author{Niels Kowalzig}

\begin{abstract}
We study the biclosedness of the monoidal categories of modules and comodules over a (left or right) Hopf algebroid, along with their bimodule category centres over the respective opposite categories and a corresponding categorical equivalence to anti Yetter-Drinfel'd contramodules and  anti Yetter-Drinfel'd modules, respectively. This is directly connected to the existence  of a trace functor on the monoidal categories of modules and comodules in question, which in turn allows to recover (or define) cyclic operators enabling cyclic cohomology.   
\end{abstract}

\address{Dipartimento di Matematica, Universit\`a di Roma Tor Vergata, Via della Ricerca Scientifica 1, 00133 Roma, Italy}

\email{niels.kowalzig@uniroma2.it}

%



\maketitle

\setcounter{tocdepth}{2}
\tableofcontents

\section{Introduction}

\addtocontents{toc}{\protect\setcounter{tocdepth}{1}}

Introducing potential coefficients in cyclic homology or cohomology typically asks for more than one algebraic structure in order to obtain from the underlying chain or cochain complex a paracyclic (or duplicial) object in the sense of Connes \cite{Con:CCEFE}.
For example,
in those cyclic theories induced by a Hopf structure on the underlying ring or coring, 
coefficients might be simultaneously modules and comodules or simultaneously modules and contramodules, whereas the underlying (simplicial) complex usually only needs one of
these structures.
However, 
the presence of two structures instead of one may not always be immediately recognised as one of them may be trivial and therefore invisible. This, for example, sometimes happens for bialgebras or bialgebroids with special properties, such as commutativity or cocommutativity.

Whereas up to this point no compatibility between these two algebraic structures is required,
passing from paracyclic to cyclic objects, {\em i.e.}, those in which
the cyclic operator powers to the identity,
in general asks for some sort of compatibility condition, which leads to the notion of {\em (stable anti) Yetter-Drinfel'd modules} resp.\ {\em stable anti Yetter-Drinfel'd contramodules} in the two cases of module-comodule resp.\ module-contramodule mentioned above, which expresses what happens if action is followed by coaction, and vice versa, resp.\ contraaction followed by action, and vice versa again; see, just to name a few,
\cite{BenPerWit:YDMUCT, BoeSte:CCOBAAVC, Brz:HCHWCC, BulCaePan:YDCFQHA, Dri:QG, JarSte:HCHARCHOHGE, HajKhaRanSom:SAYD, Kay:BCHWC, Kow:WEIABVA, PanSta:GAYDMACOABTC, RadTow:YDCATAAB, Yet:QGAROMC}
for these notions in various contexts. For example, as explained in \cite{Kow:ANCCOTCDOE},
without specifying  the technical details here,
if $U$ is a left Hopf algebroid (for example, a Hopf algebra or the enveloping algebra $\Ae$ of an associative algebra or still the enveloping algebra of a Lie algebroid) with respect to which $N$ is a Yetter-Drinfel'd module, $M$ a stable anti Yetter-Drinfel'd module, and $P$ a stable anti Yetter-Drinfel'd contramodule,
then (under suitable projectivity resp.\ flatness assumptions), the (co)chain complexes computing the various derived functors
$\Tor_\bull^\uhhu(N,M)$,
$\Ext^\bull_\uhhu(N,P)$,
$\Cotor_\bull^\uhhu(N,M)$, 
and
$\Coext_\bull^\uhhu(N,P)$
can be made into cyclic modules, which, in particular, implies the existence of (co)cyclic differentials of degree $\pm 1$:
\begin{small}
\begin{equation*}
  \begin{array}{rlclrlcl}
B \colon \!\!\!&\Tor_\bull^\uhhu(N,M) \!\!\!&\to\!\!\!& \Tor_{\bull+1}^\uhhu(N,M),
\quad & \quad
B \colon \!\!\!&\Cotor^\bull_\uhhu(N,M) \!\!\!&\to\!\!\!& \Cotor^{\bull-1}_\uhhu(N,M),
\\[1pt]
B \colon \!\!\!&\Ext^\bull_\uhhu(N,P) \!\!\!&\to\!\!\!& \Ext^{\bull-1}_\uhhu(N,P),
\quad & \quad
B \colon \!\!\!&\Coext_\bull^\uhhu(N,P) \!\!\!&\to\!\!\!& \Cotor_{\bull+1}^\uhhu(N,P),
  \end{array}
\end{equation*}
\end{small}
by abuse of notation all denoted by the same symbol $B$ here, that is, the (induced) {\em Connes-Rinehart-Tsygan (co)boundary} in its various guises.

\subsection{Aims and objectives}
In contrast to Yetter-Drinfel'd kind of objects being interpreted as monoidal centres \cite{Schau:DADOQGHA}, a categorical understanding of {\em anti} Yetter-Drinfel'd objects is only beginning to emerge.
The main objective of this article is to embed the two cases of anti Yetter-Drinfel'd objects mentioned above in a more categorical setting, inspired by and generalising the ideas in \cite{Sha:OTAYDMCC, KobSha:ACATCCOQHAAHA} to the realm of left resp.\ right Hopf algebroids, which, as already hinted at, allow for the simultaneous generalisation of various (co)homology theories such as Hopf algebras, associative algebras, Lie algebroids as well as {\em full} Hopf algebroids, that is, those with an antipode in the sense of \cite{BoeSzl:HAWBAAIAD}.

More precisely, whereas it is, as just mentioned,  well-known that the category of Yetter-Drinfel'd modules over a bialgebroid $U$
is equivalent to the (weak) monoidal centre of the category of left $U$-modules \cite[Prop.~4.4]{Schau:DADOQGHA}  
as is the case for bialgebras, we are going to show in the following that anti Yetter-Drinfel'd modules and anti Yetter-Drinfel'd contramodules correspond to the {\em bimodule category centre} of the category of left $U$-comodules resp.\ left $U$-modules over their respective opposite categories.
The main difficulty in dealing here with left resp.\ right Hopf algebroids 
is, apart from the noncommutativity of the base ring, the absence of an antipode map which leads to nontrivial associativity constraints in the bimodule categories in question and hence to considerably more laborious computations, in striking
contrast to the case of Hopf algebras (or even full Hopf algebroids for that matter).

On the other hand, the sort of disheartening abundance of possibilities for defining, for example, anti Yetter-Drinfel'd modules in the Hopf algebra case (left-left, left-right, and so on) in the left Hopf algebroid case is instantly limited to one (all other variants not being well-defined) and no further equivalences need to be established (nor discussed).

\subsection{Main results}
Corresponding to the general idea just outlined, 
assembling Lemma \ref{keineDokumente} \&  Corollary \ref{wasasesam1} with Theorem \ref{tegel1}, based on a general categorical approach
exhibited in \S\ref{cat}, 
in \S\ref{heizungs} we essentially show
(see the main text for all details, notation, and the precise statements):

\begin{theorem}
  \label{eins}
  Let a left bialgebroid $(U,A)$ in addition be left Hopf. 
  Then the category $\umod$ of left $U$-modules is biclosed, which by adjunction induces on it the structure of a bimodule category over its opposite category.
  The category of anti Yetter-Drinfel'd contramodules over $U$ is equivalent to the centre of this bimodule category.
\end{theorem}

In particular, any {\em stable} anti Yetter-Drinfel'd contramodule over $U$ can be seen as an object in a certain full subcategory of this centre.  By virtue of this result, in Theorems \ref{ganzleerheute} \& \ref{lalacrimosa}, we can not only define a so-called ({\em weak}) {\em trace functor} in the sense of Kaledin \cite{Kal:TTAL} on the category of left $U$-modules, but also explicitly construct a cyclic operator in the sense of Connes \cite{Con:CCEFE}, that is:

\begin{theorem}
  \label{zwei}
If a left bialgebroid $(U,A)$ is  
left Hopf and $M$ a stable anti Yetter-Drinfel'd contramodule over $U$ with contraaction $\gamma$, then $\Hom_\uhhu(-, M)$ yields a trace functor $\umod \to \kmod$, which, in particular, 
implies an isomorphism 
$$
\Hom_\uhhu(X \otimes_\ahha Y, M) \simeq \Hom_\uhhu(Y \otimes_\ahha X, M)
$$
for any $X, Y \in \umod$.
Its explicit form induces the cocyclic operator
$$
(\uptau f)(u^1, \ldots, u^q) = \gamma\big(((u^1_{(2)} \cdots u^{q-1}_{(2)} u^q) \pmact  
  f)(-, u^1_{(1)}, \ldots, u^{q-1}_{(1)}) \big)
  $$
   on the cochain complex $ C^\bullet(U,M) = \Hom_\Aopp(U^{\otimes_\Aopp\bullet}, M) $,  
which (under suitable projectivity assumptions) computes $\Ext^\bull_\uhhu(A,M)$.  
\end{theorem}
For details of the precise construction and all notation we refer to the main text: let us only mention here that dropping {\em stability} leads (quite in general) to a {\em weak} or {\em nonunital} trace that weakens the resulting cocyclic operator to an only para-cocyclic one.
With one more (mild) technical assumptions mentioned in Remark \ref{quelquechose}, one can even replace $A$ by a Yetter-Drinfel'd module $N$ in the $\Ext$-groups above.

Dually, passing in \S\ref{wartung} to the monoidal category $\ucomod$ of left $U$-comodules in relationship to anti Yetter-Drinfel'd modules, 
assembling the statements of Lemma \ref{keineDokumentecomod} \&  Corollary \ref{wasasesam2} with Theorem \ref{tegel2},
we can summarise:

\begin{theorem}
  \label{drei}
  Let a left bialgebroid $(U,A)$ be simultaneously left and right Hopf. 
  Then, under suitable projectivity assumptions, the category $\ucomod$ is biclosed, which by adjunction induces on it the structure of a bimodule category over its opposite category. The category of anti Yetter-Drinfel'd modules over $U$ is equivalent to the centre of this bimodule category.
  \end{theorem}

Again,
asking for stability of the anti Yetter-Drinfel'd modules establishes a categorical equivalence to a certain full subcategory of this centre.
Likewise, 
if $M$ is now an anti Yetter-Drinfel'd module, this allows for the construction of a trace functor $\Hom^\uhhu(-, M) \colon  \ucomod \to \kmod$ obeying an analogous commutation property as above, that is
$$
\Hom^\uhhu(X \otimes_\ahha Y, M) \simeq \Hom^\uhhu(Y \otimes_\ahha X, M)
$$
for any $X, Y \in \ucomod$.

Observe the somewhat unexpected asymmetry between the module and comodule case in Theorems \ref{eins} and \ref{drei} with respect to the number of Hopf structures needed; see Remarks \ref{relajante1} and \ref{relajante2} for a possible explanation.

\subsection{Notation and conventions}
\label{schonweniger}
A very brief exposition on bialgebroids and (left and right) Hopf algebroids as well as the respective relevant notation is given in Appendix \ref{bialgebroids} at the end of the main text. At this point, we only want to recall that a left bialgebroid $(U, A)$ is called {\em left} resp.\ {\em right} Hopf algebroid if the corresponding Hopf-Galois map $\ga_\ell$ resp.\ $\ga_r$ is invertible, where 
\begin{equation*}
\begin{array}{rcccccc}
\ga_\ell  \colon  \!\!\!&\due U \blact {} \otimes_{\Aopp} U_\ract &\to& U_\ract  \otimes_\ahha  \due U \lact,
& u \otimes_\Aopp v  &\mapsto&  u_{(1)} \otimes_\ahha u_{(2)}  v, \\
\ga_r  \colon  \!\!\!& U_{\!\bract}  \otimes_\ahha \! \due U \lact {}  &\to& U_{\!\ract}  \otimes_\ahha  \due U \lact,
&  u \otimes_\ahha v  &\mapsto&  u_{(1)}  v \otimes_\ahha u_{(2)}.
\end{array}
\end{equation*}
The Sweedler-type shorthand notations
\begin{equation*}
  \begin{array}{rcl}
u_+ \otimes_\Aopp u_-  & \coloneqq &  \alpha_\ell^{-1}(u \otimes_\ahha 1),
\\
   u_{[+]} \otimes_\ahha u_{[-]}  & \coloneqq &  \alpha_r^{-1}(1 \otimes_\ahha u),
\end{array}
  \end{equation*}
with summation understood, will be used throughout the entire text.
Recall moreover from Eq.~\eqref{pergolesi} the various triangle notations $\lact, \ract, \blact, \bract$ that denote the four $A$-module structures on the total space $U$ of a bialgebroid, and occasionally even on a $U$-module: sometimes we decorate $U$ or a $U$-module by one of these symbols to indicate the relevant $A$-module structure in a specific situation, {\em e.g.}, in a tensor product.
The symbol $k$ always denotes a commutative ring, usually of characteristic zero.

\section{Categorical preliminaries}
\label{cat}

\addtocontents{toc}{\protect\setcounter{tocdepth}{2}}

In this preliminary section, we gather some notions from category theory such as module categories and centres of bimodule categories that generalise the corresponding ideas from algebra and are at the base of our subsequent considerations. 

\subsection{Bimodule categories and centres}

Let $(\cC, \otimes, \mathbb{1}, \ga, l, r)$ be a monoidal category, where
$
\ga\colon (X \otimes Y) \otimes Z \stackrel{\simeq}{\longrightarrow} X \otimes (Y \otimes Z)
$
is the associativity constraint, and $l$ resp.\ $r$ the left resp.\ right unit constraint.
The following couple of definitions can be found in \cite[\S7.1]{EtiGelNikOst:TC}.

\begin{dfn}
  \label{modcat}
  A {\em left module category} over $\cC$ is a category $\cM$ equipped with a bifunctor $\fren \colon  \cC \times \cM \to \cM$ and natural isomorphisms, again called {\em associativity} and {\em unit constraint},
  \begin{equation}
    \label{assoc1}
    \phi_{\scriptscriptstyle X,Y,M} \colon  (X \otimes Y) \fren M \stackrel{\simeq}{\longrightarrow} X \fren (Y \fren M),
 \qqquad
1_{\scriptscriptstyle M} \colon  \mathbb{1} \fren M \stackrel{\simeq}{\longrightarrow} M
\end{equation}
for all $X, Y \in \cC$ and $M \in \cM$, such that the customary pentagon and triangle diagrams
\begin{small}
\begin{equation}
\label{tarrega3}
\xymatrix{
   & ((X \otimes Y) \otimes Z) \fren M  \ar[ld]_{\ga_{X,Y,Z} \fren \, \id_M \quad \ \ } \ar[rd]^{\ \ \phi_{\scriptscriptstyle X \otimes Y, Z, M}}
  & 
  \\
  (X \otimes (Y \otimes Z)) \fren M \ar[d]_{\phi_{\scriptscriptstyle X, Y \otimes Z, M}}
   & &   (X \otimes Y) \fren (Z \fren M) \ar[d]^{\phi_{\scriptscriptstyle X, Y, Z {\scalebox{0.7}{\fren}} M}}
  \\
\qqquad    X \fren ((Y \otimes Z) \fren M) \ar[rr]^{\id_X \fren \, \phi_{\scriptscriptstyle Y, Z, M}} & &   X \fren (Y \fren (Z \fren M)) \qqquad  
  }
\end{equation}
\end{small}
and
\begin{small}
\begin{equation}
\label{tarrega4}
\xymatrix{
  (X \otimes \mathbb{1}) \fren M \ar[rr]^{\phi_{\scriptscriptstyle X, \mathbb{1}, M}} \ar[rd]_{r_X \fren \, \id_M \quad}
  & & X \fren (\mathbb{1} \fren M) \ar[ld]^{\quad   \id_X \fren \, l_M}
  \\
&  X \fren M &
  }
\end{equation}
\end{small}
commute.
  \end{dfn}

This clearly generalises the idea of a module over a ring.
A {\em right module category} over $\cC$ is defined analogously and is the same as a left $\cC^\op$-module category. In this case, we use the notation
\begin{equation}
\label{assoc2}
\frenop \colon  \cM \times \cC \to \cM, 
\qqquad
\psi_{\scriptscriptstyle M,X,Y} \colon  M \frenop (X \otimes Y) \stackrel{\simeq}{\longrightarrow} (M \frenop X) \frenop Y
\end{equation}
for the bifunctor and the associativity constraint.

\begin{dfn}
  \label{ratagnan}
  A {\em bimodule category} over two monoidal categories $\cC$ and $\cD$ is a category $\cM$ that is simultaneously a left $\cC$-module and right $\cD$-module category with respective associative constraints $\phi$ and $\psi$, plus {\em middle associativity constraints} given by natural transformations
  \begin{equation}
    \label{beethovensneunte}
    \gvt_{\scriptscriptstyle X, M, Z} \colon 
(X \fren M) \frenop Z   
    \stackrel{\simeq}{\lra}
 X \fren (M \frenop Z)    
  \end{equation}
for $M \in \cM$, $X \in \cC$, and $Z \in \cD$, such that the two pentagon diagrams
\begin{small}
\begin{equation}
\label{tarrega5}
\xymatrix@C=.77cm{
   & ((X \otimes Y) \fren M) \frenop Z  \ar[ld]_{\phi_{\scriptscriptstyle X,Y,M} \frenop \, \id_Z \quad \ \ } \ar[rd]^{\ \ \gvt_{\scriptscriptstyle X \otimes Y, M, Z}}
  & 
  \\
  (X \fren (Y \fren M)) \frenop Z \ar[d]_{\gvt_{\scriptscriptstyle X, Y {\scalebox{0.7}{\fren}} M, Z}}
   & &   (X \otimes Y) \fren (M \frenop Z) \ar[d]^{\phi_{\scriptscriptstyle X, Y, M {\raisebox{-0.2pt}{\scalebox{0.7}{\frenop}}} Z}}
  \\
\qqquad    X \fren ((Y \fren M) \frenop Z) \ar[rr]^{\id_X \fren \, \gvt_{\scriptscriptstyle Y, M, Z}} & &   X \fren (Y \fren (M \frenop Z)) \qqquad  
  }
\end{equation}
\end{small}
and
\begin{small}
\begin{equation}
\label{tarrega6}
\xymatrix@C=.77cm{
   & X \fren (M \frenop (W \otimes Z))   \ar[ld]_{\id_X \fren \, \psi_{M,W,Z}  \quad \ \ \ } \ar@{<-}[rd]^{\ \ \ \gvt_{\scriptscriptstyle X, M, W \otimes Z}}
  & 
  \\
  X \fren ((M \frenop W) \frenop Z) \ar@{<-}[d]_{\gvt_{\scriptscriptstyle X, M {\scalebox{0.7}{\frenop}} W, Z}}
   & &   (X \fren M) \frenop (W \otimes Z) \ar[d]^{\psi_{X {\scalebox{0.7}{\fren}} M, W, Z}}
  \\
\qqquad    (X \fren (M \frenop W)) \frenop Z \ar@{<-}[rr]^{\gvt_{\scriptscriptstyle X, M, W} \frenop \, \id_Z} & &   ((X \fren M) \frenop W) \frenop Z \qqquad  
  }
\end{equation}
\end{small}
commute for all
 $M \in \cM$, $X, Y \in \cC$, and $Z, W \in \cD$.
\end{dfn}

\begin{rem}
  Note that whereas several relevant examples of monoidal categories are strict, {\em i.e.}, where the associative constraint $\ga$, along with the left and right unit constraint $l$ resp.\ $r$ are the identity transformations such that the diagrams \eqref{tarrega3} and \eqref{tarrega4} somewhat simplify, this cannot be said for typical examples of (bi)module categories. Here, even for underlying strict monoidal categories, the left, right, and middle associative constraints $\phi$, $\psi$, and $\gvt$ from \eqref{assoc1}, \eqref{assoc2}, and \eqref{beethovensneunte}
are not necessarily an easy guess, see Eqs.~\eqref{bademser} and \eqref{leitz2} for concrete nontrivial examples. This is mainly due to our dealing with (left or right) Hopf algebroids instead of Hopf algebras and therefore the absence of (the notion of) an antipode resp.\ its inverse.
  \end{rem}

The definition of the centre of a bimodule category was formulated
in the context of fusion categories in \cite[Def.~2.1]{GelNaiNik:FCAHT}; we relax it here to monoidal categories which is most likely already present in the literature somewhere.

\begin{dfn}
  \label{wisconsin}
  The {\em centre  of a $(\cC, \cC)$-bimodule category} $\cM$ is a category $\cZ_\cC(\cM)$ the objects of which are given by pairs $(M, \tau)$, where $M$ is an object in $\cM$ and
   $$
  \tau_X \colon 
  X \fren M
  \stackrel{\simeq}{\lra}
  M  \frenop X
  $$
are isomorphisms natural in $X$ such that the hexagon diagram
\begin{small}
\begin{equation}
\label{tarrega7}
\xymatrix{
  & X \fren (M \frenop Z) \ar@{<-}[r]_{\gvt_{\scriptscriptstyle X,M,Z}}
\ar@{<-}[ld]_{\id_X \fren \, \tau_Z \ \ \ }
&  (X \fren M) \frenop Z
\ar[rd]^{\ \ \tau_X \frenop \, \id_Z}
&
  \\
X \fren (Z \fren M)  \ar@{<-}[rd]_{\phi_{\scriptscriptstyle X, Z, M} \quad }
   & &  & (M \frenop X) \frenop Z \ar@{<-}[ld]^{\quad \psi_{M, X, Z}}
  \\
  &    (X \otimes Z) \fren M \ar[r]_{\tau_{X \otimes Z}} &    M \frenop (X \otimes Z) &
  }
\end{equation}
\end{small}
commutes for all $M \in \cM$ and $X, Z \in \cC$.
\end{dfn}

The natural transformation $\tau$ is called a {\em central structure} on $M$.
This definition clearly lifts the idea of the center of a bimodule over a ring to a categorical realm.

\subsubsection{Biclosed categories as bimodule categories}
\label{dominosteine}
  Clearly,~a~monoidal category is a bimodule category over itself by means of the mon\-oidal product, but this is often not the only possibility and indeed not what we are going to consider in the next sections. If $\cC$ is biclosed, by means of the left and right internal Homs we can define additional right and left $\cC$-actions on $\cC$ itself: 
    \begin{equation}
    \label{verdura}
    Y \fren Z  \coloneqq  \hom^r(Y, Z),
    \qquad
Z \frenop Y  \coloneqq  \hom^\ell(Y,Z)
    \end{equation}
for objects $Y, Z \in \cC$. This way, the adjunctions read as  
adjunctions 
  \begin{equation}
    \label{fueller}
  \begin{array}{rcccl}
 \xi\colon   \Hom_\cC(X \otimes Y, Z)
   \!\!\! &\xrightarrow{\ \simeq \ } \!\!\! &
  \Hom_\cC(X, \hom^r(Y,Z))
   \!\!\! &= \!\!\!&
    \Hom_\cC(X, Y \fren Z),
\\
    \zeta\colon    \Hom_\cC(X \otimes Y, Z)
\!\!\!
&\xrightarrow{\ \simeq \ } \!\!\! &
  \Hom_\cC(Y, \hom^\ell(X,Z))
    \!\!\! &= \!\!\!&
    \Hom_\cC(Y, Z \frenop X),
  \end{array}
  \end{equation}
  flipping a right action into a left resp.\ a left into a right one.
%
If we, to respect (covariant) functoriality, as usual see the left and right internal Homs as maps $\cC^\op \times \cC \to \cC$, where $\cC^\op$ denotes the category opposite to $\cC$, we can sum up the above example in the following (probably) well-known lemma:
\begin{lem}
\label{glenngould}
A biclosed monoidal category $\cC$ is a bimodule category over $\cC^\op$ by means of the (adjoint) right and left actions in \eqref{verdura}. 
\end{lem}

\begin{proof}
  By means of the adjunctions \eqref{fueller} and the associativity constraint of the monoidal category $\cC$, we have for any object $W \in \cC$:
  \begin{equation*}
    \begin{array}{lllll}
  &&           \Hom_\cC\big(W, \hom^\ell(Z, \hom^r(Y, M))\big)
\\
      &\simeq&
               \Hom_\cC\big(Z \otimes W, \hom^r(Y, M)\big)
\\
               & \simeq &
     \Hom_\cC\big((Z \otimes W) \otimes Y, M\big)
\\
  &   \simeq &
     \Hom_\cC\big(Z \otimes (W \otimes Y), M\big)
\\
     & \simeq &
     \Hom_\cC\big(W \otimes Y, \hom^\ell(Z,M)\big)
     \\
    & \simeq &
     \Hom_\cC\big(W, \hom^r(Y, \hom^\ell(Z,M))\big)
       \end{array}
  \end{equation*}
  for any objects $M, Y, Z \in \cC$, which by the Yoneda Lemma implies
$$
\hom^\ell(Z, \hom^r(Y, M))
\simeq
\hom^r(Y, \hom^\ell(Z,M)).
  $$
This yields the middle associativity \eqref{beethovensneunte} with respect to the actions \eqref{verdura}.
Likewise, one obtains the
left resp.\ right associativity constraints
$
\hom^r(X \otimes Y, M)
\simeq
\hom^r(X, \hom^r(Y,M))
$
resp.\
$
\hom^\ell(X \otimes Y, M)
\simeq
\hom^\ell(Y, \hom^\ell(X,M)),
$
along with the left and right unit constraints $\hom^r(\mathbb{1}, M) \simeq M \simeq \hom^\ell(\mathbb{1}, M)$.
    It is then straightforward to 
    verify Diagrams \eqref{tarrega5} \& \eqref{tarrega6} to complete the proof.
  \end{proof}
  
The centre of $\cC$ with respect to this $\cC^\op$-bimodule structure
will be denoted by $\cZ_{\cC^\op}(\cC)$.
Following
  \cite[Eq.~(2.11)]{Sha:OTAYDMCC},
  and similar to \cite[Def.~2.3 \& Lem.~2.4]{KobSha:ACATCCOQHAAHA}, we denote by
  $
\cZ'_{\cC^\op}(\cC)
  $
its full subcategory consisting of objects $M$ such that
the identity morphism $\id_\emme \in \Hom_\cC(M, M)$ is mapped to itself
via the chain of isomorphisms
\begin{small}
\begin{equation}
  \label{subcategory}
\begin{array}{lcccr}
  \Hom_\cC(M, M)
\!\!\!\!\!&\simeq\!\!\!\!\!&
  \\
\Hom_\cC(\mathbb{1} \otimes M, M)
\!\!\!\!\!&\simeq\!\!\!\!\!&
\Hom_\cC(\mathbb{1}, M \fren M ) \simeq
\Hom_\cC(\mathbb{1}, M \frenop M )
\!\!\!\!\!&\simeq\!\!\!\!\!&
\Hom_\cC(M \otimes \mathbb{1}, M)
\\
&&
  \!\!\!\!\!&\simeq\!\!\!\!\!&
\Hom_\cC(M, M),
\end{array}
\end{equation}
\end{small}
\!\!\!\!
induced by the left and right unit constraints, the adjunctions \eqref{fueller}, and the central structure on $M$.

\begin{dfn}
\label{stablepable}
  The full subcategory  $\cZ'_{\cC^\op}(\cC)$ of the bimodule category centre $\cZ_{\cC^\op}(\cC)$ will be called the {\em stable centre} of $\cC$ and its objects {\em stable}.
\end{dfn}

As we will see later on,
from a different perspective 
this subcategory will distinguish cyclic objects from para-cyclic ones.

\subsection{Trace functors}
\label{t'n'c}
We will need one more piece of categorical~machinery:
so-called
trace functors,
introduced by Kaledin \cite[Def.~2.1]{Kal:TTAL}
in an approach to cyclic hom\-ology with coefficients and towards a possible understanding of cyclic homology as a derived functor \cite{Kal:CHWC}. We slightly weaken the original notion here:

\begin{dfn}
  \label{kaledin}
A {\em weak trace functor} consists of a functor $T \colon  \cC \to \cE$ between a monoidal category $(\cC, \otimes, \mathbb{1}, \ga, l, r)$ and a category $\cE$, together with a family of isomorphisms
$$
\tr_{\scriptscriptstyle X,Y} \colon  T(X \otimes Y)
  \stackrel{\simeq}{\lra}
T(Y \otimes X)
$$ 
functorial in all $X, Y \in \cC$ and such that
\begin{equation}
  \label{trace0}
\tr_{\scriptscriptstyle Z, X \otimes Y}
=
   T(\ga^{-1}_{\scriptscriptstyle X, Y, Z})
\circ
\tr_{\scriptscriptstyle Y \otimes Z, X}
\circ
   T(\ga^{-1}_{\scriptscriptstyle Y, Z, X})
   \circ
   \tr_{\scriptscriptstyle Z \otimes X, Y}
   \circ   T(\ga^{-1}_{\scriptscriptstyle Z, X, Y})
\end{equation}
for all $X, Y, Z \in \cC$.
A weak trace functor is called {\em unital} if
\begin{equation}
  \label{unital}
  T(r_{\scriptscriptstyle X})
  \circ \tr_{\scriptscriptstyle \mathbb{1}, X}
  = T(l_{\scriptscriptstyle X})
  \end{equation}
for all $X \in \cC$, in case of which we simply speak of a {\em trace functor}.
\end{dfn}

\begin{rem}
  \label{kaledin1}
  If the trace functor is unital, 
using \eqref{unital} 
in \eqref{trace0} in case $Z = \mathbb{1}$ and the naturality of $\tr$ for the left and right unit constraint, \eqref{trace0} reduces to
  $$
\tr_{\scriptscriptstyle Y, X} \circ \tr_{\scriptscriptstyle X, Y} = \id. 
  $$
In this case, abbreviating
$
\tr_{\scriptscriptstyle X, Y, Z}  \coloneqq 
\tr_{\scriptscriptstyle X,  Y \otimes Z}
   \circ   T(\ga_{\scriptscriptstyle X, Y, Z}),
   $
   Eq.~\eqref{trace0} can be more compactly rewritten as
\begin{equation}
  \label{trace}
\tr_{\scriptscriptstyle Z, X, Y}
\circ
\tr_{\scriptscriptstyle Y, Z, X}
  \circ \tr_{\scriptscriptstyle X, Y, Z} = \id
\end{equation}
for all $X, Y, Z \in \cC$, which is the form in which it appears in
\cite[Def.~2.1]{Kal:TTAL}.
\end{rem}

The distinction between the unital and nonunital (or weak)  case finds its motivation in the next subsection.


\subsubsection{Trace functors from biclosed categories}
\label{federtasche}
  In the setting we are going to deal with,
  typical examples of trace functors of interest for us turn out to be closely related to bimodule category centres and internal Homs, that is,
  in continuation of \S\ref{dominosteine}, 
  they arise via the adjunctions \eqref{fueller} in connection with $\cZ_{\cC^\op}(\cC)$ (and $\cZ'_{\cC^\op}(\cC)$, respectively):

  \begin{lem}
    \label{funkyfunk}
    If $\cC$ is a biclosed monoidal category,
then
$T = \Hom_\cC(-, M)$ for any $M \in    \cZ_{\cC^\op}(\cC)$
defines a weak trace functor, which is unital if $M \in \cZ'_{\cC^\op}(\cC)$. Vice versa,  
    a weak trace functor of the form  $T = \Hom_\cC(-, M)$ induces a central structure $\tau$ on an object $M \in \cC$, and if $T$ is unital, a central structure with respect to which $M$ is stable.
    \end{lem}

  \begin{proof}
    For a biclosed monoidal category, with actions defined as in \eqref{verdura}, set
    \begin{equation}
      \label{cupper}
\tr_{\scriptscriptstyle{X,Y}} = \zeta^{-1} \circ \Hom_\cC(X, \tau_{\scriptscriptstyle Y}) \circ \xi,
      \end{equation}
    where $\zeta, \xi$ denote the adjunctions \eqref{fueller}, that is to say, we define $\tr$ by the following diagram:

    \vspace*{-.3cm}
    \begin{center}
 \begin{footnotesize}
    \begin{tikzcd}[column sep={18em,between origins}]
 \Hom_\cC(X \otimes Y, Z)
 \arrow[r, "\tr_{X,Y}"]
 \arrow[d, swap, outer sep=2pt, "\xi"]
 &
 \Hom_\cC(Y \otimes X, Z) 
  \\
  \Hom_\cC(X, \hom^r(Y,Z))
  \arrow[r, outer sep=1pt, "{\Hom_\cC(X, \tau_Y)}"]
&
  \Hom_\cC(Y, \hom^\ell(X,Z))
\arrow[u, swap, outer sep=2pt, "\zeta^{-1}"]
    \end{tikzcd}
\end{footnotesize}
    \end{center}
    \vspace*{-.1cm}
    for all $X, Y, Z \in \cC$.
 It is then a direct check that the following honeycomb of isomorphisms commutes:

 \smallskip
 \hspace*{-1cm}{
\begin{footnotesize}
 \begin{tikzcd}[column sep={7.5em,between origins},row sep=3.0em]
  & &
  \Hom_\cC(Y \otimes (Z \otimes X), M) 
  \arrow[rr, leftarrow]
  \ar[ld, rightarrow]
  \ar[ld, leftarrow, 
    end anchor={[xshift=-1.9ex]north east}]
        & & \Hom_\cC((Y \otimes Z) \otimes X, M) 
  \ar[ld, rightarrow]
  \ar[rdd, rightarrow, "\tr_{Y \otimes Z, X}"]
    &
    \\
  & 
    \Hom_\cC(Z, X \fren (M \frenop Y))
      \ar[d, start anchor={[xshift=-4.95ex]south east}, end anchor={[xshift=-0.87ex]}, dash, "\tr_{Z \otimes X, Y}"]
& &
\Hom_\cC(Z, (X \fren M) \frenop Y)
\ar[from=ll, crossing over, leftarrow]
 \ar[rdd, rightarrow, swap, "\tau_X \frenop \id_Y"]
& &
  \\
  &
  \Hom_\cC((Z \otimes X) \otimes Y, M) 
  \ar[ld, rightarrow]
  & & & &
  \Hom_\cC(X \otimes (Y \otimes Z), M)
  \ar[ld, rightarrow]
        \ar[ld, leftarrow, start anchor={[xshift=0ex]}, end anchor={[xshift=-1.9ex]north east}]
  \\
\Hom_\cC(Z,  X \fren (Y \fren M))
\ar[rdd, leftarrow]
  \ar[ruu, crossing over, rightarrow, near start, "\id_X \fren \tau_Y"]
& &  & & 
\Hom_\cC(Z,  (M \frenop X) \frenop Y)
      \ar[d, dash, start anchor={[xshift=-4.97ex]south east}, end anchor={[xshift=-0.87ex]}]
&
  \\
  & &
  \Hom_\cC(Z \otimes (X \otimes Y), M) 
\ar[luu, leftarrow]
\ar[ld, rightarrow]
 \ar[rr, rightarrow, "\tr_{Z, X \otimes  Y}"]
 & &
  \Hom_\cC((X \otimes Y) \otimes Z, M) 
 \ar[ld, rightarrow]
    & 
 \\
 &
 \Hom_\cC(Z, (X \otimes Y) \fren M)
 \ar[rr, rightarrow, "\tau_{X \otimes Y}"]
  & &
\Hom_\cC(Z,  M \frenop (X \otimes Y))
 \ar[from=ruu, crossing over, leftarrow]
  & &
\end{tikzcd}
\end{footnotesize}  
 }
 
 \noindent where for better readability we did not label the arrows coming from all kinds of associators, by abuse of notation wrote $\tau_X$ etc.\ for  postcompositions of the type $\Hom_\cC(Z, \tau_X)$, and where the arrows pointing out of the depths are those coming from the adjunctions \eqref{fueller}, applied once or twice. In particular, this proves that if (and only if) the front hexagon in the above honeycomb commutes, {\em i.e.}, if $M \in    \cZ_{\cC^\op}(\cC)$ such that \eqref{tarrega7} commutes, then the back hexagon commutes as well, which amounts to the property \eqref{trace0}, and hence $T = \Hom_\cC(-, M)$ is a weak trace functor.

 Now consider the natural transformation $\eta\colon T \to T$ given by
$
\eta  \coloneqq  T(r) \circ \tr_{\scriptscriptstyle \mathbb{1}, -} \circ \, T(l^{-1}).
$
If $M$ is stable, by \eqref{subcategory} one has
$
\eta_\emme = \id.
$
Since the Yoneda Lemma implies that the natural transformation $\eta$ is uniquely determined by its values on $M$, we obtain
$
\eta_\ikks = \id
$
for any object $X \in \cC$ as well, and therefore
$
T(r_\ikks) \circ \tr_{\scriptscriptstyle X, \mathbb{1}} = T(l_\ikks),
$
that is, the weak trace functor is unital and what was said in Remark \ref{kaledin1} applies.

Vice versa, if $T = \Hom_\cC(-, M)$ is a weak trace functor with a family of isomorphisms $\tr$ such that \eqref{trace0} holds, then reading \eqref{cupper} from right to left, the honeycomb implies the hexagon property \eqref{tarrega7} for the central structure by a Yoneda Lemma argument again. If $T$ is unital, then \eqref{subcategory} is automatic and $M$ lies in $\cZ'_{\cC^\op}(\cC)$, which is to say, $M$ is stable. 
    \end{proof}


 The study of trace functors of the form $T = \Hom_\cC(-, M)$, where
  $
M \in \cZ_{\cC^\op}(\cC),
  $
will be the content of the next sections for concrete biclosed monoidal categories.
In \S\ref{shine}, all this will lead to how to explicitly re-obtain 
the cyclic operator on the cochain complex computing certain $\Ext$ groups from such a trace functor.


\section{Centres and anti Yetter-Drinfel'd contramodules}
\label{heizungs}

According to Lemmata \ref{glenngould} \& \ref{funkyfunk},
the main idea in what follows is to define (or find) the internal Homs of a biclosed monoidal category of our interest, which then allows for a left and a right adjoint action, a corresponding bimodule category and finally its centre inducing a trace functor, as dealt with generally in the previous section.

\subsection{Left and right closedness of $\umod$}
\label{threeone}
Let $(U, A)$ be a left bialgebroid (see \S\ref{bialgebroids1}).
As in the bialgebra case, 
the monoidal structure on the (strict) monoidal category $\umod$ of left $U$-modules corresponds to the diagonal $U$-action on the tensor product $N \otimes_\ahha M$ of two left $U$-modules $N, M$:
\begin{equation}
  \label{mancino}
u \mancino (n \otimes_\ahha m)  \coloneqq 
\gD(u)(n \otimes_\ahha m) =
u_{(1)} n \otimes_\ahha u_{(2)} m
\end{equation}
for $n \in N$, $m \in M$, and $u \in U$.

With respect to the obvious forgetful functor $\umod \to \amoda$, we sometimes denote the induced $A$-bimodule structure on a left $U$-module $M$ by
\begin{equation}
  \label{peinture}
  a \lact m \ract b  \coloneqq  s(a)t(b)m, \qquad \forall \ m \in M, \ a, b \in A.
  \end{equation}

\begin{lem}
\label{keineDokumente}
Let $(U, A)$ be a left bialgebroid.
\begin{enumerate}
  \compactlist{99}
\item[({\it i})]
  The category $\umod$ of left $U$-modules is left closed monoidal,
  that is, has left internal Hom functors:
  $$
\hom^\ell(N, M)  \coloneqq  \Hom_\uhhu(N \otimes_\ahha \due U \lact {}, M), 
$$
for all $N, M \in \umod$,
equipped with the left $U$-action
\begin{equation}
  \label{umact}
(v \umact f)(n \otimes_\ahha u)  \coloneqq  f(n \otimes_\ahha uv)
  \end{equation}
for every $u, v \in U$ and $n \in N$.
\item[({\it ii})]
  If the left bialgebroid is left Hopf (see \S\ref{bialgebroids2}), the category $\umod$ is right closed monoidal with right internal Hom functors of the form:
  \begin{equation}
    \label{ravel0}
\hom^r(N, M)  \coloneqq  \Hom_\Aopp(N,M)
  \end{equation}
for all $N, M \in \umod$,
equipped with the left $U$-action
  \begin{equation}
  \label{pmact}
  (u \pmact g)(n)  \coloneqq  u_+ g(u_- n)
  \end{equation}
  for every $u \in U$ and $n \in N$.
\item[({\it iii})]
Consequently, for a left Hopf algebroid $(U,A)$ over an underlying left bialgebroid, the category $\umod$ is biclosed monoidal, that is, has both left and right internal Hom functors.
\end{enumerate}
\end{lem}

\begin{proof}
  This is a well-known result and has been originally proven in
\cite[\S3]{Schau:DADOQGHA}, see also \cite[Lem.~4.16]{Kow:WEIABVA} for the conventions used in the setting at hand. For later use, we give the adjunction morphisms. As for part (i), this would be
\begin{small}
\begin{equation}
\label{ad1}
\begin{array}{rcl}
\zeta \colon  \Hom_\uhhu(N \otimes_\ahha P, M) &\to&
\Hom_\uhhu(P, \hom^\ell(N, M)),
\\[2pt]
f &\mapsto& \big\{p \mapsto \{n \otimes_\ahha u \mapsto f(n \otimes_\ahha up)\}\big\},
\\[2pt]
\big\{\tilde f(p)(n \otimes_\ahha 1) \mapsfrom n \otimes_\ahha p \big\} &\mapsfrom& \tilde f,
\end{array}
\end{equation}
\end{small}
and in part (ii), the claimed adjunction 
\begin{equation}
\label{ad2}
\begin{array}{rcl}
\xi \colon  \Hom_\uhhu(P \otimes_\ahha N, M)
&\to& \Hom_\uhhu(P, \hom^r(N, M)),
\\[2pt]
g &\mapsto& \{p \mapsto g(p \otimes_\ahha -) \}
\end{array}
\end{equation}
is simply the Hom-tensor adjunction.
\end{proof}

\begin{notation}
As the left and right internal Homs we use are quite different in nature and it sometimes  turns out to be necessary to remember the explicit $U$- or $A$-linearity in question, we shall not always use the sort of concealing notation $\hom^r$ and $\hom^\ell$ but often write $\Hom_\Aopp$ and $\Hom_\uhhu( - \otimes_\ahha U, -)$ instead, even if the internal Homs with their $U$-module structure are meant.
  \end{notation}

\begin{rem}
  \label{hours}
  The preceding lemma precisely establishes the setting adapted to our needs; nevertheless, even without any left Hopf structure, symmetrically to the case of the left internal Homs, the category $\umod$ over a left bialgebroid has right internal Homs as well (see \cite[Prop.~3.3]{Schau:DADOQGHA}). Put
  \begin{equation}
    \label{ravel}
\hom^r(N,M)  \coloneqq  \Hom_\uhhu(U_\ract \otimes_\ahha N, M), 
 \end{equation}
being a left $U$-module by right multiplication on $U$ in the argument.
 The original definition of a left Hopf algebroid \cite[Thm.~3.5]{Schau:DADOQGHA} then states that a left bialgebroid $(U,A)$ is called left Hopf if the forgetful functor $\umod \to \amoda$ preserves internal Homs (which in {\em loc.~cit.}~is shown to be equivalent to the definition of left Hopf algebroids mentioned below Eq.~\eqref{nochmehrRegen}). In this case, its right internal Homs \eqref{ravel} are isomorphic (as $U$-modules) to the ones given in \eqref{ravel0}, with isomorphism given by
 $$
 \Hom_\Aopp(N,M) \to \Hom_\uhhu(U_\ract \otimes_\ahha N, M), \quad g \mapsto \sma\cdot  \pmact g,
 $$
 and inverse $f \mapsto f(1 \otimes_\ahha -)$.
%
On the contrary, the left internal Homs can be simplified (or complicated, depending on the point of view) in case more (or rather a different) structure is present. More precisely, in case the left bialgebroid $(U,A)$ in addition is right Hopf, one can set $\hom^\ell(N,M)  \coloneqq  \Hom_\ahha(N,M)$ with left $U$-module structure given by
\begin{equation}
    \label{mpaction}
(u \mpact g)(n)  \coloneqq  u_{[+]} g(u_{[-]} n), \qquad g \in \Hom_\ahha(N,M), \ n \in N,
\end{equation}
and the same comments apply as above.
In the Hopf algebra case, the condition of being right Hopf corresponds to the antipode being invertible, see Eq.~\eqref{sesam}. We are, however, more interested in the more general approach in Lemma \ref{keineDokumente} with only one Hopf structure present, {\em i.e.}, the left one.
  \end{rem}

\subsection{$\umod$ as a bimodule category}
\label{arpino}
The internal Homs exhibited in the previous section that turn $\umod$ into a biclosed category 
allow to define the structure of a bimodule category on it
in the sense discussed in \S\ref{dominosteine}. More precisely, 
Lemmata \ref{glenngould} \& \ref{keineDokumente} directly imply:

\begin{cor}
  \label{wasasesam1}
  Let $(U,A)$ be a left bialgebroid.
Then the operation
   \begin{equation*}
      \umod \times \umod^\op \to \umod,
      \quad
      (M,N) \mapsto
      M
\frenop
      N  \coloneqq 
      \hom^\ell(N, M)
   \end{equation*}
   defines on $\umod$ the structure of a right module category over $\umod^\op$.
If $(U, A)$ is in addition  left Hopf, the operation
  \begin{equation}
    \label{methamill7}
      \umod^\op \times \umod \to \umod,
      (N,M) \mapsto
      N
\fren
      M  \coloneqq 
      \hom^r(N, M)
    \end{equation}
  defines on $\umod$ the structure of a left module category over $\umod^\op$.
  Hence, if the left bialgebroid $(U,A)$ is in addition  left Hopf, then $\umod$ is a bimodule category over the monoidal category $\umod^\op$.
\end{cor}

\begin{proof}
  All statements directly follow from the general case in Lemma \ref{glenngould} along with Lemma \ref{keineDokumente}, so there is nothing to prove.
  Nevertheless, for later use and the sake of explicit illustration of the abstract theory, let us discuss the involved associative constraints.
  
As for the right action,
for three left $U$-modules $M, N, P \in \umod$ there is 
 a left $U$-module isomorphism $(M \frenop P) \frenop
N \simeq M \frenop (P \otimes_\ahha N)$, that is, 
  \begin{equation*}
    \psi_{\scriptscriptstyle M, P, N} \colon 
    \hom^\ell(P \otimes_\ahha N, M)
  \to \hom^\ell(N, \hom^\ell(P, M)),
  \end{equation*}
  which on the level of $k$-modules translates into a map
\begin{small}
  \begin{equation}
        \label{leitz12a}
        \begin{array}{rcl}
              \psi_{\scriptscriptstyle M, P, N} \colon 
      \Hom_\uhhu(P \otimes_\ahha N \otimes_\ahha U, M)
  \!\!&\to\!\!& \Hom_\uhhu(N \otimes_\ahha U, \Hom_\uhhu(P \otimes_\ahha U, M)),
      \\[2pt]
      f \mapsto \big\{n \otimes_\ahha u \!\!&\mapsto\!\!& \{ p \otimes_\ahha v \mapsto
f(p \otimes_\ahha v_{(1)} n \otimes_\ahha v_{(2)}u) 
\} \big\},
\\[2pt]
\{  g(n \otimes_\ahha u)(p \otimes_\ahha 1) \!\!&\mapsfrom\!\!& p \otimes_\ahha n \otimes_\ahha u \} \mapsfrom g.
    \end{array}
  \end{equation}
\end{small}
It is straightforward to see that both maps in \eqref{leitz12a} are well-defined and mutual inverses. That these are maps of left $U$-modules follows from \eqref{umact} by
\begin{small}
  \begin{equation*}
\begin{split}    
&(w \umact \psi_{\scriptscriptstyle M, P, N} f)(n \otimes_\ahha u)(p \otimes_\ahha v)
=
(\psi_{\scriptscriptstyle M, P, N} f)(n \otimes_\ahha uw)(p \otimes_\ahha v)
\\
&
=
f(p \otimes_\ahha v_{(1)} n \otimes_\ahha v_{(2)}uw)
=
(w \umact f)(p \otimes_\ahha v_{(1)} n \otimes_\ahha v_{(2)}uw)
\\
&
=
(\psi_{\scriptscriptstyle M, P, N} (w \umact f))(n \otimes_\ahha u)(p \otimes_\ahha v)
\end{split}
  \end{equation*}
  \end{small}
for $w, u, v \in U$, $p \in P$, and $n \in N$.

   As for the left action, for three left $U$-modules $M, N, P \in \umod$ there is a left $U$-module isomorphism
  $
  P \fren
(N \fren
M) \simeq (P \otimes_\ahha N) \fren
M
$
as well, 
or
  \begin{equation}
    \label{leitz11}
    \phi_{\scriptscriptstyle P, N, M} \colon 
    \hom^r(P \otimes_\ahha N, M)
  \to \hom^r(P, \hom^r(N, M)),
  \end{equation}
  which 
  results into a map
  $
   \Hom_\Aopp(P \otimes_\ahha N, M)
  \to \Hom_\Aopp(P, \Hom_\Aopp(N, M))
  $
  given by the Hom-tensor adjunction.
That this is an isomorphism (of $k$-modules) is obvious, whereas
using the left $U$-action \eqref{pmact} on $\Hom_\Aopp(N,M)$, along with Eq.~\eqref{Sch5} we immediately see that for
$f \in \Hom_\Aopp(P \otimes_\ahha N, M)$ one has, abbreviating $\phi =  \phi_{\scriptscriptstyle P, N, M}$, 
\begin{small}
  \begin{equation*}
    \begin{split}
    (u \pmact (\phi f))(p)(n)
    &=
    \big(u_+ \pmact (\phi f)(u_- p)\big)(n)
    =
    u_{++}(\phi f)(u_- p)(u_{+-} n)
\\ &
    =
    u_{+} f(u_{-(1)} p \otimes_\ahha u_{-(2)} n)
    =
    (u \pmact f)(p \otimes_\ahha n)
    = \phi(u \pmact f)(p)(n)
\end{split}
    \end{equation*}
  \end{small}
for any $u \in U$, 
hence
$
u \pmact (\phi f) = \phi(u \pmact f),
$
and therefore we obtain an isomorphism of left $U$-modules as well.


Finally, let us discuss the {\em middle associativity constraint} from Definition \ref{ratagnan}. This is the left $U$-module isomorphism
 $
 ( P \fren M) \frenop N 
 \overset{\simeq}{\lra}
P \fren (M \frenop N)
 $
for any $M, N, P \in \umod$, that is, 
$
\hom^\ell(N, \hom^r(P, M)) \simeq \hom^r(P, \hom^\ell(N, M)). 
$
Explicitly, this map is given by the $k$-module isomorphism
\begin{small}
  \begin{equation}
    \label{leitz33}
    \begin{array}{rcl}
   \gvt_{\scriptscriptstyle {\scriptscriptstyle{P, M, N}}} \colon  
    \Hom_\uhhu(N \otimes_\ahha \due U \lact {}, \Hom_\Aopp(P, M)) &\!\!\!\to&\!\!\!
    \Hom_\Aopp(P, \Hom_\uhhu (N \otimes_\ahha \due U \lact {}, M)),
    \\[2pt]
    f &\!\!\!\mapsto&\!\!\! \big\{p \mapsto \{n \otimes_\ahha u \mapsto f(n \otimes_\ahha u_{(1)})(u_{(2)}p) \} \big\},
    \\[2pt]
    \big\{ \{g(u_-p)(n \otimes_\ahha u_+) \mapsfrom p\} \mapsfrom n \otimes_\ahha u\big\} &\!\!\!\mapsfrom&\!\!\! g. 
    \end{array}
    \end{equation}
  \end{small}
Verifying that these maps are well-defined and mutual inverses is easy and
that $\gvt$ is in particular a map (and hence an isomorphism) of left $U$-modules is seen by
  \begin{equation*}
    \begin{split}
      (v \pmact \gvt f)(p)(n \otimes_\ahha u)
      &=
      \big(v_+ \umact (\gvt f)(v_- p)\big)(n \otimes_\ahha u)
      = (\gvt f)(v_- p)(n \otimes_\ahha uv_+)
      \\
      &
      = f(n \otimes_\ahha u_{(1)} v_{+(1)})(u_{(2)} v_{+(2)}v_- p)
       = f(n \otimes_\ahha u_{(1)} v)(u_{(2)} p)
 \\
      &
       = \big(\gvt (v \umact f)\big)(p)(n \otimes_\ahha u),
    \end{split}
\end{equation*}
abbreviating $\gvt =  \gvt_{\scriptscriptstyle P, M, N}$, where we used the left $U$-actions \eqref{pmact} and \eqref{umact} in the first step and Eq.~\eqref{Sch2} in the fourth.
  \end{proof}

\vspace*{.1cm}
\begin{center}
* \quad * \quad *
\end{center}
\vspace*{.15cm}

The preceding lemma allows to investigate the centre $\cZ_{\umod^\op}(\umod)$
in the sense of Definition \ref{wisconsin} of the bimodule category $\umod$ over $\umod^\op$; but before doing so, we need to introduce more algebraic structure to get meaningful statements, {\em i.e}, that of contramodules resp.\ anti Yetter-Drinfel'd contramodules as already hinted at in the Introduction.

\subsection{Contramodules over bialgebroids}
Contramodules in the sense of \cite{EilMoo:FORHA}
over coalgebras or corings are a not too wide-spread notion, which is somehow surprising as they turn out to be as natural as comodules (see, {\em e.g.}, \cite{BoeBrzWis:MACOMC, Brz:HCHWCC, Pos:C}): as a first approach, they can be thought of as an infinite dimensional version of modules over the dual of the coring in question. They are of interest since not only they are related to the centre of the bimodule category $\umod$ under investigation but (as a consequence) 
also because they appear as
natural coefficients in the cyclic theory of $\Ext$ groups. As such, they were implicitly used right from the beginning in Connes' classical cyclic cohomology theory with its values in the $k$-linear dual of an associative algebra,
as elucidated in \cite[\S6]{Kow:WEIABVA}.


\begin{dfn}
\label{schoenwaers}
A {\em right contramodule} over a left bialgebroid $(U,A)$ is a right $A$-module $M$ together with a right $A$-module map 
$$
\gamma \colon  \Hom_\Aopp(U_\ract,M) \to M, 
$$
usually termed the {\em contraaction}, subject to the diagram
\begin{small}
  \begin{equation*}
  \begin{split}
  &	\xymatrix{\Hom_\Aopp(U, \Hom_\Aopp(U,M))
	\ar[rrr]^-{\scriptstyle{\Hom_\Aopp(U,\gamma)}} \ar[d]_-{\simeq}&
	& &
	\Hom_\Aopp(U,M) \ar[d]^-{\gamma} \\
	\Hom_\Aopp(U_\ract \otimes_\ahha \due U \lact {},M)
	\ar[rr]_-{\scriptstyle{\Hom_\Aopp(\gD_\ell, M)}} && \Hom_\Aopp(U,M) \ar[r]_-{\gamma}
&	M }
\end{split}
\end{equation*}
\end{small}
\mbox{\normalsize{which we will refer to as {\em contraassociativity}, as well as}}
\begin{small}
  \begin{equation*}
  \begin{split}
& \xymatrix{\Hom_\Aopp(A,M) \ar[rr]^-{\Hom_\Aopp(\gve,M)} \ar[dr]_-{\simeq} & & \Hom_\Aopp(U,M)
    \ar[dl]^-{\gamma} \\ & M &  }
\end{split}
\end{equation*}
\end{small}
to which we refer as {\em contraunitality.}
\end{dfn}
The adjunction of the leftmost vertical arrow in the first diagram is to be understood
with respect to the right $A$-action
$
fa  \coloneqq  f(a \lact -) 
$
on $\Hom_\Aopp(U_\ract,M)$; the required right $A$-linearity of $\gamma$ then reads
\begin{equation}
\label{passionant}
\gamma\big(f(a \lact -) \big) = \gamma(f)a,
\end{equation}
usually excluding the well-definedness of a {\em trivial} right contraaction $f \mapsto f(1)$.
Any contramodule $M$ moreover has an {\em induced} left $A$-action 
\begin{equation}
\label{alleskleber}
am  \coloneqq  \gamma\big(m\gve(- \bract a) \big) = \gamma\big(m\gve(a \blact -) \big),
\end{equation}
which turns $M$ into an $A$-bimodule and $\gamma$ into an $A$-bimodule map,
\begin{equation}
  \label{tamtamdatam}
\gamma\big(f ( - \bract a)\big) = a \gamma\big(f{\sma -}\big),
\end{equation}
see \cite[Eq.~(2.37)]{Kow:WEIABVA}. This yields a
a forgetful functor 
\begin{equation}
\label{gaeta}
\mathbf{Contramod}\mbox{-}U \to \amoda
\end{equation}
from the category of right $U$-contramodules to that of $A$-bimodules.

For $f \in \Hom_\Aopp(U,M)$ we may (non-consistently, depending on readability in long computations) write both $\gamma(f \sma{-})$ as well as $\gamma(f \sma{\cdot})$ or simply $\gamma(f)$ to underline where the $U$-dependency is located: this way, the contraassociativity may be more compactly expressed as
\begin{equation}
\label{carrefour1}
\dot\gamma\big(\ddot\gamma(g( \cdot \otimes_\ahha \cdot\cdot))\big) 
= \gamma\big(g(-_{(1)} \otimes_\ahha -_{(2)})\big),
\end{equation}
for $g \in \Hom_\Aopp(U_\ract \otimes_\ahha \due U \lact {},M)$,
where the number of dots match the map $\gamma$ with the respective argument, and
contraunitality as 
\begin{equation}
\label{carrefour2}
\gamma( m \gve\sma{-}) = m
\end{equation}
for $m \in M$.
Finally, 
a {\em morphism} $\gvf \colon  M \to M'$ of contramodules is a map of right $A$-modules commuting with the contraaction, that is,
$
\gvf\big(\gamma(f)\big) =
\gamma\big(\gvf \circ f\big).
$

\subsubsection{Anti Yetter-Drinfel'd contramodules}
\label{atacvantaggi}
As already mentioned, coefficients in cyclic (co)homology theories typically have more than one algebraic structure, like actions, coactions, contraactions, and so forth. A compatibility between these is in general not required as long as one does not impose the condition that the cyclic operator powers to the identity. On the contrary, if one does, one is led to the notion of {\em anti Yetter-Drinfel'd} kind of objects:

\begin{dfn}
\label{chelabertaschen1}
An {\em anti Yetter-Drinfel'd (aYD) contramodule} $M$ over a left Hopf algebroid $(U,A)$ is a left $U$-module (with action denoted by juxtaposition) being at the same time a right $U$-contramodule (with contraaction $\gamma$) such that both underlying $A$-bimodule structures \eqref{peinture} and \eqref{gaeta} coincide, 
{\em i.e.}, 
\begin{equation}
\label{romaedintorni}
a \lact m \ract b = amb, \qquad m \in M, \ a,b \in A,
\end{equation}
and such that contraaction followed by action results in 
\begin{equation}
\label{nawas1}
u (\gamma(f)) = \gamma \big(u_{+(2)} f(u_-\sma{-}u_{+(1)}) \big), \qquad \forall u \in U, \ f \in \Hom_\Aopp(U,M).
\end{equation}
If action followed by contraaction results in the identity, {\em i.e.}, for all $m \in M$
\begin{equation}
\label{stablehalt}
\gamma(\sma{-}m)= m
\end{equation}
holds, then $M$ is called {\em stable}, where 
$\sma{-}m  \colon u \mapsto um$ as a map in $ \Hom_\Aopp(U,M)$.
\end{dfn}
In \cite[p.~1093]{Kow:WEIABVA}
one can find additional information about the (not so obvious) well-definedness of Eq.~\eqref{nawas1} and further implications: for example, if \eqref{romaedintorni} holds, then
\begin{equation}
  \label{hatschi}
\gamma(a \lact f\sma{-}) = 
\gamma\big(f(a \blact -)\big)
\end{equation}
is true as well,
where on the left hand side the left $A$-action on $M$ is meant.

\begin{rem}
  \label{yd}
  The category $\contramodu$ of right $U$-contramodules is, in general, not monoidal and therefore neither is $\cayd$, the category of anti Yetter-Drinfel'd contramodules, nor $\csayd$, the category of stable ones.
  However, in \cite[Prop.~3.3]{Kow:ANCCOTCDOE}
  it is shown that $\contramodu$ is a left module category over $\ucomod$, the monoidal category of left $U$-comodules ({\em cf.}~\S\ref{suppe}).
This
restricts to the structure
  $$
  \yd \times  \cayd
  \to  \cayd, \quad (N, M) \mapsto \hom^r(N,M)
$$
  of a left module category on $\cayd$ over the monoidal category $\yd$ of Yetter-Drinfel'd modules (these are $A$-bimodules with compatible left $U$-action and left $U$-coaction, which form the monoidal centre of $\umod$, see \cite[\S4]{Schau:DADOQGHA}).
  Note that this restriction 
 is precisely induced by the action \eqref{methamill7} defining the right internal Homs for $\umod$.
\end{rem}

\subsection{The bimodule centre in the bialgebroid module category}

Having introduced contramodules, we can now come back to examine the centre of $\umod$ with respect to its adjoint actions. 
Recall from Definition \ref{wisconsin}
that the centre $\cZ_{\umod^\op}(\umod)$ is
formed by all pairs $(M, \tau)$ of objects $M \in \umod$ for which there is a family of isomorphisms
$$
\tau_\enne \colon 
N \fren M
\stackrel{\simeq}{\lra}
M \frenop N
$$
natural in $N$. With respect to its full subcategory~$\cZ'_{\umod^\op}(\umod)$ of stable objects as in Definition \ref{stablepable},
we
have the following result:

\begin{theorem}
  \label{tegel1}
Let a left bialgebroid $(U, A)$ in addition be left Hopf.
\begin{enumerate}
  \compactlist{99}
  \item
    Then any
    aYD contramodule $M$ induces a central structure
  \begin{equation*}
    \tau_\enne \colon 
    \hom^r(N,M)
    \to
    \hom^\ell(N,M),
    \end{equation*}
explicitly given on the level of $k$-modules by
    \begin{small}
    \begin{equation}
  \label{michigan21}  
      \begin{array}{rcl}
        \tau_\enne \colon
\Hom_\Aopp(N, M)
    &\to&
        \Hom_\uhhu(N \otimes_\ahha \due U \lact {}, M),
\\[2pt]
g &\mapsto& \big\{
n \otimes_\ahha u \mapsto
\gamma\big((u \pmact g)(\sma\cdot n)\big)
\big\}, 
  \\[2pt]
  \big\{\gamma(f(n \otimes_\ahha -)) \mapsfrom n\big\} &\mapsfrom& f.
      \end{array}
    \end{equation}
    \end{small}
   \item
      Vice versa, for a pair $(M, \tau)$ in the centre $\cZ_{\umod^\op}(\umod)$, the right $U$-contraaction on $M$ defined by means of
      \begin{equation}
        \label{ghettokaisers3}
\gamma(g)  \coloneqq  (\tau_\uhhu g)(1 \otimes_\ahha 1), 
\end{equation}
      for every $g \in \Hom_\Aopp(U_\ract,M)$, induces the structure of an anti Yet\-ter-Drin\-fel'd contramodule on $M$.
%
\item
Both preceding parts together imply an equivalence
$$
 \cayd \simeq \cZ_{\umod^\op}(\umod)
  $$
 of categories.
\item
  Imposing stability on the respective objects leads to an equivalence
  $$
 \csayd \simeq \cZ'_{\umod^\op}(\umod)
 $$
of subcategories.
\end{enumerate}
\end{theorem}

\begin{proof}
  (i): That $\tau_\enne$ (and its inverse) is well-defined and
  a morphism of left $U$-modules if $M$ is an aYD contramodule, and invertible in the given sense if $M$ is stable has already been proven in \cite[Thm.~4.15]{Kow:WEIABVA}. We only explicitly show here that $\tau^{-1}_\enne$ is a $U$-module morphism to illustrate where the aYD condition \eqref{nawas1} is precisely needed: for $f \in \Hom_\uhhu(N \otimes_A U,M)$ and $n \in N$, we have
\begin{small}
  \begin{equation*}
    \begin{array}{rcl}
(u \pmact \tau^{-1}_\enne f)(n) 
&
      \stackrel{\scriptscriptstyle{\eqref{pmact}}}{=}
      &
u_+(\tau^{-1}_\enne f)(u_-n) 
\stackrel{\scriptscriptstyle{\eqref{michigan21}}}{=}
     u_+ \big( \gamma(f(u_- n \otimes_A -) \big)
     \\
     &
     \stackrel{\scriptscriptstyle{\eqref{nawas1}}}{=}
&
     \gamma\big(u_{++(2)} f(u_- n \otimes_A u_{+-} {\sma -}u_{++(1)}) \big)
\\
  &   \stackrel{\scriptscriptstyle{\eqref{Sch5}}}{=}
&
\gamma\big(u_{+(2)}u_- f(n \otimes_A {\sma -}u_{+(1)}) \big)
      \stackrel{\scriptscriptstyle{\eqref{Sch2}}}{=}
      \gamma\big(f(n \otimes_A {\sma -}u) \big)
      \stackrel{\scriptscriptstyle{\eqref{umact}}}{=}
\tau^{-1}_\enne(u \umact f)(n),
 \end{array}
   \end{equation*}
\end{small}
where in the fourth step we used the $U$-linearity of $f$.
  Hence, 
  \begin{equation}
    \label{abholstation}
    u \pmact \tau^{-1}_\enne(f)
    =
    \tau^{-1}_\enne(u \umact f),
\end{equation}
as claimed.
Let us moreover show that $\tau$ is natural in $N$: for any left $U$-module morphism $\gs \colon  N \to N'$ we want to see that $\tau_{\enne} \circ \hom^r(\gs, M) = \hom^\ell(\gs, M) \circ \tau_{\enne'}$.
  Indeed, by the $U$-linearity of $\gs$, one obtains
  \begin{equation}
    \label{omonia}
\begin{split}
    \tau_{\enne} (g \circ \gs)(n \otimes_\ahha u)
    &=  \gamma\big((u \pmact (g \circ \gs))(\sma\cdot n)\big)
    \\
& = \gamma\big(u_+g(u_-\sma\cdot \gs(n))\big)
    = \tau_{\scriptscriptstyle N'}(g) (\gs(n) \otimes_\ahha u), 
\end{split}
\end{equation}
for any $g \in \Hom_\Aopp(N', M)$ and $n \in N$.

On top, we need to prove that the hexagon axiom \eqref{tarrega7} commutes, which here takes the following explicit form:
\begin{small}
\begin{equation}
\label{tarrega81}
\xymatrix@C=0.38cm{
  \Hom_\Aopp(P, \Hom_\uhhu(N \otimes_\ahha U, M))
  \ar@{<-}[rr]_{\gvt_{\scriptscriptstyle P,M,N}}
\ar@{<-}[d]_{\Hom_\Aopp(P, \tau_\enne) \, }
& &
\Hom_\uhhu(N \otimes_\ahha U, \Hom_\Aopp(P,M))
\ar[d]^{\, \Hom_\uhhu(N \otimes_\ahha U, \tau_\pehhe)}
&
  \\
  \Hom_\Aopp(P, \Hom_\Aopp(N, M))
  \ar@{<-}[d]_{\phi_{\scriptscriptstyle P, N, M}}
  & &
  \Hom_\uhhu(N \otimes_\ahha U, \Hom_\uhhu(P \otimes_\ahha U, M))
  \ar@{<-}[d]^{\quad \psi_{\scriptscriptstyle M, P, N}}
  \\
  \Hom_\Aopp(P \otimes_\ahha N, M)
  \ar[rr]_{\tau_{P \otimes N}}
  & &
    \Hom_\uhhu(P \otimes_\ahha N \otimes_\ahha U, M)
  }
\end{equation}
\end{small}
Verifying that this diagram in fact commutes with respect to the central structure \eqref{michigan21} is done as follows. First, for better readability, by abuse of notation let us again abbreviate $\gvt = \gvt_{\scriptscriptstyle P,N,M}$, and likewise for $\phi$ and $\psi$.
For
$f \in \Hom_\Aopp(P \otimes_\ahha N, M)$, one then directly computes:
\begin{small}
\begin{eqnarray*}
  &&
  (\psi^{-1} \circ \Hom_\uhhu(N \otimes_A U, \tau_\pehhe)  \circ \gvt^{-1} \circ \Hom_\Aopp(P, \tau_\enne) \circ \phi \circ f)(p \otimes_\ahha n \otimes_\ahha u)
\\
&\overset{\scriptscriptstyle{\eqref{leitz12a}}}{=}&
  (\Hom_\uhhu(N \otimes_A U, \tau_\pehhe)  \circ \gvt^{-1} \circ \Hom_\Aopp(P, \tau_\enne) \circ \phi \circ f)(n \otimes_\ahha u)(p \otimes_\ahha 1)
   \\
  &
\overset{\scriptscriptstyle{\eqref{michigan21}}}{=} 
&
\gamma\pig(
  (\gvt^{-1} \circ \Hom_\Aopp(P, \tau_\enne) \circ \phi \circ f)(n \otimes_\ahha u)(\sma\cdot p)
\pig)
    \\
  &
\overset{\scriptscriptstyle{\eqref{leitz33}}}{=} 
&
\gamma\pig(
\big(\Hom_\Aopp(P, \tau_\enne) \circ \phi \circ f\big)
(u_- \sma\cdot p)(n \otimes_\ahha u_+)
\pig)
    \\
  &
\overset{\scriptscriptstyle{\eqref{michigan21}}}{=} 
&
\dot\gamma\Big(\ddot\gamma\pig(
\big(u_+ \pmact \big((\phi \circ f)
(u_- \sma\cdot p)\big)(\sma{\cdot\cdot}n)
\pig)\Big)
    \\
  &
\overset{\scriptscriptstyle{\eqref{pmact}}}{=} 
&
\dot\gamma\Big(\ddot\gamma\pig( u_{++}\big((\phi \circ f)
(u_- \sma\cdot p)\big)(u_{+-} \sma{\cdot\cdot}n)
\pig)\Big)
\\
&
\overset{\scriptscriptstyle{\eqref{carrefour1}, \eqref{Sch5}}}{=} 
&
\gamma\pig( u_{+} (\phi \circ f)
(u_{-(1)} {\sma\cdot}_{(1)} p)(u_{-(2)} {\sma\cdot}_{(2)} n)
\pig)
\\
&
\overset{\scriptscriptstyle{\eqref{leitz11}, \eqref{mancino}}}{=} 
&
\gamma\pig( u_{+} f
\big((u_{-} \sma\cdot) \mancino (p \otimes_\ahha n)\big)
\pig)
\\
&
\overset{\scriptscriptstyle{\eqref{pmact}}}{=} 
&
\gamma\pig( \big(u \pmact f\big)\big(\sma\cdot \mancino (p \otimes_\ahha n)\big)
\pig)
\\
&
\overset{\scriptscriptstyle{\eqref{michigan21}}}{=} 
&
\tau_{\scriptscriptstyle P \otimes_\ahha N} f (p \otimes_\ahha n \otimes_\ahha u),
\end{eqnarray*}
\end{small}
which proves the commutativity of Diagram \eqref{tarrega81}.

  (ii):
In this part, we have to show first that \eqref{ghettokaisers3} indeed defines a contraaction in the sense of Definition \ref{schoenwaers}. To start with, the $U$-linearity \eqref{abholstation} of $\tau_\uhhu$ resp.\ of its inverse implies that
\begin{small}
\begin{eqnarray*}
 \gamma\big(g(a \lact -)\big)
 &\overset{\scriptscriptstyle{\eqref{pmact}, \eqref{Sch9}}}{=}
 &
 \big(\tau_\uhhu (t(a) \pmact g)\big)(1 \otimes_\ahha 1)
   \\
  &
\overset{\scriptscriptstyle{\eqref{abholstation}}}{=} 
&
 \big(t(a) \umact \tau_\uhhu g\big)(1 \otimes_\ahha 1)
    \\
  &
\overset{\scriptscriptstyle{\eqref{umact}}}{=} 
&
 \big(\tau_\uhhu g\big)(1 \otimes_\ahha t(a))
\ \ \overset{\scriptscriptstyle{\eqref{mancino}}}{=} \ \
 \big(\tau_\uhhu g\big)(1 \otimes_\ahha 1)a
\ \ \overset{\scriptscriptstyle{\eqref{ghettokaisers3}}}{=} \ \
 \gamma(g)a
\end{eqnarray*}
    \end{small}
for any $a \in A$, which is the required right $A$-linearity \eqref{passionant}.

As for contraassociativity, observe first that the coproduct $\Delta \colon  U \to U \otimes_\ahha U$ is a morphism in $\umod$ as implied by the diagonal action \eqref{mancino}. We therefore have, by means of the naturality \eqref{omonia} of the central structure,
that
$
\tau_{\scriptscriptstyle{U}} (g \circ \Delta)
=
(\tau_{\scriptscriptstyle{U \otimes_\ahha U}} g) \circ (\Delta \otimes_\ahha \id)
$
for $g \in \Hom_\Aopp(U \otimes_\ahha U, M)$, and using this in the first step below, together with the hexagon axiom \eqref{tarrega81} in the third, we obtain:
\begin{small}
\begin{eqnarray*}
\gamma\big(g \circ \gD)
&\overset{\scriptscriptstyle{\eqref{ghettokaisers3}}}{=}&
 \pig(\tau_{\scriptscriptstyle{U \otimes_\ahha U}} g) \circ (\Delta \otimes_\ahha \id)\pig)(1 \otimes_\ahha 1)
   \\
  &
\overset{\scriptscriptstyle{}}{=} 
&
 (\tau_{\scriptscriptstyle{U \otimes_\ahha U}} g)(1 \otimes_\ahha 1 \otimes_\ahha 1)
    \\
  &
\overset{\scriptscriptstyle{\eqref{tarrega81}}}{=} 
&
\big(\psi^{-1} \circ \Hom_\uhhu(U \otimes_\ahha U, \tau_\uhhu)  \circ \gvt^{-1} \circ \Hom_\Aopp(U, \tau_\uhhu) \circ \phi \circ g\big)(1 \otimes_\ahha 1 \otimes_\ahha 1)
    \\
  &
\overset{\scriptscriptstyle{\eqref{leitz12a}}}{=} 
&
\big(\Hom_\uhhu(U \otimes_\ahha U, \tau_\uhhu)  \circ \gvt^{-1} \circ \Hom_\Aopp(U, \tau_\uhhu) \circ \phi \circ g\big)(1 \otimes_\ahha 1)(1 \otimes_\ahha 1)
    \\
  &
\overset{\scriptscriptstyle{\eqref{ghettokaisers3}}}{=} 
&
\dot\gamma\big((\gvt^{-1} \circ \Hom_\Aopp(U, \tau_\uhhu) \circ \phi \circ g)(1 \otimes_\ahha 1)\sma\cdot \big)
\\
&
\overset{\scriptscriptstyle{\eqref{leitz33}}}{=} 
&
\dot\gamma\big((\Hom_\Aopp(U, \tau_\uhhu) \circ \phi \circ g)\sma\cdot (1 \otimes_\ahha 1)\big)
\\
&
\overset{\scriptscriptstyle{\eqref{ghettokaisers3}}}{=} 
&
\dot\gamma\big(\ddot\gamma((\phi \circ g)\sma\cdot \sma{\cdot\cdot})\big)
\\
&
\overset{\scriptscriptstyle{\eqref{leitz11}}}{=} 
&
\dot\gamma\big(\ddot\gamma(g(\cdot \otimes_\ahha \cdot\cdot))\big),
\end{eqnarray*}
\end{small}
which proves the contraassociativity \eqref{carrefour1}.
Contraunitality 
is once more proven with the help of the naturality 
of $\tau$: the bialgebroid counit $U \to A$ defines an $U$-action on $A$ by means of $u \mancino a  \coloneqq  \gve(u \bract a)$ 
and, by $\gve(uv) = \gve(u \bract \gve(v))$, this yields a morphism in $\umod$.
Considering then that for $N = A$ the central structure $\tau_\ahha \colon  \hom^r(A, M) \!\simeq\! M \!\to\! M \!\simeq\! \hom^\ell(A,M)$ is the identity map, we have
$$
\gamma(m\gve\sma\cdot )
=
\tau_\uhhu(L_m \circ \gve)(1_\uhhu \otimes_\ahha 1_\uhhu)
=
\tau_\ahha(L_m)(\gve(1_\uhhu) \otimes_\ahha 1_\uhhu)
=
L_m(1_\ahha)
= m, 
$$
which is the contraunitality \eqref{carrefour2}, 
where we defined
$L_m \colon  A \to M, \ a \mapsto ma$ as an element in $\Hom_\Aop(A, M) \simeq M$.

That the so-defined right $U$-contraaction \eqref{ghettokaisers3} together with the left $U$-action \eqref{pmact} defines on $M$ an aYD structure is seen as follows: for $u \in U$ and $g \in \Hom_\Aopp(U,M)$, we have
\begin{equation*}
\begin{split}
  u\big(\gamma(g)\big)
  &=
  u (\tau_\uhhu g)(1 \otimes_\ahha 1)
  \\
  &
  = (\tau_\uhhu g)(u_{(1)} \otimes_\ahha u_{(2)})
  \\
  &
  = (u_{(2)} \umact \tau_\uhhu g)(u_{(1)} \otimes_\ahha 1)
\\&
  = \tau_\uhhu (u_{(2)} \pmact g)(u_{(1)} \otimes_\ahha 1)
  \\
  &
  = \tau_\uhhu \big((u_{(2)} \pmact g)(\sma\cdot u_{(1)})\big)(1 \otimes_\ahha 1)
  = \gamma \big((u_{(2)} \pmact g)(\sma\cdot u_{(1)})\big),
  \end{split}
\end{equation*}
where in the second step we used the $U$-linearity of $\tau_\uhhu g$ together with \eqref{mancino}, moreover Eq.~\eqref{umact} in the third step, in the fourth that  $\tau_\uhhu$ is a left $U$-module morphism, see Eq.~\eqref{abholstation}, and in the fifth the naturality of $\tau_\uhhu$ along with the fact that right multiplication $R_u \colon  U \to U, \ v \mapsto vu$ with an element $u \in U$ is a morphism in $\umod$. By \eqref{peinture},
this simultaneously proves \eqref{romaedintorni} and \eqref{nawas1}.

(iii):
In this third part, we have to show two things:
first,
that any morphism $M \to M'$ of aYD contramodules over $U$ induces a morphism $(M, \tau) \to (\tilde M, \tilde\tau)$ between the corresponding objects in the bimodule centre (and vice versa); second, that the two procedures of how to obtain a central structure from a right $U$-contraaction and a right $U$-contraaction from a central structure are mutually inverse.

As for the first issue, if $\gvf \colon  M \to \tilde M$ is a morphism of aYD contramodules,
 we have to show that for any $N \in \umod$ the diagram
\begin{small}
\begin{equation}
\label{verite1}
\xymatrix{
\Hom_\uhhu(N \otimes_\ahha U, M)
  \ar@{<-}[r]^-{\tau_N}
\ar[d]_{\Hom_\uhhu(N \otimes_\ahha U, \gvf) \, }
& 
\Hom_\Aopp(N,M) 
\ar[d]^{\, \Hom_\Aopp(N, \gvf)}
  \\
  \Hom_\uhhu(N \otimes_\ahha U, \tilde M)
  \ar@{<-}[r]_-{\tilde\tau_N}
  & 
  \Hom_\Aopp(N, \tilde M)
  }
\end{equation}
\end{small}
commutes, and this is obvious since $\gvf$ is both a morphism of right $U$-contramodules and left $U$-modules: therefore, for $f \in \Hom_\uhhu(N \otimes_\ahha U, M)$,
\begin{small}
  \begin{equation*}
    \begin{split}
 \gvf\big(\tau^{-1}_N f(n) \big)
 =
 \gvf\big(\gamma(f(n \otimes_\ahha -))\big)
 =
 \gamma\big((\gvf \circ f)(n \otimes_\ahha -)\big)
 =
 \tilde\tau^{-1}_{N}(\gvf \circ f)(n),
 \end{split}
 \end{equation*}
\end{small}
and hence also for $\tau_\enne$.
The other way round, let $\gvf \colon  (M, \tau) \to (\tilde M, \tilde\tau)$ be a morphism of objects in the centre $\cZ_{\umod^\op}(\umod)$, which means that $\gvf$ is a left $U$-module map and that diagram \eqref{verite1} commutes. In order to define a morphism of aYD contramodules, we only need to prove that $\gvf$ is also a right $U$-contramodule morphism as well. Indeed,
 $$
 \gvf\big(\gamma(g)\big)
 = \gvf\big(\tau_\uhhu g(1 \otimes_\ahha 1)\big) 
 = \tau_\uhhu \big((\gvf \circ g)(1 \otimes_\ahha 1)\big)
 = \gamma(\gvf \circ g)
 $$
for $g \in \Hom_\Aopp(N,M)$.

Second, we have to show that obtaining a central structure from a right $U$-contraaction and a right $U$-contraaction from a central structure are mutually inverse. As a matter of fact, if a right $U$-contraaction $\gamma$ on $M$ is given and a corresponding central structure $\tau$ (and its inverse) is defined by means of Eq.~\eqref{michigan21}, which, in turn, defines a right $U$-contraaction as in Eq.~\eqref{ghettokaisers3}, denoted by $\tilde\gamma$ for the moment, we have for $g \in \Hom_\Aopp(U, M)$
$$
\tilde\gamma(g) = \tau_\uhhu g(1 \otimes_\ahha 1) = \gamma\big((1 \pmact g)(\sma\cdot 1)\big) = \gamma(g),
$$
which is precisely the right $U$-contraaction we started with.

Vice versa, given a central structure $\tau$ that defines a right $U$-contraaction as in Eq.~\eqref{ghettokaisers3} that, in turn, defines a central structure as in Eq.~\eqref{michigan21}, denoted by $\sigma$ for the moment, equally reproduces the central structure $\tau$ we started with.
Indeed,
by Eqs.~\eqref{abholstation} and \eqref{umact},
we have 
\begin{equation*}
  \begin{split}
  \sigma_\enne^{-1}
  g(n \otimes_\ahha u)
  &=
  \gamma\big((u \pmact g)(\sma\cdot n)\big)
  =
 \big(u \umact (\tau_\uhhu
 g(\sma\cdot n))\big) (1 \otimes_\ahha 1)
 \\
 &=
 (\tau_\uhhu
 g(\sma\cdot n)) (1 \otimes_\ahha u)
= \tau_\uhhu
 g (n \otimes_\ahha u),
  \end{split}
  \end{equation*}
for $g \in \Hom_\Aopp(N, M)$,
where in the last step we once again used the naturality of $\tau_{\sma\cdot }$ with respect to the map $R_n \colon  U \to N, \ u \mapsto un$ as above.

(iv):
Given part (iii), here we only have to add that the two notions of stability from Definition \ref{stablepable} \& \ref{chelabertaschen1} imply each other.

Let us first show that the object $(M, \tau) \in \cZ_{\umod^\op}(\umod)$ constructed in (i) actually lies in the full subcategory $\cZ'_{\umod^\op}(\umod)$
if $M$ is  stable in the sense of \eqref{stablehalt} as an aYD contramodule, {\em i.e.}, we have to verify \eqref{subcategory}.
Suppressing the left and right unit constraints for $\mathbb{1} = A$, we have
\begin{equation*}
  \begin{array}{rcl}
(\zeta^{-1} \circ \Hom_\uhhu(A, \tau_\emme) \circ \xi \circ \id_\emme)(m)
    & \overset{\scriptscriptstyle{\eqref{ad1}}}{=} &
(\Hom_\uhhu(A, \tau_\emme) \circ \xi \circ \id_\emme)(1_\ahha)(m \otimes_\ahha 1_\uhhu)
\\
    & \overset{\scriptscriptstyle{\eqref{michigan21}}}{=} &
\gamma\big( (\xi \circ \id_\emme)(1_\ahha)(\sma\cdot m) \big)
\\
& \overset{\scriptscriptstyle{\eqref{ad2}}}{=} &
\gamma\big(\id_\emme(\sma\cdot m) \big)
\\
& \overset{\scriptscriptstyle{\eqref{stablehalt}}}{=} &
m 
  \end{array}
\end{equation*}
for all $m \in M$,
hence $\zeta^{-1} \circ \Hom_\uhhu(A, \tau_\emme) \circ \xi \circ \id_\emme = \id_\emme$,
which was to prove.

Vice versa,
assume that
$(M, \tau) \in \cZ'_{\umod^\op}(\umod)$, that is, $(M, \tau)$ belongs to those objects in the centre for which $\id_\emme \in \Hom_\uhhu(M,M)$ is mapped to itself by the chain of isomorphisms in \eqref{subcategory}. As above, 
the map $R_m \colon  u \mapsto um$ in $\Hom_\Aopp(U,M)$ is a morphism in $\umod$ for any $m \in M$,
and therefore 
\begin{small}
$$
\gamma(\sma{-}m)
= (\tau_\uhhu (R_m))(1 \otimes_\ahha 1)
\stackrel{\scriptscriptstyle{\eqref{omonia}}}{=}
(\tau_\emme \id_\emme )(R_m(1) \otimes_\ahha 1)
=
(\tau_\emme \id_\emme )(m \otimes_\ahha 1)
=
m,
$$
\end{small}
by naturality again, which signifies
the stability of $M$ in the sense of \eqref{stablehalt} with respect to the contraaction. Here, in the last step we used the defining property of $\cZ'_{\umod^\op}(\umod)$ as it explicitly results from the inverses of the adjunctions \eqref{ad1} and \eqref{ad2} in case $P=A$.
\end{proof}

\begin{rem}
  \label{hours2}
  If one desires more structural symmetry and decides to work with the left and right internal Homs that already exist on the bialgebroid level in the spirit of Remark \ref{hours}, then the central structure comes out as 
   \begin{small}
    \begin{equation*}
  \label{michigan213}  
      \begin{array}{rcl}
        \tau_\enne \colon
\Hom_\uhhu(U_\ract \otimes_\ahha N, M)
       &\to&
       \Hom_\uhhu(N \otimes_\ahha \due U \lact {}, M),        
  \\[2pt]
f &\mapsto&
  \big\{
n \otimes_\ahha u \mapsto
  \gamma\big(f(u \otimes_
\ahha \sma\cdot n)\big)
\big\}, 
\\[2pt]
\big\{
\gamma\big((u_{(2)} \pmact \tilde{f})(n \otimes_\ahha \sma\cdot u_{(1)})\big)
\mapsfrom u \otimes_\ahha n 
\big\}
&\mapsfrom& \tilde{f}
      \end{array}
    \end{equation*}
   \end{small}
  for any (stable) aYD contramodule $M$.
%
 Quite on the contrary, if not only a left Hopf structure but also a right one were present, as also already briefly touched on in Remark \ref{hours}, one would obtain
   \begin{small}
    \begin{equation*}
  \label{michigan212}  
      \begin{array}{rcl}
\tau_\enne \colon   \Hom_\Aopp(N, M)
    &\to&
\Hom_\ahha(N, M),
  \\[2pt]
g &\mapsto& 
\big\{ n \mapsto 
\gamma\big(g(\sma\cdot n)\big)\big\},
  \\[2pt]
  \big\{
  \gamma\big((\sma\cdot  \pmact f)(n)\big)
\mapsfrom n
  \big\}
 &\mapsfrom& f
      \end{array}
    \end{equation*}
   \end{small}
   for the central structure.
However, we will be going on with the more general approach presented in Theorem \ref{tegel1}.
  \end{rem}

\subsection{Traces on $\umod$}
\label{internalflight1}

The bimodule category centre just discussed now leads to a (weak) trace functor on $\umod$ as in Eq.~\eqref{cupper}, which, in turn, allows for a cyclic operator on the cochain complex defining a cyclic cohomology theory for $\Ext$-groups as will be discussed in the subsequent section. Combining the general Lemma \ref{funkyfunk} with Theorem \ref{tegel1}, we immediately obtain:

\begin{theorem}
\label{ganzleerheute}
If a left bialgebroid $(U,A)$ is left Hopf and $(M, \gamma)$ an anti Yetter-Drinfel'd contramodule, then $T \coloneqq  \Hom_\uhhu(-, M)$ yields a weak trace functor $\umod \to \kmod$, which is unital if $M$ is stable (and vice versa).
In particular, 
there are isomorphisms
$$
\tr_{\scriptscriptstyle{N,P}} \colon  \Hom_\uhhu(N \otimes_\ahha P, M) \stackrel\simeq\lra \Hom_\uhhu(P \otimes_\ahha N, M) 
$$
being functorial in  $N, P \in \umod$ that explicitly read
\begin{equation}
\label{hunga}
(\tr_{\scriptscriptstyle{N,P}} f)(p \otimes_\ahha n) = \gamma\big(f(n \otimes_\ahha \sma\cdot p)\big),
\end{equation}
for $n \in N, p \in P$.
\end{theorem}

\begin{proof}
  This essentially follows, as already anticipated, directly from  Lemma \ref{funkyfunk} along with Theorem \ref{tegel1}.
  One thing left to show is that in this context the general trace maps \eqref{cupper} assume the explicit form \eqref{hunga}.
%
 Indeed, 
\begin{equation*}
\begin{array}{rcl}
(\zeta^{-1} \circ \Hom_\uhhu(N,\tau_\pehhe) \circ \xi \circ f)(p \otimes_\ahha n)
  &
 \stackrel{\scriptscriptstyle{\eqref{ad1}}}{=}
  &
 (\Hom_\uhhu(N,\tau_\pehhe) \circ \xi \circ f)(n)(p \otimes_\ahha 1)
 \\
  &
 \stackrel{\scriptscriptstyle{\eqref{michigan21}}}{=}
  &
 \gamma\big((\xi f)(n)({\sma -}p)\big)
 \\
  &
 \stackrel{\scriptscriptstyle{\eqref{ad1}}}{=}
  &
 \gamma\big( f(n \otimes_\ahha {\sma -}p)\big)
\end{array}
\end{equation*}
for any $f \in \Hom_\uhhu(N \otimes_\ahha P, M)$, that is, \eqref{hunga} indeed.
As for the unitality of the trace functor in case $M$ is stable, setting $N = A$ we directly see that
$$
(\tr_{\scriptscriptstyle A,P}f)(p) = \gamma\big(f(\sma\cdot p) \big)
= \gamma\big(\sma\cdot f(p)\big) = f(p),
$$
using the $U$-linearity of $f$ and the stability of $M$.
%
  \end{proof}


\subsection{Cyclic structures on $\Ext$ and cyclic cohomology}
\label{shine}

In \cite[\S3.2]{Kow:ANCCOTCDOE},
we defined the structure of a cocyclic $k$-module
on the cochain complex computing $\Ext^\bullet_\uhhu(A,M)$,
where $M$ is, to begin with, a left $U$-module right $U$-contramodule with contraaction $\gamma$: that is, we added a {\em cocyclic operator} $\uptau$ compatible with the simplicial structure inducing the cochain complex. This way, if $M$ is a stable aYD contramodule, one obtains a cyclic coboundary
$$
B \colon  \Ext^\bullet_\uhhu(A,M) \to \Ext^{\bullet-1}_\uhhu(A,M)
$$
that squares to zero
(see, for example, \cite[\S\S2.5 \& 6.1]{Lod:CH} or \cite[\S5]{Tsy:CH} for general details on (co)\-cyc\-lic $k$-modules).

In this subsection, we want to show that the trace functor $T$ from Theorem \ref{ganzleerheute} resp.\ the map $\tr$ in \eqref{hunga} induce the same cocyclic operator that was obtained in \cite[Eq.~(3.10)]{Kow:ANCCOTCDOE}, hence induce the same cyclic cohomology for the complex computing $\Ext^\bullet_\uhhu(A,M)$.

  Let us assume that $U_\ract$ is flat as an $A$-module. In this case,
  $$
\Ext^\bullet_\uhhu(A, M) = H(\Hom_\uhhu({\rm Bar}_\bullet(U),M), b'),
  $$
where $ {\rm Bar}_\bullet(U) = (\due U \blact \ract)^{\otimes_\Aopp \bullet+1}$ with differential $b'$ 
is the {\em bar resolution} of $A$ (essentially defined by the multiplication in $U$ and with augmentation given by the counit $\gve$), and which is a left $U$-module by left multiplication on the first tensor factor. Elements in tensor powers over $\Aop$ will typically be denoted by the comma notation, that is, an elementary tensor in $U^{\otimes_\Aopp q}$ by $(u^1, \ldots, u^{q})$.

Applying for any $q \in \N$ the isomorphism
\begin{small}
  \begin{equation}
    \label{theta}
 \begin{array}{rcl}
   \gt \colon  \Hom_\uhhu({\rm Bar}_q(U),M) &\to& \Hom_\Aopp(U^{\otimes_\Aopp q}, M),
   \\
   g &\mapsto& g(1,\cdot),
   \\
  \sma\cdot f &\mapsfrom& f,
\end{array}
\end{equation}
\end{small}
where we denoted the left $U$-action on $M$ by juxtaposition,
we obtain that the $\Ext$-groups can equally be computed by the complex
\begin{equation*}
  \label{mostaccioli}
C^\bullet(U,M)  \coloneqq  \Hom_\Aopp(U^{\otimes_\Aopp\bullet}, M),
\end{equation*}
where the cofaces and codegeneracies in degree $q \in \N$ are explicitly given as
\begin{small}
\begin{equation*}
 \begin{array}{rcl}
  (\gd_i f)(u^1, \ldots, u^{q+1}) \!\!\!\!
  &=\!\!\!\!&
  \left\{\!\!\!
\begin{array}{l} 
u^1 f(u^2, \ldots, u^{q+1})
\\ 
 f(u^1, \ldots, u^{i} u^{i+1}, \ldots, u^{q+1})
\\
 f(u^1, \ldots, \gve(u^{q+1}) \blact u^q)
\end{array}\right.  
  \begin{array}{l} \mbox{if} \ i=0, \\ \mbox{if} \
  1 \leq i \leq q, \\ \mbox{if} \ i = q + 1,  \end{array} 
  \\[.5cm]
(\gs_j f)(u^1, \ldots, u^{q-1}) \!\!\!\! 
&=\!\!\!\!& f(u^1, \ldots, u^{j}, 1, u^{j+1}, \ldots, u^{q-1}) \quad \mbox{for} \  0 \leq j \leq q-1.
  \\[4pt]
\end{array}
\end{equation*}
\end{small}
By means of the cocyclic operator in the form
  \begin{small}
\begin{equation}
  \label{nadennwommama}
\begin{array}{rcl}
  (\uptau f)(u^1, \ldots, u^q) \!\!\!\!&=\!\!\!\!& \gamma\big(((u^1_{(2)} \cdots u^{q-1}_{(2)} u^q) \pmact  
  f)(-, u^1_{(1)}, \ldots, u^{q-1}_{(1)}) \big),
\end{array}
\end{equation}
\end{small}
  this becomes a cocyclic $k$-module in the sense of \cite[\S2.5]{Lod:CH}.

  To see that this cocyclic operator can indeed be considered as originating from a trace functor, we first have to lift it to
  $
\Hom_\uhhu({\rm Bar}_\bullet(U),M)
$
by the isomorphism \eqref{theta} in order to place it in the realm of $U$-linear maps. Secondly, the tensor products appearing in ${\rm Bar}_\bullet(U)$ are not the monoidal products in $\umod$, which would be needed in \eqref{hunga}; hence, another two $k$-module isomorphisms $\eta$ and $\chi$ are required. More precisely:

\begin{theorem}
  \label{lalacrimosa}
  Let the left bialgebroid $(U,A)$ be left Hopf and $M$ a stable anti Yetter-Drinfel'd contramodule. Then the diagram
\begin{equation}
\label{cyclictrace}
\xymatrix{
  \Hom_\Aopp(U^{\otimes_\Aopp \bullet}, M)
  \ar[rr]^{\uptau}
\ar[d]_{\gt^{-1}}
& &
  \Hom_\Aopp(U^{\otimes_\Aopp \bullet}, M)
\ar@{<-}[d]^{\, \gt}
&
  \\
   \Hom_\uhhu({\rm Bar}_\bullet(U), M)
  \ar[d]_{\eta^{-1}}
  & &
  \Hom_\uhhu({\rm Bar}_\bullet(U), M)
  \ar@{<-}[d]^{\chi}
  \\
  \Hom_\uhhu(U \otimes_\ahha U^{\otimes_\Aopp \bullet}, M)
  \ar[rr]_{\tr}
  & &
    \Hom_\uhhu(U^{\otimes_\Aopp \bullet} \otimes_\ahha U, M)
  }
\end{equation}
  commutes in any degree.
  \end{theorem}

\begin{proof}
  Explicitly, the two $k$-module isomorphisms $\eta$ and $\chi$ are given as follows: 
  define for any $q \in \N$
  \begin{small}
  \begin{eqnarray}
\nonumber
    \eta \colon 
    \Hom_\uhhu(U \otimes_\ahha (U^{\otimes_\Aopp q}), M) \!\!\!&\to\!\!\!& \Hom_\uhhu(U^{\otimes_\Aopp q+1}, M),
    \\
        \label{eta}
f \mapsto \big\{\hspace*{.5pt} (v, u^1, \ldots, u^q ) \!\!\!&\mapsto\!\!\!&  f(v_{(1)} \otimes_\ahha (v_{(2)} u^1, u^2, \ldots, u^q)\hspace*{.5pt}\big\},
\\
\nonumber
\big\{\hspace*{.5pt}g(v_+, v_- u^1,  u^2, \ldots  u^q) \!\!\!&\mapsfrom\!\!\!& \big( v \otimes_\ahha (u^1, \ldots, u^q)\big) \hspace*{.5pt} \big\} \mapsfrom g,
    \end{eqnarray}
\end{small}
  as well as
  \begin{small}
  \begin{eqnarray}
\nonumber
    \chi \colon 
    \Hom_\uhhu((U^{\otimes_\Aopp q}) \otimes_A U, M) \!\!\!&\to\!\!\!& \Hom_\uhhu(U^{\otimes_\Aopp q+1}, M),
    \\
        \label{chi}
f \mapsto \big\{\hspace*{.5pt}(u^1, \ldots, u^q , v) \!\!\!&\mapsto\!\!\!&  f\big((u^1_{(1)}, \ldots, u^q_{(1)}) \otimes_\ahha u^1_{(2)} \cdots u^q_{(2)}  v \big)\hspace*{.5pt}\big\},
\\
\nonumber
\big\{ \hspace*{.5pt} g(u^1_+, \ldots, u^q_+ ,  u^q_- \cdots  u^1_- v) \!\!\!&\mapsfrom\!\!\!& ((u^1, \ldots, u^q) \otimes_\ahha v\big) \hspace*{.5pt} \big\} \mapsfrom g,
    \end{eqnarray}
  \end{small}
  where on the left hand side  $(U^{\otimes_\Aopp q}) \otimes_A U$ is seen as left $U$-module via
  $
  w\big((u^1, \ldots, u^q) \otimes_A v\big)  \coloneqq 
  (w_{(1)} u^1, u^2, \ldots, u^q) \otimes_A w_{(2)} v,
$
  which is well-defined if the tensor product over $A$ relates the last tensor factor with the first by left multiplication with the target map. It is a straightforward check that the maps $\eta$ resp.\ $\chi$ are well-defined and that their asserted inverses invert them, indeed. We can then
  compute, for any $f \in  \Hom_\Aopp(U^{\otimes_\Aopp q}, M)$,
\begin{small}
\begin{equation*}
  \begin{array}{rcl}
    &&
    (\gt \circ \chi \circ \tr \circ \eta^{-1} \circ \gt^{-1} \circ f)(u^1, \ldots, u^q)
    \\
  &
 \stackrel{\scriptscriptstyle{\eqref{theta}}}{=}
  &
    (\chi \circ \tr \circ \eta^{-1} \circ \gt^{-1} \circ f)(1, u^1, \ldots, u^q)
 \\
  &
 \stackrel{\scriptscriptstyle{\eqref{chi}}}{=}
  &
    \big(\tr \circ \eta^{-1} \circ \gt^{-1} \circ f\big)\big((1, u^1_{(1)}, \ldots, u^{q-1}_{(1)}) \otimes_\ahha u^1_{(2)} \cdots u^{q-1}_{(2)}  u^q \big)
 \\
  &
 \stackrel{\scriptscriptstyle{\eqref{hunga}}}{=}
  &
 \gamma\pig( \big(\eta^{-1} \circ \gt^{-1} \circ f\big)\big(u^1_{(2)} \cdots u^{q-1}_{(2)}  u^q \otimes_\ahha (\sma\cdot , u^1_{(1)}, \ldots, u^{q-1}_{(1)})  \big) \pig)
 \\
  &
 \stackrel{\scriptscriptstyle{\eqref{eta}}}{=}
  &
 \gamma\pig( \big(\gt^{-1} \circ f\big)\big((u^1_{(2)} \cdots u^{q-1}_{(2)}  u^q)_+, (u^1_{(2)} \cdots u^{q-1}_{(2)}  u^q)_- \sma\cdot , u^1_{(1)}, \ldots, u^{q-1}_{(1)}  \big) \pig)
  \\
  &
 \stackrel{\scriptscriptstyle{\eqref{theta}}}{=}
  &
 \gamma\pig((u^1_{(2)} \cdots u^{q-1}_{(2)}  u^q)_+ f\big((u^1_{(2)} \cdots u^{q-1}_{(2)}  u^q)_- \sma\cdot , u^1_{(1)}, \ldots, u^{q-1}_{(1)}  \big) \pig)
  \\
  &
 \stackrel{\scriptscriptstyle{\eqref{pmact}}}{=}
  &
 \gamma\big(((u^1_{(2)} \cdots u^{q-1}_{(2)}  u^q) \pmact f)(-, u^1_{(1)}, \ldots, u^{q-1}_{(1)}) \big)
  \\
  &
 \stackrel{\scriptscriptstyle{\eqref{nadennwommama}}}{=}
  &
(\uptau f)(u^1, \ldots, u^{q}),
\end{array}
\end{equation*}
\end{small}
which means that diagram \eqref{cyclictrace} commutes and hence implies
that the co\-cyclic operator \eqref{nadennwommama} is induced by the trace functor from Theorem \ref{ganzleerheute}.
  \end{proof}

\begin{rem}
  \label{quelquechose}
  As already mentioned in Remark \ref{yd}, if $M$ is an aYD contra\-module and, say,
  $Q$ a Yetter-Drinfel'd module ({\em i.e.}, an element in the centre of $\umod$ if seen as a bimodule category over itself via the monoidal product),
  then $\hom^r(Q, M) = \Hom_\Aopp(Q,M)$ is again an aYD contramodule. Hence, if this aYD contramodule is stable (which is not equivalent to $M$ being stable),
  by once more exploiting the Hom-tensor adjunction
  $
\xi \colon   \Hom_\uhhu(P \otimes_\ahha N \otimes_\ahha Q, M)
  \simeq 
\Hom_\uhhu(P \otimes_\ahha N, \hom^r(Q, M))
  $, 
  it is possible to construct a trace functor
  $$
  T  \coloneqq  \Hom_\uhhu(- \otimes_A Q, M),
$$
  with $M$ and $Q$ as above, and corresponding trace map
  $$
\tr_{\scriptscriptstyle{N,P}} \colon  \Hom_\uhhu(N \otimes_\ahha P \otimes_\ahha Q, M) \stackrel\simeq\lra \Hom_\uhhu(P \otimes_\ahha N \otimes_\ahha Q, M), 
  $$
for arbitrary $N, P \in \umod$, which 
in the same way as in Theorem \ref{lalacrimosa} leads to the structure of a cyclic $k$-module on the complex computing $\Ext^\bullet_\uhhu(Q,M)$ if $U_\ract$ is $A$-flat. Since this produces
  even more unpleasant formul\ae \ than those seen so far \cite[Prop.~3.5]{Kow:ANCCOTCDOE}, we refrain from spelling out the details here.
\end{rem}

\section{Centres and anti Yetter-Drinfel'd modules}
\label{wartung}
We now, in a sense, dualise most of the ideas and results of the preceding section and dedicate our attention to the category of bialgebroid comodules.

\subsection{Comodules over bialgebroids}
\label{suppe}
A left (and analogously right) comodule over a left bialgebroid $(U,A)$ is simply a comodule over the appurtenant $A$-coring, see \cite[\S3]{BrzWis:CAC}: that is, a left $A$-module $M$ equipped with a coassociative and counital coaction
$\gl \colon  M \to U_\ract \otimes_\ahha M, \ m \mapsto m_{(-1)} \otimes_\ahha m_{(0)}$.
Defining $ ma \coloneqq  \gve(m_{(-1)} \bract a) m_{(0)} = \gve(a \blact m_{(-1)}) m_{(0)}$ for all $a \in A$ equips $M$ with a right $A$-action as well, and with respect to the resulting $A$-bimodule structure the coaction is $A$-bilinear in the sense of
\begin{equation}
  \label{taklinco}
\gl(amb) =  a \lact m_{(-1)} \bract b \otimes_\ahha m_{(0)}, \qquad a,b \in A.
  \end{equation}
On the other hand, 
by virtue of the bialgebroid properties, we have
\begin{small}
  \begin{equation*}
    \begin{split}
m_{(-1)} \otimes_\ahha m_{(0)}a
&=
m_{(-2)} \otimes_\ahha \gve(a \blact m_{(-1)})m_{(0)}
\\
&
=
m_{(-2)} \ract \gve(a \blact m_{(-1)}) \otimes_\ahha m_{(0)}
=
a \blact m_{(-1)} \otimes_\ahha m_{(0)},
    \end{split}
    \end{equation*}
\end{small}
so that the coaction effectively $\gl$ corestricts to a map
\begin{equation}
  \label{takcom}
\gl \colon  M \to U_\ract \times_\ahha M,
  \end{equation}
where the subspace $U \times_\ahha M \subset U \otimes_\ahha M$ is defined 
as
\begin{small}
\begin{equation}
    \label{takcom2}
   U_\ract \times_{\scriptscriptstyle A} M    \coloneqq 
   \big\{ {\textstyle \sum_i} u_i \otimes  m_i \in U_\ract  \otimes_\ahha M
    \mid {\textstyle \sum_i} a \blact u_i \otimes m_i = {\textstyle \sum_i} u_i \otimes m_ia,  \ \forall a \in A  \big\},
\end{equation}
\end{small}
see \cite{Tak:GOAOAA} for more information on the (lax monoidal) product $\times_\ahha$.
%
%
The category $\ucomod$ of left $U$-comodules is (strict) monoidal.

Analogous considerations hold for the category $\comodu$ of right $U$-co\-mod\-ules with respect to which we only explicitly state the $A$-bilinearity of a right coaction
$\rho \colon  M \to M \times_\ahha \due U \lact {}$, which reads
\begin{equation}
  \label{taklinco2}
\rho(amb) =  m_{(0)} \otimes_\ahha a \blact m_{(1)} \ract b, \qquad a,b \in A.
  \end{equation}
\subsubsection{A functor between comodule categories}
\label{paco}
The standard Hopf algebraic way of transforming a left $U$-comodule into a right one via the antipode or its possible inverse does not apply here (as there is no antipode, not even if $U$ is left or right Hopf)
but nevertheless if the left bialgebroid $(U, A)$ is right Hopf and $\due U \lact {}$ is $A$-projective, there is a strict monoidal functor $\ucomod \to \comodu$, as shown originally in \cite{Phu:TKDFHA} and later, somewhat enhanced, in \cite[Thm.~4.1.1]{CheGavKow:DFOLHA}. More concretely, given a left $U$-comodule $M$, the map
\begin{equation}
\label{soestrene}
M \to M \otimes_\ahha \due U \lact {}, \quad m \mapsto m_{[0]} \otimes_\ahha m_{[1]}  \coloneqq  \gve(m_{(-1)[+]} ) m_{(0)} \otimes_\ahha m_{(-1)[-]}
  \end{equation}
is a {\em right} coaction.
We refer to the proof of
\cite[Thm.~4.1.1]{CheGavKow:DFOLHA}
for the not entirely obvious verification that
if $\due U \lact {}$ is $A$-projective, then 
this is a well-defined operation.
We reserve the square bracket Sweedler notation $m \mapsto m_{[0]} \otimes_\ahha m_{[1]}$ throughout the entire text for {\em this} kind of right $U$-coaction {\em only}, starting from a left $U$-comodule.

Vice versa, if the left bialgebroid $(U, A)$ is {\em left} Hopf and
$U_\ract$ is $A$-projective, then there is a strict monoidal functor $\comodu \to \ucomod$ but we are not going to need this fact in the sequel.

\begin{rem}
  \label{curry}
  In case of a Hopf algebra, as follows from Eqs.~\eqref{sesam}, the above functor
  $\ucomod \to \comodu$
  is precisely the one induced by the inverse $S^{-1}$ of the antipode, whereas
  $\comodu \to \ucomod$ is induced by $S$.
  However, in striking contrast to the Hopf algebra case where it essentially does not matter whether one uses $S$ or $S^{-1}$ for either of the functors,
for a left bialgebroid
  there does not seem to be an obvious way to obtain a functor $\ucomod \to \comodu$ in case $(U,A)$ is {\em left} Hopf instead of {\em right} Hopf. This defect will become very visible when defining the left and right internal Homs in $\ucomod$.
  \end{rem}

For later and frequent use in technical computations,
for a left $U$-comodule $M$ over a left bialgebroid that is, in addition, right Hopf, one easily verifies by \eqref{taklinco} and \eqref{Tch4} that for any $m \in M$ the compatibility condition
\begin{equation}
  \label{blumare1}
(m_{[0](-1)} \otimes_\ahha m_{[0](0)}) \otimes_\ahha m_{[1]} = (m_{(-1)[+]} \otimes_\ahha m_{(0)}) \otimes_\ahha m_{(-1)[-]}
  \end{equation}
holds between left $U$-coaction and induced right $U$-coaction \eqref{soestrene} as tensor products in $(U_\ract \otimes_\ahha M)_\bract \otimes_\ahha \due U \lact {}\!$, where $(U \otimes_\ahha M)_\bract = U_\bract \otimes_\ahha M$.
If the left bialgebroid $(U,A)$ is both left and right Hopf, by \eqref{mampf3} one even has
\begin{equation}
  \label{blumare2}
m_{(-1)} \otimes_\ahha (m_{(0)[0]} \otimes_\ahha m_{(0)[1]}) = m_{[1]-} \otimes_\ahha (m_{[0]} \otimes_\ahha m_{[1]+}) 
\end{equation}
as 
tensor products in $U_\ract \otimes_\ahha \due {(M \otimes_\ahha \due U \lact {})} \blact {}\!$, where $\due {(M \otimes_\ahha U)} \blact {} = M \otimes_\ahha \due U \blact {}\!$.

\subsubsection{Anti Yetter-Drinfel'd modules}

In the previous sections, we added to objects in the monoidal category $\umod$ an additional structure (of right $U$-contraaction) compatible with the action, which led to the notion of aYD contramodules. Now, the monoidal category of interest is $\ucomod$ and the additional structure will be that of a {\em right} $U$-action. Note that the category $\modu$ of right $U$-modules over a left bialgebroid is not monoidal; nonetheless, one still has a forgetful functor $\modu \to \amoda$, with respect to which we denote the induced $A$-bimodule structure on a right $U$-module $M$ by
\begin{equation}
  \label{peinture2}
  a \blact m \bract b  \coloneqq  mt(a)s(b)
  \end{equation}
for $m \in M$, $a, b \in A$. Moreover, $\modu$ can be seen as a right module category over $\umod$ by means of 
$
\modu \times \umod \to \modu, \ (M, N) \mapsto M \otimes_\ahha N,
$
induced by the action on elements
\begin{equation}
  \label{wszw1}
  (m \otimes_\ahha n)
\mpsqact
  u
   \coloneqq 
  m u_{[+]} \otimes_\ahha u_{[-]} n, 
\end{equation}
for $m \in M, n \in N$, and $u \in U$.   

Analogous comments apply to the category of anti Yetter-Drinfel'd modules, which we are going to recall next \cite{BoeSte:CCOBAAVC, JarSte:HCHARCHOHGE} and which arise when
asking for compatibility between right $U$-action and left $U$-coaction.

\begin{dfn}
\label{chelabertaschen2}
An {\em anti Yetter-Drinfel'd (aYD) module} $M$ over a left Hopf algebroid is simultaneously a left $U$-comodule and {\em right} $U$-module (with action denoted by juxtaposition) such that both underlying $A$-bimodule structures from \eqref{peinture2} and \eqref{taklinco} coincide, 
 and such that action followed by coaction results into
 \begin{equation}
\label{funkydesertbreaks1}
 (mu)_{(-1)} \otimes_\ahha (mu)_{(0)} = u_- m_{(-1)} u_{+(1)} \otimes_\ahha m_{(0)} u_{+(2)}, \qquad u \in U, m \in M.
 \end{equation}
 An anti  Yetter-Drinfel'd contramodule is called {\em stable} if 
 $
 m = m_{(0)} m_{(-1)}.
 $
\end{dfn}

The category $\ayd$ of aYD modules (resp.~the category $\sayd$ of stable ones) is not monoidal as already the category of right $U$-modules is not so.

For later use, we want to state some alternative compatibility conditions in presence of more structure:
if the left bialgebroid $(U, A)$ is not only left Hopf but also right Hopf, the aYD condition \eqref{funkydesertbreaks1} is equivalent to
  \begin{equation}
   \label{funkydesertbreaks1a}
(mu_{[+]})_{(-1)} u_{[-]} \otimes_\ahha (mu_{[+]})_{(0)} = u_- m_{(-1)} \otimes_\ahha m_{(0)} u_+,
  \end{equation}
as an easy check using \eqref{mampf1} and \eqref{Tch2} reveals. In this case, Eq.~\eqref{funkydesertbreaks1} can also be reformulated with respect to the right $U$-coaction \eqref{soestrene}, that is,
\begin{equation}
  \label{funkydesertbreaks2}
  (mu)_{[0]}
  \otimes_\ahha (mu)_{[1]}
  =
  m_{[0]} u_{[+](1)} \otimes_\ahha u_{[-]} m_{[1]} u_{[+](2)}, 
\end{equation}
as one obtains (after a while) applying to \eqref{funkydesertbreaks1}, in this order, Eqs.~\eqref{soestrene}, \eqref{Tch9}, \eqref{takcom}, \eqref{Tch4}, \eqref{peinture2}, \eqref{mampf1}, \eqref{mampf2}, \eqref{Sch1}, and finally \eqref{Sch8}, along with the properties of a bialgebroid counit.  Moreover,
if $M$ is stable with respect to its left coaction,
that is,
  $m_{(0)} m_{(-1)} = m$,
  then it is so with respect to its right coaction \eqref{soestrene} as well, by which we mean
\begin{small}
  \begin{equation}
  \label{stableagain}
  m_{[0]} m_{[1]} = m_{[0](0)} m_{[1](-1)} m_{[1]} =  m_{(0)} m_{(-1)[+]} m_{(-1)[-]} =
  \gve(m_{(-1)}) \blact m_{(0)} = m,
  \end{equation}
  \end{small}
as results from \eqref{blumare1} and \eqref{Tch7}; and vice versa.

\subsection{Left and right closedness of $\ucomod$}

As said before, for a monoidal category being closed or even biclosed means the existence of internal Homs. In case of comodules, this leads to the notion of {\em rational} morphisms as introduced by Ulbrich \cite{Ulb:SPACOLM}, see also \cite{CaeGue:OTCORHM, SteOys:TWMTFCA} for more information on the subject in the realm of Hopf algebras. We adapt the idea to the bialgebroid case here.

\subsubsection{Right internal Homs in $\ucomod$}
\label{sahale}
Let $(U,A)$ be a left algebroid, $P$ be a right $U$-comodule with right coaction $p \mapsto p_{(0)} \otimes_\ahha p_{(1)}$ and $M$ a left $U$-comodule with left coaction $m \mapsto m_{(-1)} \otimes_\ahha m_{(0)}$, in the sense of \S\ref{suppe}. On $\Hom_\Aopp(P,M)$, consider the following customary $A$-bimodule structure
\begin{equation}
  \label{priamo}
(a \lact f \bract b)(p) = af(bp), \qquad a, b \in A, p \in P.
\end{equation}
Define then the map 
\begin{equation}
    \label{teams1}
\begin{array}{rcl}
  \gl^r \colon  \Hom_\Aopp(P,M) &\to& \Hom_\Aopp(P,U_\ract \otimes_\ahha M), \\
  f &\mapsto& \big\{p \mapsto f(p_{(0)})_{(-1)} p_{(1)} \otimes_\ahha f(p_{(0)})_{(0)} \big\}.
\end{array}
\end{equation}
  %
Now, the canonical map
$\jmath \colon  U_\ract \otimes_\ahha \Hom_\Aopp(P,M) \to \Hom_\Aopp(P,U_\ract \otimes_\ahha M)$
is an injection if $U_\ract$ is $A$-projective. We can then make the following definition:

\begin{dfn}
  \label{rightcom}
  For a right $U$-comodule $P$ and a left $U$-comodule $M$ over a left bialgebroid $(U,A)$ with $U_\ract$ projective over $A$, the $A$-bimodule
  $$
  \HOM^r(P,M) \coloneqq
\{ f \in \Hom_\Aopp(P,M) \mid \gl^r f \in  \mathrm{im}(\jmath) \}
  $$
is called the space of {\em (right) rational morphisms} from $P$ to $M$.
  \end{dfn}

In other words, $\HOM^r(P,M)$ consists of all $f \in \Hom_\Aopp(P,M)$
for which there exists an element
$
f_{(-1)} \otimes_\ahha f_{(0)} \in U_\ract \otimes_\ahha \Hom_\Aopp(P,M)
$
such that
$$
(\gl^r f)(p) = f_{(-1)} \otimes_\ahha f_{(0)}(p) 
$$
for all $p \in P$. By injectivity of the canonical map $\jmath$, we may simply write
$$
\gl^r f = f_{(-1)} \otimes_\ahha f_{(0)} 
$$
for any (right) rational $f$. If $U_\ract$ is finitely generated projective over $A$, then clearly all morphisms in $\Hom_\Aopp(P,M)$ are (right) rational.

\begin{lem}
  \label{high-gain}
Let $(U,A)$ be a left bialgebroid such that $U_\ract$ is projective over $A$.
If $P$ is a right $U$-comodule and $M$ a left $U$-comodule, then $\HOM^r(P,M)$ is a left $U$-comodule with coaction
given by $\jmath^{-1} \circ \gl^r$.
%
%
\end{lem}

\begin{proof}
  We need to show that 
  $\gl^r f$ lands in $ U_\ract \otimes_\ahha \HOM^r(P,M)$ for any $f \in \HOM^r(P,M)$ and to
  check that $\gl^r$ is counital and coassociative, and this will be done along
  the same line of argumentation as in \cite[Lem~2.2]{Ulb:SPACOLM}.
  Counitality is straightforward using the properties of a bialgebroid counit
  along with the $A$-linearity \eqref{taklinco} of the coaction on $M$. Furthermore, we have for any $p \in P$
    \begin{small}
\begin{eqnarray*}
\big((\id \otimes_\ahha \gl^r)\gl^r f\big)(p)
&
=
&
f_{(-1)} \otimes_\ahha (\gl^r f_{(0)})(p)
\\
  &
\overset{\scriptscriptstyle{\eqref{teams1}}}{=} 
&
f_{(-1)} \otimes_\ahha f_{(0)}(p_{(0)})_{(-1)} p_{(1)} \otimes_\ahha f_{(0)}(p_{(0)})_{(0)}
\\
  &
\overset{\scriptscriptstyle{\eqref{teams1}}}{=} 
&
f(p_{(0)})_{(-2)} p_{(1)} \otimes_\ahha f(p_{(0)})_{(-1)} p_{(2)} \otimes_\ahha f(p_{(0)})_{(0)}
\\
  &
=
&
\big((\gD \otimes_\ahha \id)\gl^r f\big)(p).
\end{eqnarray*}
    \end{small}   
The so-obtained equation
    not only shows coassociativity but also that $\gl^r f \in U \otimes_\ahha \HOM^r(P,M)$: the $A$-bimodule $\Hom^r(N,M)$ can be seen as a pull-back for $\gl^r$ and $\jmath$; but tensoring with the flat $A$-module $U_\ract$ preserves finite limits and hence $U_\ract \otimes_\ahha \Hom^r(N,M)$ is the pullback for $\id_\uhhu \otimes_\ahha \gl^r$ and $\id_\uhhu \otimes_\ahha \jmath$. Then, from
    $
(\id_\uhhu \otimes_\ahha \gl^r)\gl^r f = (\gD \otimes_\ahha \id_\uhhu)\gl^r f
$
one observes $(\id_\uhhu \otimes_\ahha \gl^r)\gl^r f \in \mathrm{im}(\id_\uhhu \otimes_\ahha \jmath)$ and therefore $\gl^r f \in U_\ract \otimes_\ahha \HOM^r(N,M)$.
  \end{proof}

For the sake of simplicity, by slight abuse of notation, we will denote the coaction on $\HOM^r(N, M)$ by $\gl^r$ instead of $\jmath^{-1} \circ \gl^r$ when
$\gl^r f \in  \mathrm{im}(\jmath)$.

Observe that with respect to the $A$-bimodule structure \eqref{priamo}, we have by \eqref{taklinco} and the the right $U$-comodule version of \eqref{takcom2},
\begin{equation*}
  \label{teams18}
(\gl^r (a \lact f \bract b))(p) = a \lact f(p_{(0)})_{(-1)} p_{(1)} \bract b \otimes_\ahha f(p_{(0)})_{(0)},
\end{equation*}
as one rightly would expect from the property \eqref{taklinco} of a left $U$-coaction.

Now, if the left bialgebroid $(U,A)$ is right Hopf and $\due U \lact {}$ projective over $A$, using the monoidal functor $\ucomod \to \comodu$ mentioned in \S\ref{paco}, we can start from two left $U$-comodules $N$ and $M$
and transform the former into a right one as in Eq.~\eqref{soestrene}.  Repeating then an analogous discussion as above, we can define the left $U$-comodule
$$
 \HOM^r(N,M)  \coloneqq  \{ f \in \Hom_\Aopp(N,M) \mid \gl^r f \in  \mathrm{im}(\jmath) \},
$$
where
\begin{equation}
  \label{coaction1}
  (\gl^r f)(n) = f\big(\gve(n_{(-1)[+]}) n_{(0)}\big)_{(-1)} n_{(-1)[-]} \otimes_\ahha f\big(\gve(n_{(-1)[+]}) n_{(0)}\big)_{(0)}.
\end{equation}
However, instead of using the explicit expression \eqref{coaction1} in later intricate computations, for better readability it is more convenient to consider the left $U$-comodule $N$ as a right one as in Eq.~\eqref{soestrene} and to stick to the notation used there, that is, we will {\em always} write the left coaction \eqref{coaction1} on $\HOM^r(N,M)$
as
\begin{equation}
  \label{coaction2}
(\gl^r f)(n) = f(n_{[0]})_{(-1)} n_{[1]} \otimes_\ahha f(n_{[0]})_{(0)}.
\end{equation}
Lemma \ref{high-gain} then becomes: 

\begin{prop}
Let $(U,A)$ be a left bialgebroid such that $U_\ract$ and $\due U \lact {}$ are projective.
If $(U, A)$ in addition  is right Hopf and both $N, M$ are left $U$-comodules, then $\HOM^r(N,M)$ is a left $U$-comodule with left coaction induced by Eq.~\eqref{coaction1}.
\end{prop}

Observe that the projectivity of $U_\ract$ is needed to have $\jmath$ injective (and $\ucomod$ abelian) whereas the one of  $\due U \lact {}$ to guarantee well-definedness of Eq.~\eqref{soestrene}. We will refer to this situation henceforth as $U$ being {\em $A$-biprojective}. 

\subsubsection{Left internal Homs in $\ucomod$}

Let $(U,A)$ be a left bialgebroid and $N, M \in \ucomod$.
With respect to the canonical codiagonal left $U$-coaction
$$
M \otimes_\ahha \due U \lact {}
\to U_\ract \otimes_\ahha (M \otimes_\ahha \due U \lact {}), \quad m \otimes_\ahha u \mapsto m_{(-1)} u_{(1)} \otimes_\ahha (m_{(0)} \otimes_\ahha u_{(2)}),
$$
on $M \otimes_\ahha \due U \lact {}$,
consider the space
$
\Hom^\uhhu(N,M \otimes_\ahha \due U \lact {})
$
of left $U$-colinear maps: for each of its elements, we are going to deploy 
a sort of Sweedler notation with summation understood, that is, for any $g \in \Hom^\uhhu(N,M \otimes_\ahha \due U \lact {})$, write
$$
g'(n) \otimes_\ahha g''(n)  \coloneqq  g(n),
$$
and equip $ \Hom^\uhhu(N,M \otimes_\ahha \due U \lact {})$ with an $A$-bimodule structure by means of
\begin{equation}
\label{kohinoor}
(a \blact g \ract b)(n)  \coloneqq  g'(n) \otimes_\ahha a \blact g''(n) \ract b
\end{equation}
for all $a, b \in A$, see Eq.~\eqref{pergolesi} for notation.
If $(U,A)$ in addition is left Hopf, define 
\begin{equation}
    \label{teams}
\begin{array}{rcl}
  \gl^\ell \colon  \Hom^\uhhu(N,M \otimes_\ahha \due U \lact {})
  &\to& \Hom^\uhhu(N, U_\ract \otimes_\ahha \due {(M \otimes_\ahha \due U \lact {})} \blact {}\!),
  \\[2pt]
  g &\mapsto& \big\{
  n \mapsto
  g''(n)_- \otimes_\ahha \due {(g'(n) \otimes_\ahha g''(n)_+)} \blact {}
  \big\},
\end{array}
\end{equation}
which becomes well-defined if the first tensor factor relates to the third by means of multiplying with the target map from the right.
Again, the canonical map
$
\jmath \colon  U_\ract \otimes_\ahha  \Hom^\uhhu(N,M \otimes_\ahha \due U \lact {}) \to \Hom^\uhhu(N, U_\ract \otimes_\ahha (M \otimes_\ahha \due U \lact {}))
$
is injective if $U_\ract$ is $A$-projective, which allows us to define:

\begin{dfn}
  \label{leftcom}
  For two left $U$-comodules $N, M$ over a left bialgebroid $(U,A)$ which is left Hopf and with $U_\ract$ projective over $A$, the $A$-bimodule
  $$
\HOM^\ell(N,M) = \{g \in \Hom^\uhhu(N,M \otimes_\ahha \due U \lact {}) \mid \gl^\ell g \in  \mathrm{im}(\jmath) \}
  $$
is called the space of {\em (left) rational morphisms} from $N$ to $M$.
  \end{dfn}

In other words, $\HOM^\ell(N,M)$ consist of all $g \in  \Hom^\uhhu(N,M \otimes_\ahha \due U \lact {})$
for which there exists an element
$
g_{(-1)} \otimes_\ahha g_{(0)} \in U_\ract \otimes_\ahha  \Hom^\uhhu(N,M \otimes_\ahha \due U \lact {})
$
such that
$$
(\gl^\ell g)( n) = g_{(-1)} \otimes_\ahha g_{(0)}(n) 
$$
for all $n \in N$. Again, by injectivity of the canonical map $\jmath$, we may simply write
$$
\gl^\ell g  \coloneqq  g_{(-1)} \otimes_\ahha g_{(0)} 
$$
for any (left) rational $g$.
As before, if $U_\ract$ is finitely generated projective over $A$, then all morphisms in $ \Hom^\uhhu(N,M \otimes_\ahha \due U \lact {})$ are (left) rational.

\begin{lem}
  \label{high-gain2}
Let $(U,A)$ be a left Hopf algebroid over a left bialgebroid such that $U_\ract$ is projective, and $N, M \in \ucomod$.
Then $\HOM^\ell(N,M)$ is a left $U$-comodule as well, with coaction
given by $\jmath^{-1} \circ \gl^\ell$.
\end{lem}

\begin{proof}
  Here we argue exactly as in \S\ref{sahale} by which essentially the only aspect left to show is coassociativity (counitality being obvious from Eq.~\eqref{Sch8}), that is, for any $g \in \HOM^\ell(N,M)$, we have
  \begin{small}
\begin{eqnarray*}
\big((\id \otimes_\ahha \gl^\ell)\gl^\ell g\big)(n)
&
=
&
g_{(-1)} \otimes_\ahha (\gl^\ell g_{(0)})(n)
\\
  &
\overset{\scriptscriptstyle{\eqref{teams}}}{=} 
&
g_{(-1)} \otimes_\ahha \big(
g''_{(0)}(n)_{-} \otimes_\ahha (g'_{(0)}(n) \otimes_\ahha g''_{(0)}(n)_{+}) \big)
\\
  &
\overset{\scriptscriptstyle{\eqref{teams}}}{=} 
&
g''(n)_{-}
\otimes_\ahha 
g''(n)_{+-}
\otimes_\ahha (g'(n) \otimes_\ahha g''(n)_{++}) 
\\
  &
\overset{\scriptscriptstyle{\eqref{Sch5}}}{=} 
&
g''(n)_{-(1)}
\otimes_\ahha 
g''(n)_{-(2)}
\otimes_\ahha (g'(n) \otimes_\ahha g''(n)_{+}) 
\\
&=&
\big((\gD \otimes_\ahha \id)\gl^\ell g\big)(n),
\end{eqnarray*}
    \end{small} 
\!\!\!\!  which again implies $\gl^\ell g  \in  U_\ract  \otimes_\ahha  \HOM^\ell(N,M) $ as in the proof of Lemma~\ref{high-gain}.
  \end{proof}

As above, to lighten notation, we will write the coaction on $\HOM^\ell(N,M)$ simply as $\gl^\ell$ instead of $\jmath^{-1} \circ \gl^\ell$.

\begin{rem}
  \label{relajante1}
  The striking asymmetry in defining $\HOM^r$ and $\HOM^\ell$ and their coactions is due to the fact ({\em cf.} Remark \ref{curry}) that for a left bialgebroid one has a functor $\ucomod \to \comodu$ only in presence of a right Hopf structure but not in presence of a left one (which {\em notably} complicates matters in all what follows). Even worse, and in strong contrast to the case of $\umod$ in \S\ref{threeone}, where the left internal Homs did not require {\em any} Hopf structure at all, for defining a coaction on  $\HOM^\ell$ a left Hopf structure is sufficient but for its being left internal Homs, we additionally will assume the presence of a right Hopf structure as well, see the subsequent Lemma \ref{keineDokumentecomod}.
\end{rem}

\begin{rem}
  \label{castelnuovo}
  Nevertheless, in case $(U, A) = (H, k)$ is a Hopf algebra over a field $k$ with invertible antipode $S$, all these difficulties disappear and a short computation reveals that $\HOM^\ell(N,M)$ and  $\HOM^r(N,M)$ reproduce the well-known internal Hom functors (see, {\em e.g.}, \cite[Prop.~1.2]{CaeGue:OTCORHM}) which use the antipode and its inverse, {\em i.e.}, in both cases the $k$-module $\Hom_k(N, M)$ with left coactions
  \begin{eqnarray*}
  (\gl^\ell f)(n) \!\!&=\!\!& S(n_{(-1)}) f(n_{(0)})_{(-1)} \otimes_k f(n_{(0)})_{(0)},
 \\
  (\gl^r f)(n) \!\!&=\!\!& 
  f(n_{(0)})_{(-1)} S^{-1}(n_{(-1)}) \otimes_k f(n_{(0)})_{(0)},
\end{eqnarray*}
  respectively, for all $n \in N$. 
  \end{rem}

\begin{lem}
\label{keineDokumentecomod}
Let $(U,A)$ be a left bialgebroid which is biprojective over $A$.
\begin{enumerate}
\compactlist{99}
\item[({\it i})]
If $(U,A)$ is in addition  right Hopf, then the category $\ucomod$ of left $U$-comodules is right closed monoidal, {\em i.e.}, has right internal Hom functors.
  \item[({\it ii})]
  If the left bialgebroid $(U,A)$ is simultaneously left and right Hopf, $\ucomod$ is left closed monoidal as well, that is, has left internal Hom functors. As a consequence, in this case $\ucomod$ is biclosed monoidal, {\em i.e.}, has both left and right internal Hom functors.
\end{enumerate}
\end{lem}

\begin{proof}
  Let $M, N, P \in \ucomod$ be left $U$-comodules.

(i):
As the notation suggests, the right internal Homs are given by the $\HOM^r(N,M)$ from Definition \ref{rightcom}, where $N$ is seen as a right $U$-comodule via \eqref{soestrene},   equipped with the left $U$-coaction \eqref{coaction1} resp.~\eqref{coaction2}, along with the adjunction (iso)morphism
\begin{equation}
  \label{coad1}
  \begin{array}{rcl}
  \xi \colon  \Hom^\uhhu(P \otimes_\ahha N, M)
  &\to& \Hom^\uhhu(P, \HOM^r(N, M)),
  \\[2pt]
  f &\mapsto& \{p \mapsto f(p \otimes_\ahha -) \},
  \\[2pt]
\{ \tilde f(p)(n) \mapsfrom p \otimes_A n\} &\mapsfrom& \tilde f,
  \end{array}
  \end{equation}
  induced by the customary Hom-tensor adjunction.
  To see that $\xi f$ indeed lands in $\Hom^\uhhu(P, \HOM^r(N, M))$, we have to show that $(\xi f)(p) \in \Hom_\Aopp(N,M)$ is a (right) rational morphism from $N$ to $M$ and that $\xi f$ is left $U$-colinear with respect to the coactions of $P$ and $\HOM^r(N,M)$ in \eqref{coaction2}.
Both statements are shown in a single computation only: one has
   \begin{small}
\begin{eqnarray*}
\big(\gl^r \xi f(p)\big)(n) 
  &
\overset{\scriptscriptstyle{\eqref{coaction2}}}{=} 
&
  \big(\xi f(p)\big)(n_{[0]})_{(-1)} n_{[1]} \otimes_\ahha \big(\xi f(p)\big)(n_{[0]})_{(0)}
    \\
  &
\overset{\scriptscriptstyle{}}{=} 
&
  f(p \otimes_\ahha n_{[0]})_{(-1)} n_{[1]} \otimes_\ahha f(p \otimes_\ahha n_{[0]})_{(0)}
    \\
  &
\overset{\scriptscriptstyle{}}{=} 
&
p_{(-1)} n_{[0](-1)} n_{[1]} \otimes_\ahha f(p \otimes_\ahha n_{[0](0)})
    \\
  &
\overset{\scriptscriptstyle{\eqref{blumare1}}}{=} 
&
p_{(-1)} n_{(-1)[+]} n_{(-1)[-]} \otimes_\ahha f(p \otimes_\ahha n_{[0]})
\\
&
\overset{\scriptscriptstyle{\eqref{Tch7}, \eqref{takcom}}}{=} 
&
p_{(-1)} \otimes_\ahha \big(\xi f(p_{(0)})\big)(n), 
\end{eqnarray*}
   \end{small}
   where we used the $U$-colinearity of $f$ in the third step;
this not only shows that $\gl^r \xi f(p) \in U_\ract \otimes_\ahha \Hom_\Aopp(N,M)$
 and  hence $\xi f(p) \in \HOM^r(N,M)$ but simultaneously that $\xi f$ is left $U$-colinear as well. That the asserted inverse indeed inverts $\xi$ is obvious.

  (ii):
Here, in turn, as the notation again suggests, the left internal Homs are given by the $\HOM^\ell$ from Definition \ref{leftcom} equipped with the left coaction in Eq.~\eqref{teams}, along with the adjunction morphism
\begin{equation}
      \label{coad2}
  \begin{array}{rcl}
    \zeta \colon  \Hom^\uhhu(N \otimes_\ahha P, M) &\to& \Hom^\uhhu(P, \HOM^\ell(N,M)),
    \\
f &\mapsto& \{ p  \mapsto   f(- \otimes_\ahha p_{[0]}) \otimes_\ahha  p_{[1]} \},
\end{array}
\end{equation}
where $p \mapsto p_{[0]} \otimes_A p_{[1]}$ is the right $U$-coaction \eqref{soestrene} on the left $U$-comodule $P$.
  To verify that $\zeta f$ indeed lands in $\Hom^\uhhu(P, \HOM^\ell(N,M))$, we need to check that $\zeta f$ is $U$-colinear and also that $\zeta f(p) \in \HOM^\ell(N,M)$ for any $p \in P$, hence that $\zeta f(p)$ is a (left) rational morphism from $N$ to $M$ and as such left $U$-colinear again.
  As for the first issue,
  we compute by means of the codiagonal coaction on $M \otimes_\ahha U$:
\begin{small}
  \begin{eqnarray*}
    &&
\pig(\big(\zeta f(p)\big)(n)\pig)_{(-1)} \otimes_\ahha \pig(\big(\zeta f(p)\big)(n)\pig)_{(0)}
\\
&
\overset{\scriptscriptstyle{\eqref{coad2}}}{=} 
&
f(n \otimes_\ahha p_{[0]})_{(-1)}  p_{[1](1)} \otimes_\ahha \big(f(n \otimes_\ahha p_{[0]})_{(0)}  \otimes_\ahha p_{[1](2)}\big)
    \\
  &
\overset{\scriptscriptstyle{}}{=} 
&
n_{(-1)} p_{[0](-1)} p_{[1](1)} \otimes_\ahha \big(f(n_{(0)} \otimes_\ahha p_{[0](0)}) \otimes_\ahha  p_{[1](2)}\big)
    \\
  &
\overset{\scriptscriptstyle{\eqref{blumare1}}}{=} 
&
n_{(-1)} p_{(-1)[+]} p_{(-1)[-](1)} \otimes_\ahha \big(f(n_{(0)} \otimes_\ahha p_{(0)}) \otimes_\ahha  p_{(-1)[-](2)}\big)
    \\
  &
\overset{\scriptscriptstyle{\eqref{Tch5}, \eqref{Tch7}, \eqref{takcom}}}{=} 
&
n_{(-1)} \otimes_\ahha \big(\zeta f(p)\big)(n_{(0)}),
\end{eqnarray*}
\end{small}
where we used the colinearity of $f$ in the second step. Secondly,
   \begin{small}
\begin{eqnarray*}
\big(\gl^\ell \zeta f(p)\big)(n) 
  &
\overset{\scriptscriptstyle{\eqref{teams}}}{=} 
&
\big(\zeta f(p)\big)''(n)_- \otimes_\ahha \pig(\big(\zeta f(p)\big)'(n) \otimes_\ahha \big(\zeta f(p)\big)''(n)_+\pig) 
    \\
  &
\overset{\scriptscriptstyle{\eqref{coad2}}}{=} 
&
p_{[1]-} \otimes_\ahha \big(f(n \otimes_\ahha p_{[0]}) \otimes_\ahha p_{[1]+}\big) 
    \\
  &
\overset{\scriptscriptstyle{\eqref{blumare2}}}{=} 
&
p_{(-1)} \otimes_\ahha \big(f(n \otimes_\ahha p_{(0)[0]}) \otimes_\ahha p_{(0)[1]} \big)
    \\
  &
\overset{\scriptscriptstyle{\eqref{coad2}}}{=} 
&
p_{(-1)} \otimes_\ahha \big(\zeta f(p_{(0)})\big)(n), 
\end{eqnarray*}
   \end{small}
which not only shows that $\gl^\ell \zeta f(p) \in U_\ract \otimes_\ahha \Hom^\uhhu(N, M \otimes_\ahha \due U \lact {})$ for any $p \in P$ and hence $\zeta f(p) \in \HOM^\ell(N, M)$ but simultaneously also that $\zeta f$ is $U$-colinear in the desired sense, that is, $(\zeta f(p))_{(-1)} \otimes_\ahha (\zeta f(p))_{(0)} = p_{(-1)} \otimes_\ahha \zeta f(p_{(0)})$.

The inverse $\Hom^\uhhu\!(P, \HOM^\ell(N,M)) \to  \Hom^\uhhu\!(N \otimes_\ahha P, M)$ of $\zeta$ will be given by
\begin{equation}
  \label{coad2a}
(\zeta^{-1} g)(n \otimes_\ahha p) = (\id \otimes \gve) g(p)(n)
= g(p)'(n) \gve\big(g(p)''(n) \big).
 \end{equation}
In turn, to show that $\zeta^{-1} g$ is in fact a left $U$-colinear map from $N \otimes_\ahha P$ to $M$, observe first that $g \in  \Hom^\uhhu(P, \HOM^\ell(N,M))$ implies two identities, namely
   \begin{small}
\begin{eqnarray*}
  p_{(-1)} \otimes_\ahha \big(g(p_{(0)})'(n) \otimes_\ahha g(p_{(0)})''(n) \big)
  & \!\!\!\! = & \!\!\!\! 
  g(p)''(n)_- \otimes_\ahha \big(g(p)'(n) \otimes_\ahha g(p)''(n)_+ \big),
  \\
 n_{(-1)} \otimes_\ahha \big(g(p)'(n_{(0)}) \otimes_\ahha g(p)''(n_{(0)}) \big)
  & \!\!\!\!  = & \!\!\!\! 
  g(p)'(n)_{(-1)}  g(p)''(n)_{(1)}  \otimes_\ahha \big(g(p)'(n)_{(0)} \otimes_\ahha g(p)''(n)_{(2)} \big),  
\end{eqnarray*}
\end{small}
for any $n \in N$ and $p \in P$. With the help of these two equations, we proceed by
   \begin{small}
     \begin{eqnarray*}
       &&
       n_{(-1)}p_{(-1)} \otimes_\ahha (\zeta^{-1} g)(n_{(0)} \otimes_\ahha p_{(0)})
       \\
  &
\overset{\scriptscriptstyle{}}{=} 
&
 n_{(-1)}p_{(-1)} \otimes_\ahha g(p_{(0)})'(n_{(0)}) \gve\big(g(p_{(0)})''(n_{(0)}) \big)
    \\
  &
\overset{\scriptscriptstyle{}}{=} 
&
g(p_{(0)})'(n)_{(-1)}  g(p_{(0)})''(n)_{(1)}
p_{(-1)} \otimes_\ahha g(p_{(0)})'(n)_{(0)} \gve\big(g(p_{(0)})''(n)_{(2)} \big)
    \\
  &
\overset{\scriptscriptstyle{\eqref{takcom}}}{=} 
&
g(p_{(0)})'(n)_{(-1)}  g(p_{(0)})''(n)
p_{(-1)} \otimes_\ahha g(p_{(0)})'(n)_{(0)} 
    \\
  &
\overset{\scriptscriptstyle{}}{=} 
&
g(p)'(n)_{(-1)}  g(p)''(n)_+ g(p)''(n)_-
\otimes_\ahha g(p)'(n)_{(0)}
   \\
  &
\overset{\scriptscriptstyle{\eqref{Sch7}}}{=} 
&
g(p)'(n)_{(-1)} \bract \gve\big(g(p)''(n)\big)
\otimes_\ahha g(p)'(n)_{(0)}
   \\
  &
\overset{\scriptscriptstyle{\eqref{taklinco}}}{=} 
&
(\zeta^{-1} g)(n \otimes_\ahha p)_{(-1)} \otimes_\ahha  (\zeta^{-1} g)(n \otimes_\ahha p)_{(0)},
\end{eqnarray*}
   \end{small}
\!\!\! and therefore $\zeta^{-1} g \in  \Hom^\uhhu(N \otimes_\ahha P, M)$ as claimed.
Verifying that $\zeta^{-1}$ effective\-ly inverts $\zeta$ is shown by similar computations and is therefore skipped. The last statement is an obvious consequence of (i) and the statements just verified.
  \end{proof}

\begin{rem}
  \label{relajante2}
  One might wonder whether one could not, in the spirit of Remark \ref{hours} for the case of $\umod$, simply transport the left $U$-coaction \eqref{teams} to $\Hom_\ahha(N,M)$
  by means of the $k$-linear isomorphism
$$
\nu \colon  \Hom^\uhhu(N,M \otimes_\ahha U) \to \Hom_\ahha(N,M), \quad f \mapsto (\id \otimes_\ahha \gve) f
$$
(with inverse
$g \mapsto \{n \mapsto  g(n_{(0)})_{[0]} \otimes_\ahha g(n_{(0)})_{[1]} n_{(-1)} 
\}$),
so as to work with the seemingly easier $\Hom_\ahha(N,M)$ instead of $\Hom^\uhhu(N,M \otimes_\ahha U)$. However, this will not work since $\nu$ is not a morphism of $A$-bimodules when considering the $A$-bimodule structure \eqref{kohinoor}. Apparently, and in clear contrast to what was said in Remark \ref{hours}, the left internal Homs $\HOM^\ell(N,M) = \Hom^\uhhu(N,M \otimes_\ahha U)$ cannot be simplified, not even in presence of more structure, {\em cf.} also Remark \ref{relajante1}.
\end{rem}

\begin{notation}
Again, as the left and right internal Homs are quite different and it sometimes is convenient to remember the explicit $U$-colinearity or $A$-linearity in question, we shall not always use the sort of concealing notation $\HOM^r$ and $\HOM^\ell$ but often write $\Hom_\Aopp$ and $\Hom^\uhhu( -, - \otimes_\ahha U)$ even if the internal Homs with their left $U$-comodule structure are meant.
  \end{notation}

\subsection{$\ucomod$ as a bimodule category}
Again, the internal Homs allow to define the
structure of a bimodule category on $\ucomod$ with the help of the adjoint actions.
More precisely, Lemmata \ref{glenngould} \& \ref{keineDokumentecomod}
directly imply:

\begin{cor}
  \label{wasasesam2}
  Let $(U,A)$ be a left bialgebroid with $U$ biprojective over $A$.
  \begin{enumerate}
    \compactlist{99}
\item
  If $(U, A)$ is in addition  right Hopf,
  then the operation
  \begin{equation}
    \label{methamill1}
    \begin{array}{rcl}
      \ucomod^\op \times \ucomod &\to& \ucomod,
      \\
      (N,M) &\mapsto&
      N
\fren
      M  \coloneqq 
      \HOM^r(N, M)
\end{array}
    \end{equation}
gives $\ucomod$ the structure of a left module category over $\ucomod^\op$.
    If $(U, A)$ is both left and right Hopf,~then
   the operation
   \begin{equation*}
    \begin{array}{rcl}
      \ucomod \times \ucomod^\op &\to& \ucomod,
      \\
      (M,N) &\mapsto&
      N
\frenop
      M  \coloneqq 
      \HOM^\ell(N, M)
\end{array}
   \end{equation*}
     defines on $\ucomod$ the structure of a right module category over $\ucomod^\op$.
Hence, if the left bialgebroid $(U, A)$ is simultaneously left and right Hopf, then
     $\ucomod$ is a bimodule category over the monoidal category $\ucomod^\op$.
       \item
         The operation \eqref{methamill1} restricts to a left action
  $$
  \yd \times  \ayd
  \to  \ayd
$$
 if $\HOM^r(N,M)$ is seen as a right $U$-module by means of 
  the right $U$-action on $\Hom_\Aopp(N,M)$
defined by 
\begin{equation}
\label{giovedi}
(f \copact u)(n)  \coloneqq  f(u_{(1)} n)u_{(2)} 
\end{equation}
for $N \in \umod$ and $M \in \modu$. 
Hence, $\ayd$ is a left module category over the monoidal category $\yd$.
\end{enumerate}
\end{cor}

\begin{proof}
  (i): All statements in this part follow, as said, directly from 
Lemmata \ref{glenngould} \& \ref{keineDokumentecomod}.
Nevertheless, as in the module case, for later use and the sake of explicit illustration of the abstract theory, let us discuss the involved associative constraints.

  As for the left action, for $M, N, P \in \ucomod$, there is an isomorphism
  $P \fren
(N \fren
M) \simeq (P \otimes_\ahha N) \fren
M$
of left $U$-comodules, which amounts to the map
  \begin{equation}
    \label{leitz1}
 \phi_{\scriptscriptstyle P, N, M} \colon  \HOM^r(P \otimes_\ahha N, M)
  \to \HOM^r(P, \HOM^r(N, M))
  \end{equation}
given by the customary Hom-tensor adjunction. Let us see that this is, indeed, a map of left $U$-comodules. To start with, if $P \otimes_\ahha N$ is a left $U$-comodule with codiagonal coaction,
as one expects its induced right coaction \eqref{soestrene} is given by
\begin{equation}
\label{cosmicflux}
  (p \otimes_\ahha n)_{[0]} \otimes_\ahha (p \otimes_\ahha n)_{[1]}  \coloneqq  (p_{[0]} \otimes_\ahha n_{[0]}) \otimes_\ahha n_{[1]} p_{[1]}.
  \end{equation}
Abbreviating $\phi = \phi_{\scriptscriptstyle P,N,M}$, one then has  for any $p \in P$, $n \in N$: 
  \begin{small}
\begin{eqnarray*}
f_{(-1)} \otimes_\ahha (\phi f_{[0]})(p)(n)
   &
\overset{\scriptscriptstyle{}}{=} 
&
f_{(-1)} \otimes_\ahha f_{(0)}(p \otimes_\ahha n)
    \\
  &
\overset{\scriptscriptstyle{}}{=} 
&
f(p_{[0]} \otimes_\ahha n_{[0]})_{(-1)}
n_{[1]}p_{[1]}
\otimes_\ahha
f(p_{[0]} \otimes_\ahha n_{[0]})_{(0)}
    \\
  &
\overset{\scriptscriptstyle{}}{=} 
&
(\phi f)(p_{[0]})(n_{[0]})_{(-1)}
n_{[1]}p_{[1]}
\otimes_\ahha
(\phi f)(p_{[0]}) (n_{[0]})_{(0)}
\\
&
\overset{\scriptscriptstyle{}}{=} 
&
(\phi f)(p_{[0]})_{(-1)}
p_{[1]}
\otimes_\ahha
       (\phi f)(p_{[0]})_{(0)}(n)
 \\
&
\overset{\scriptscriptstyle{}}{=} 
&
\big((\phi f)_{(-1)}
\otimes_\ahha
       (\phi f)_{(0)}(p)\big)(n),      
\end{eqnarray*}
  \end{small}
  hence
  $
(\id \otimes_\ahha \phi) \gl^r f =
\gl^r(\phi f)
$,
as desired.


Discussing the associativity constraint for the right action is slightly more laborious as
the standard Hom-tensor adjunction is not the map that will induce the comodule isomorphism
$
M \frenop (N \otimes_\ahha P) \simeq (M \frenop N) \frenop P
$
in question. Observe first that
$$
M \frenop (N \otimes_\ahha P) = \HOM^\ell(N \otimes_\ahha P, M) = \Hom^\uhhu( N \otimes_\ahha P, M \otimes_\ahha \due U \lact {})
$$
on the level of $k$-modules, along with
$$
(M \frenop N) \frenop P = \HOM^\ell(P, \HOM^\ell(N,M)) =   \Hom^\uhhu(P, \Hom^\uhhu(N, M \otimes_\ahha \due U \lact {}) \otimes_\ahha \due U \lact {}),
$$
where $\Hom^\uhhu(N, M \otimes_\ahha \due U \lact {})$ is seen as an $A$-bimodule as in \eqref{kohinoor} and as a left $U$-comodule as in \eqref{teams}.
We then claim that the map
\begin{small}
  \begin{equation}
    \label{bademser}
    \begin{split}
      \psi_{\scriptscriptstyle M, N, P} 
       \colon 
      & \Hom^\uhhu( N \otimes_\ahha P, M \otimes_\ahha  \due U \lact {}) \to  \Hom^\uhhu(P, \Hom^\uhhu(N, M \otimes_\ahha  \due U \lact {}) \otimes_\ahha  \due U \lact {}),
  \\
& f  \mapsto \big\{ p \mapsto f'(- \otimes_\ahha p_{[0]}) \otimes_\ahha  f''(- \otimes_\ahha p_{[0]})_{(1)} p_{[1]} \otimes_\ahha   f''(- \otimes_\ahha p_{[0]})_{(2)} \big\},
\end{split}
\end{equation}
\end{small}
where we wrote $f(n \otimes_\ahha p) =: f'(n \otimes_\ahha p) \otimes_\ahha f''(n \otimes_\ahha p)$, 
is an isomorphism of left $U$-comodules.
Using the same kind of component-wise notation twice for elements in
$\Hom^\uhhu(P, \Hom^\uhhu(N, M \otimes_\ahha  \due U \lact {}) \otimes_\ahha \due U \lact {})$, and abbreviating $\psi = \psi_{\scriptscriptstyle M, N, P}$, this can be rewritten as
\begin{small}
  \begin{equation*}
    \begin{split}
      (\psi f)(p)(n) &= (\psi f)'(p)'(n) \otimes_\ahha (\psi f)'(p)''(n) \otimes_\ahha (\psi f)''(p)
      \\
&=
      f'(n \otimes_\ahha p_{[0]}) \otimes_\ahha  f''(n \otimes_\ahha p_{[0]})_{(1)} p_{[1]} \otimes_\ahha   f''(n \otimes_\ahha p_{[0]})_{(2)},
\end{split}
\end{equation*}
\end{small}
for all $n \in N$ and $p \in P$.

We have to show four things now:
that $(\psi f)(p) \in
\HOM^\ell(N,M) \otimes_\ahha \due U \lact {}$
for any $p \in P$ and any $f \in \HOM^\ell(N \otimes_\ahha P,M)$,
that $\psi f$ is $U$-colinear in the given sense, that $\psi$ is a morphism of left $U$-comodules, and finally that it is bijective.
As for the first issue, observe that from the left $U$-colinearity
\begin{small}
  \begin{equation}
    \label{subscribe21}
  \begin{split}
    &
   f'(n \otimes_\ahha p)_{(-1)} f''(n \otimes_\ahha p)_{(1)} \otimes_\ahha
 f'(n \otimes_\ahha p)_{(0)} \otimes_\ahha f''(n \otimes_\ahha p)_{(2)}  
\\
&    =
    n_{(-1)} p_{(-1)} \otimes_\ahha  f'(n_{(0)} \otimes_\ahha p_{(0)}) \otimes_\ahha
    f''(n_{(0)} \otimes_\ahha p_{(0)})
  \end{split}
\end{equation}
\end{small}
of an $f \in \Hom^\uhhu( N \otimes_\ahha P, M \otimes_\ahha U)$ follows with Eqs.~\eqref{blumare1}, \eqref{Tch7}, and \eqref{takcom}
that
\begin{small}
  \begin{equation}
    \label{subscribe2}
  \begin{split}
    &
f'(n \otimes_\ahha p_{[0]})_{(-1)} f''(n \otimes_\ahha p_{[0]})_{(1)} p_{[1]}
\\
&
\quad
\otimes_\ahha
 f'(n \otimes_\ahha p_{[0]})_{(0)} \otimes_\ahha f''(n \otimes_\ahha p_{[0]})_{(2)}p_{[2]}   \otimes_\ahha f''(n \otimes_\ahha p_{[0]})_{(3)}  
    \\
&    =
n_{(-1)} \otimes_\ahha  f'(n_{(0)} \otimes_\ahha p_{[0]}) \otimes_\ahha f''(n_{(0)} \otimes_\ahha p_{[0]})_{(1)} p_{[1]} \otimes_\ahha  f''(n_{(0)} \otimes_\ahha p_{[0]})_{(2)},
  \end{split}
\end{equation}
\end{small}
and therefore directly
  \begin{small}
\begin{eqnarray*}
  &&
  \gl^\ell\Big( (\psi f)'(p)'(n) \otimes_\ahha (\psi f)'(p)''(n)\Big) \otimes_\ahha (\psi f)''(p)
  \\
    &
 \overset{{\scriptscriptstyle{}}}{=} 
&
(\psi f)'(p)'(n)_{(-1)} (\psi f)'(p)''(n)_{(1)}
\otimes
(\psi f)'(p)'(n)_{(0)} \otimes_\ahha (\psi f)'(p)''(n)_{(2)}
\otimes_\ahha (\psi f)''(p)
    \\
  &
\overset{\scriptscriptstyle{\eqref{bademser}}}{=} 
&
f'(n \otimes_\ahha p_{[0]})_{(-1)} f''(n \otimes_\ahha p_{[0]})_{(1)} p_{[1]}
\\
& &
\quad
\otimes_\ahha
 f'(n \otimes_\ahha p_{[0]})_{(0)} \otimes_\ahha f''(n \otimes_\ahha p_{[0]})_{(2)}p_{[2]}   \otimes_\ahha f''(n \otimes_\ahha p_{[0]})_{(3)}  
    \\
  &
\overset{\scriptscriptstyle{\eqref{subscribe2}, \eqref{bademser}}}{=} 
&
n_{(-1)} \otimes_\ahha (\psi f)(p)(n_{(0)}),
\end{eqnarray*}
   \end{small}
  hence
  $(\psi f)(p) \in
\HOM^\ell(N,M) \otimes_\ahha \due U \lact {}$ for any $p \in P$,
as claimed.
The second issue above, {\em i.e.}, that $\psi f$ is $U$-colinear, is left to the reader. More interesting, 
 \begin{small}
\begin{eqnarray*}
  &&
(\psi f)_{(-1)} \otimes_\ahha (\psi f)_{(0)}(p)(n)
  \\
    &
 \overset{{\scriptscriptstyle{\eqref{teams}}}}{=} 
&
 (\psi f)''(p)_- \otimes_\ahha (\psi f)'(p)'(n) \otimes_\ahha (\psi f)'(p)''(n)
 \otimes_\ahha (\psi f)''(p)_+
    \\
  &
\overset{\scriptscriptstyle{\eqref{bademser}}}{=} 
&
f''(n \otimes_\ahha p_{[0]})_{(2)-} \otimes_\ahha f'(n \otimes_\ahha p_{[0]})
\otimes_\ahha
f''(n \otimes_\ahha p_{[0]})_{(1)} p_{[1]}   \otimes_\ahha f''(n \otimes_\ahha p_{[0]})_{(2)+}  
    \\
  &
\overset{\scriptscriptstyle{\eqref{Sch4}}}{=} 
&
f''(n \otimes_\ahha p_{[0]})_{-} \otimes_\ahha f'(n \otimes_\ahha p_{[0]})
\otimes_\ahha
f''(n \otimes_\ahha p_{[0]})_{+(1)} p_{[1]}   \otimes_\ahha f''(n \otimes_\ahha p_{[0]})_{+(2)}
    \\
  &
\overset{\scriptscriptstyle{\eqref{teams}}}{=} 
&
f_{(-1)} \otimes_\ahha
f'_{(0)}(n \otimes_\ahha p_{[0]}) \otimes_\ahha
f''_{(0)}(n \otimes_\ahha p_{[0]})_{(1)}p_{[1]} \otimes_\ahha
f'_{(0)}(n \otimes_\ahha p_{[0]})_{(2)}
    \\
  &
\overset{\scriptscriptstyle{\eqref{bademser}}}{=} 
&
f_{(-1)} \otimes_\ahha (\psi f_{(0)})(p)(n),
\end{eqnarray*}
   \end{small}
 hence $\psi$ is in fact a left $U$-comodule map, which proves the third issue mentioned above.
 Finally, we claim that $\psi$ is bijective, the inverse being given by
\begin{small}
  \begin{equation*}
\begin{array}{rcl}
    \psi^{-1} \colon  
\Hom^\uhhu(P, \Hom^\uhhu(N, M \otimes_\ahha U) \otimes_\ahha U)
       &\to&
      \Hom^\uhhu( N \otimes_\ahha P, M \otimes_\ahha U),     
  \\
   g  &\mapsto& \big\{ n \otimes_\ahha p \mapsto
  (\id \otimes_\ahha \gve \otimes_\ahha \id)g(p)(n)
  \big\},
\end{array}
\end{equation*}
\end{small}
or, explicitly,
\begin{equation}
  \label{letatcestmoi}
(\psi^{-1} g)(n \otimes_\ahha p)  \coloneqq  g'(p)'(n)
  \otimes_\ahha \gve\big(g'(p)''(n) \big) \lact g''(p).
  \end{equation}
While $\psi^{-1} \circ \psi = \id$ follows directly from the counitality of the coproduct, that $\psi \circ \psi^{-1}$ yields the identity is slightly more laborious: the left $U$-colinearity of
$g \in \Hom^\uhhu(P, \Hom^\uhhu(N, M \otimes_\ahha U) \otimes_\ahha U)$ explicitly reads
\begin{small}
  \begin{equation*}
    \begin{split}
      p_{(-1)} \otimes_\ahha g(p_{(0)})(n)
      &=
      g(p)_{(-1)} \otimes_\ahha g(p)_{(0)}(n)
      \\
       &=
      g'(p)_{(-1)} g''(p)_{(1)} \otimes_\ahha 
      g'(p)_{(0)}'(n) \otimes_\ahha g'(p)_{(0)}''(n)
      \otimes_\ahha g''(p)_{(2)}
      \\
      &=
      g'(p)''(n)_-
      g''(p)_{(1)} \otimes_\ahha 
      g'(p)'(n)
      \otimes_\ahha g'(p)''(n)_+
      \otimes_\ahha
      g''(p)_{(2)},
    \end{split}
\end{equation*}
\end{small}
and therefore with Eqs.~\eqref{blumare1} and \eqref{Tch7}
\begin{small}
  \begin{equation}
    \label{traritrara}
    \begin{split}
      &      1 \otimes_\ahha
      g'(p)'(n)
      \otimes_\ahha g'(p)''(n)
      \otimes_\ahha g''(p)
      \\
&= p_{[0](-1)} p_{[1]} \otimes_\ahha g'(p_{[0](0)})'(n)
      \otimes_\ahha g'(p_{[0](0)})''(n)
      \otimes_\ahha g''(p_{[0](0)})
\\
     &=
      g'(p_{[0]})''(n)_-
      g''(p_{[0]})_{(1)}p_{[1]} \otimes_\ahha 
      g'(p_{[0]})'(n)
      \otimes_\ahha g'(p_{[0]})''(n)_+
      \otimes_\ahha
      g''(p_{[0]})_{(2)}.
    \end{split}
\end{equation}
\end{small}
With this,
 \begin{small}
\begin{eqnarray*}
  &&
  (\psi \psi^{-1} g)(p)(n)
    \\
    &
 \overset{{\scriptscriptstyle{\eqref{bademser}}}}{=} 
&
 (\psi^{-1} g)'(n \otimes_\ahha p_{[0]}) \otimes_\ahha
 (\psi^{-1} g)''(n \otimes_\ahha p_{[0]})_{(1)} p_{[1]}
 \otimes_\ahha
  (\psi^{-1} g)''(n \otimes_\ahha p_{[0]})_{(2)} 
    \\
  &
\overset{\scriptscriptstyle{\eqref{letatcestmoi}}}{=} 
&
g'(p_{[0]})'(n) \otimes_\ahha
 \gve\big(g'(p_{[0]})''(n) \big) \lact g''(p_{[0]})_{(1)} p_{[1]}
 \otimes_\ahha
 g''(p_{[0]})_{(2)}
    \\
  &
\overset{\scriptscriptstyle{\eqref{Sch7}}}{=} 
&
g'(p_{[0]})'(n) \otimes_\ahha
g'(p_{[0]})''(n)_+ g'(p_{[0]})''(n)_-
g''(p_{[0]})_{(1)} p_{[1]}
 \otimes_\ahha
 g''(p_{[0]})_{(2)}
    \\
  &
\overset{\scriptscriptstyle{\eqref{traritrara}}}{=} 
&
g'(p)'(n) \otimes_\ahha
g'(p)''(n) \otimes_\ahha
g''(p)
= g(p)(n),
\end{eqnarray*}
 \end{small}
 as desired.
 

 To conclude the discussion of the associators,
let us consider the {\em middle associativity constraint} required in Definition \ref{ratagnan}, that is, 
the
 isomorphism of left $U$-comodules
 $$
 \gvt_{\scriptscriptstyle P, M, N} \colon 
 ( P \fren M) \frenop N 
 \overset{\simeq}{\lra}
P \fren (M \frenop N)
 $$
for any $M, N, P \in \ucomod$,
which amounts to a left $U$-comodule isomorphism
$
\HOM^\ell(N, \HOM^r(P,M)) \simeq \HOM^r(P, \HOM^\ell(N,M)). 
$
To start with, define the $k$-module isomorphism
\begin{small}
  \begin{equation*}
\begin{split}
    \gvt_{\scriptscriptstyle \scriptscriptstyle P, M, N} \colon  
    \Hom^\uhhu(N , \Hom_\Aopp(P, M) \otimes_\ahha U ) &\to
        \Hom_\Aopp(P, \Hom^\uhhu (N, M \otimes_\ahha U)) 
\end{split}
\end{equation*}
\end{small}
  given by
  \begin{small}
          \begin{equation}
\begin{split}
    \label{leitz2}
    (\gvt_{\scriptscriptstyle P, M, N} f)(p)(n) &= (\gvt_{\scriptscriptstyle P, M, N} f)(p)'(n) \otimes_\ahha (\gvt_{\scriptscriptstyle P, M, N} f)(p)''(n)
 \\
         &= f'(n)(p_{[0]}) \otimes_\ahha p_{[1]} f''(n).
    %
\end{split}
\end{equation}
\end{small}
  Its inverse will be 
defined as
 \begin{small}
          \begin{equation}
\begin{split}
    \label{leitz2a}
    (\gvt^{-1}_{\scriptscriptstyle P, M, N} g)'(n)(p)
\otimes_\ahha (\gvt^{-1}_{\scriptscriptstyle P, M, N} g)''(n)
&=
g(p_{[0]})'(n) \gve(p_{[1][+]}) \otimes_\ahha p_{[1][-]} g(p_{[0]})''(n),
\end{split}
\end{equation}
\end{small}  
 the well-definedness of which over the Sweedler presentation of the right Hopf structure ({\em i.e.}, that is does not depend on the choice of a representative for the formal expression $p_{[0]} \otimes_\ahha p_{[1][+]} \otimes_\ahha p_{[1][-]}$) is not immediately visible to the naked eye but follows from a detailed consideration not unlikely the proof of the well-definedness of the coaction \eqref{soestrene} in  \cite[Thm.~4.1.1]{CheGavKow:DFOLHA} from the property $\gve(u \bract a) = \gve(a \blact u)$ of a bialgebroid counit, along with Eqs.~\eqref{Tch9}, \eqref{taklinco2}, and the right $A$-module structure on $\Hom_\Aopp(P, M)$ as in \eqref{priamo}, which implies that the tensor product $\Hom_\Aopp(P, M) \otimes_\ahha U$ is to be understood with respect to the ideal generated by $g(a\sma\cdot ) \otimes u - g\sma\cdot  \otimes a \lact u$ for $a \in A$ and $g \in  \Hom_\Aopp(P, M)$.

 That the two given maps in \eqref{leitz2} and \eqref{leitz2a} are mutual inverses follows more or less immediately from Eqs.~\eqref{Tch3}, \eqref{Tch4}, and \eqref{Tch7}.

 Next, let us verify that $\gvt$ is in fact a map (and hence an isomorphism) of left $U$-comodules. Abbreviating again
$\gvt = \gvt_{\scriptscriptstyle P, M, N}$ for better readability,  
one has by Eqs.~\eqref{teams}, \eqref{coaction2}, and \eqref{Sch3} for all $p \in P$ and $n \in N$:
\begin{equation*}
\begin{split}
& (\gvt f)_{(-1)} \otimes_\ahha (\gvt f)_{(0)}(p)(n) 
  \\
  & = (\gvt f)(p_{[0]})_{(-1)} p_{[1]} \otimes_\ahha (\gvt f)(p_{[0]})_{(0)}(n)
    \\
    & = \big((\gvt f)(p_{[0]})''(n)\big)_- \, p_{[1]} \otimes_\ahha
    (\gvt f)(p_{[0]})'(n) \otimes_\ahha \big((\gvt f)(p_{[0]})''(n)\big)_+
    \\
    & = f''(n)_- p_{[1](1)-} p_{[1](2)}
    \otimes_\ahha
    f'(n)(p_{[0]}) \otimes_\ahha p_{[1](1)+} f''(n)_+
     \\
    & = f''(n)_- 
    \otimes_\ahha
    f'(n)(p_{[0]}) \otimes_\ahha p_{[1]} f''(n)_+
    \\
    & = f_{(-1)} 
    \otimes_\ahha
    f'_{(0)}(n)(p_{[0]}) \otimes_\ahha p_{[1]} f''_{(0)}(n)
    \\
    & =
    f_{(-1)} \otimes_\ahha (\gvt f_{(0)})(p)(n), 
\end{split}
\end{equation*}
for any $f \in \HOM^\ell(N, \HOM^r(P, M)) \subset \Hom^\uhhu(N, \Hom_\Aopp(P,M) \otimes_\ahha U)$ and therefore 
$
\gl^\ell \circ \gvt
=
(\id_\uhhu \otimes_\ahha \gvt)
\circ \gl^r,
$
as claimed.

(ii): Finally, let $N \in \yd$ and $M \in \ayd$. We have to show that in this case  $\HOM^r(N,M)$ is an aYD module as well with respect to the right action \eqref{giovedi} and the left coaction \eqref{coaction2}. Indeed, for
any $f \in \HOM^r(N,M)$, one has
  \begin{small}
\begin{eqnarray*}
\gl^r(f \copact u)(n)
  &
\overset{\scriptscriptstyle{\eqref{coaction2}}}{=} 
&
(f \copact u)(n_{[0]})_{(-1)} n_{[1]} \otimes_\ahha (f \copact u)(n_{[0]})_{(0)}
    \\
  &
\overset{\scriptscriptstyle{\eqref{giovedi}}}{=} 
&
\big(f(u_{(1)} n_{[0]})u_{(2)}\big)_{(-1)} n_{[1]} \otimes_\ahha \big(f(u_{(1)} n_{[0]})u_{(2)}\big)_{(0)}
    \\
  &
\overset{\scriptscriptstyle{\eqref{funkydesertbreaks1}, \eqref{Sch4}}}{=} 
&
u_- f(u_{+(1)} n_{[0]})_{(-1)} u_{+(2)} n_{[1]} \otimes_\ahha f(u_{+(1)} n_{[0]})_{(0)} u_{+(3)}
    \\
  &
\overset{\scriptscriptstyle{\eqref{Tch3}}}{=} 
&
u_- f\big((u_{+(2)} n)_{[0]}\big)_{(-1)}
(u_{+(2)} n)_{[1]} u_{+(1)} \otimes_\ahha f\big((u_{+(2)} n)_{[0]}\big)_{(0)}
u_{+(3)}
\\
&
\overset{\scriptscriptstyle{\eqref{coaction2}}}{=} 
&
u_- f_{(-1)} u_{+(1)}
\otimes_\ahha f_{(0)}(u_{+(2)} n)
u_{+(3)}
\\
&
\overset{\scriptscriptstyle{\eqref{giovedi}}}{=} 
&
u_- f_{(-1)} u_{+(1)}
\otimes_\ahha (f_{(0)} \copact u_{+(2)})(n)
\end{eqnarray*}
   \end{small}
\!\!\!  for $n \in N$, $u \in U$, where in the third step we used the fact that $M \in \ayd$ and 
  that $N \in \yd$ in the fourth (see \cite[Def.~4.2]{Schau:DADOQGHA}).
This concludes the~proof.
  \end{proof}

\subsection{The bimodule centre in the {bialgebroid~comodule~category}}

We can now, thanks to Corollary \ref{wasasesam2}, examine the centre
$\cZ_{\ucomod^\op}(\ucomod)$
of the bimodule category $\ucomod$ with respect to its adjoint actions
given by all pairs $(M, \tau)$ of objects $M \in \ucomod$ for which there is a family of isomorphisms
$$
\tau_\enne \colon 
N \fren M
\stackrel{\simeq}{\lra}
N \frenop M
$$
of left $U$-comodules natural in $N \in \ucomod$. With respect to
this centre and
its full subcategory
$\cZ'_{\ucomod^\op}(\ucomod)$ of stable objects as in Definition \ref{stablepable},
we
can state the following result:

\begin{theorem}
  \label{tegel2}
Let an $A$-biprojective left bialgebroid $(U, A)$ be both left and right Hopf.
\begin{enumerate}
  \compactlist{99}
  \item
    Then any anti Yetter-Drinfel'd module $M$ induces a central structure
  \begin{equation*}
    \tau_\enne \colon 
    \HOM^r(N,M)
    \to
    \HOM^\ell(N,M),
    \end{equation*}
explicitly given on the level of $k$-modules by
    \begin{small}
    \begin{equation}
  \label{michigan2}  
      \begin{array}{rcl}
\Hom_\Aopp(N, M)
    &\!\!\!\to&\!\!\!
        \Hom^\uhhu(N, M \otimes_\ahha U),
  \\[2pt]
g &\!\!\!\mapsto&\!\!\!    \big\{n \mapsto \big(g(n_{[0]})_{[0]} \otimes_\ahha g(n_{[0]})_{[1]}\big) \mpsqact n_{[1]}\big\},
  \\[2pt]
  \big\{
f'(n)_{(0)}f'(n)_{(-1)}f''(n) 
\mapsfrom n \big\}
&\!\!\!\mapsfrom&\!\!\! f,
      \end{array}
    \end{equation}
    \end{small}
    where the right $U$-action $\!\!\mpsqact\!$ is the one defined in \eqref{wszw1}.
   \item
     Vice versa, for a pair $(M, \tau)$ in the centre $\cZ_{\ucomod^\op}(\ucomod)$,
defining
     \begin{equation}
        \label{ghettokaisers1}
mu  \coloneqq  (\tau^{-1}_\uhhu f_m)(u), \qquad \forall \ u \in U,
\end{equation}
     where $f_m \in \Hom^\uhhu(U, M \otimes_\ahha U)$ is given by $f_m(u) = m_{[0]} \otimes_\ahha m_{[1]} u$ for any $m \in M$, induces
a right $U$-action on $M$ and, in particular,
     the structure of an anti Yetter-Drinfel'd module on $M$.  
\item
Both preceding parts together amount to an equivalence
$$
\ayd \simeq  \cZ_{\ucomod^\op}(\ucomod) 
  $$
  of categories.
\item
  Imposing stability 
on the respective objects
  leads to an equivalence
$$
  \sayd \simeq \cZ'_{\ucomod^\op}(\ucomod)
  $$
of subcategories.
\end{enumerate}
\end{theorem}

\begin{rem}
  Using that any $f \in \Hom^\uhhu(N, M \otimes_\ahha U)$ is colinear, we can rewrite
     \begin{equation}
  \label{michigan1}  
  \tau^{-1}_\enne f(n) = \big(f'(n_{(0)}) \gve(f''(n_{(0)}) ) \big) n_{(-1)}
     \end{equation}
for the inverse of the central structure     instead of \eqref{michigan2},
which is sometimes more convenient to work with.
\end{rem}

\begin{proof}[Proof of Theorem \ref{tegel2}]
  (i):  We leave it to the reader to check that the two given maps in \eqref{michigan2} are well-defined (checking that $\tau_\enne g$ is so is somewhat laborious but very similar to the computations that follow below).
  That they are mutual inverses is in one direction almost immediate, whereas
  \begin{small}
\begin{eqnarray*}
 &&  \tau_\enne (\tau^{-1}_\enne f)(n)
\\
&\overset{\scriptscriptstyle{\eqref{michigan2}}}{=}&
\big((\tau^{-1}_\enne f)(n_{[0]})_{[0]}
  \otimes_\ahha  (\tau^{-1}_\enne f)(n_{[0]})_{[1]} \big) \mpsqact n_{[1]}
   \\
  &
\overset{\scriptscriptstyle{\eqref{michigan1}}}{=} 
&
\Big(\big(f'(n_{[0](0)}) \gve(f''(n_{[0](0)})) \lact n_{[0](-1)}   \big)_{[0]}
\otimes_\ahha  \big(f'(n_{[0](0)}) \gve(f''(n_{[0](0)})) \lact n_{[0](-1)}   \big)_{[1]}
\Big) \mpsqact n_{[1]}
    \\
  &
\overset{\scriptscriptstyle{\eqref{blumare1}, \eqref{funkydesertbreaks2}, \eqref{Tch5}}}{=} 
&
\Big(
f'(n_{(0)})_{[0]}
n_{(-1)[+](1)} 
\otimes_\ahha n_{(-1)[-](1)} 
t \gve(f''(n_{(0)}))
f'(n_{(0)})_{[1]}
n_{(-1)[+](2)}   
\Big) \mpsqact n_{(-1)[-](2)}
    \\
  &
\overset{\scriptscriptstyle{\eqref{wszw1}, \eqref{Tch2}, \eqref{Tch3}}}{=} 
&
f'(n_{(0)})_{[0]}
\otimes_\ahha 
t \gve(f''(n_{(0)}))
f'(n_{(0)})_{[1]}
n_{(-1)}   
    \\
  &
\overset{\scriptscriptstyle{}}{=} 
&
f'(n)_{(0)[0]}
\otimes_\ahha 
t \gve(f''(n)_{(2)})
f'(n)_{(0)[1]}
f'(n)_{(-1)}
f''(n)_{(1)}
\\
&
\overset{\scriptscriptstyle{\eqref{blumare2}, \eqref{Sch7}}}{=} 
&
f'(n)_{[0]}
\otimes_\ahha 
t \gve(f''(n)_{(2)})
s \gve(f'(n)_{[1]})
f''(n)_{(1)}
\\
&
\overset{\scriptscriptstyle{}}{=} 
&
f'(n) \otimes_\ahha f''(n)
=
f(n)
\end{eqnarray*}
  \end{small}
  for any $n \in N$
  proves the other direction, using left $U$-colinearity of $f$ in the fifth step and the aYD condition \eqref{funkydesertbreaks2} in the third, plus the fact that all four $A$-actions on $U$ as defined in \eqref{pergolesi} commute.

  Next, let us check that
  $\tau_\enne$, or rather its inverse, is natural in $N$: for any left $U$-comodule morphism $\gs \colon  N \to N'$ we want to see that
$\tau^{-1}_{\scriptscriptstyle{N}} \circ \HOM^r(\gs, M) = \HOM^\ell(\gs, M) \circ \tau^{-1}_{\enne'}$.
  Indeed, by left $U$-colinearity of $\gs$, 
  \begin{equation}
    \label{lebkuchen}
\begin{split}
    \tau^{-1}_\enne (f \circ \gs)(n)
    &=  \big(f'(\gs(n_{(0)})) \gve(f''(\gs(n_{(0)})) ) \big) n_{(-1)}
    \\
  &=  \big(f'(\gs(n)_{(0)})) \gve(f''(\gs(n)_{(0)})) ) \big) \gs(n)_{(-1)}
  = (\tau^{-1}_{\scriptscriptstyle N'} f) (\gs(n)), 
\end{split}
\end{equation}
for any $f \in \HOM^\ell(N', M)$, hence the claim.

Furthermore, we need to prove that  $\tau_\enne$ (we prefer $\tau^{-1}_\enne$) is itself a left $U$-comodule morphism, that is, 
  $
\gl^r \tau^{-1}_\enne = (\id \otimes \tau^{-1}_\enne)\gl^\ell.
  $
As a matter of fact, one has for any $f \in \HOM^\ell(N,M)$:
\begin{small}
\begin{eqnarray*}
 &&  (\gl^r \tau^{-1}_\enne f)(n)
\\
  &\overset{\scriptscriptstyle{\eqref{coaction2}}}{=}& (\tau^{-1}_\enne f)(n_{[0]})_{(-1)} n_{[1]}
  \otimes_\ahha  (\tau^{-1}_\enne f)(n_{[0]})_{(0)} 
   \\
  &
\overset{\scriptscriptstyle{\eqref{michigan1}}}{=} 
&
\Big(f'(n_{[0](0)}) \gve\big( f''(n_{[0](0)}) \big) n_{[0](-1)} \Big)_{(-1)} n_{[1]}
  \otimes_\ahha  \Big(f'(n_{[0](0)}) \gve\big( f''(n_{[0](0)}) \big) n_{[0](-1)} \Big)_{(0)}
    \\
  &
\overset{\scriptscriptstyle{\eqref{blumare1}, \eqref{funkydesertbreaks1}, \eqref{taklinco}}}{=} 
&
n_{(-1)[+]-}
f'(n_{(0)})_{(-1)} \Big(\gve\big( f''(n_{(0)}) \big) \lact n_{(-1)[+]+(1)} n_{(-1)[-]} \Big)
\otimes_\ahha  f'(n_{(0)})_{(0)} n_{(-1)[+]+(2)} 
    \\
  &
\overset{\scriptscriptstyle{\eqref{mampf1}, \eqref{Tch2}}}{=} 
&
n_{(-1)-}
f'(n_{(0)})_{(-1)} \bract \gve\big( f''(n_{(0)}) \big) \otimes_\ahha  f'(n_{(0)})_{(0)} n_{(-1)+} 
    \\
  &
\overset{\scriptscriptstyle{}}{=} 
&
f''(n)_{(1)-} f'(n)_{(-2)-} f'(n)_{(-1)}
\bract \gve\big( f''(n)_{(2)} \big) \otimes_\ahha  f'(n)_{(0)} f'(n)_{(-2)+} f''(n)_{(1)+} 
\\
&
\overset{\scriptscriptstyle{\eqref{Sch3}, \eqref{Sch9}}}{=} 
&
f''(n)_{-} \otimes_\ahha  f'(n)_{(0)} f'(n)_{(-1)} f''(n)_{+} 
\\
&
\overset{\scriptscriptstyle{\eqref{Sch4}, \eqref{takcom}}}{=} 
&
f''(n)_{(2)-} \otimes_\ahha  \big(f'(n)_{(0)} \gve(f''(n)_{(2)+}) \big) f'(n)_{(-1)} f''(n)_{(1)} 
\\
&
\overset{\scriptscriptstyle{}}{=} 
&
f''(n_{(0)})_{-} \otimes_\ahha  \big(f'(n_{(0)}) \gve(f''(n_{(0)})_{+}) \big) n_{(-1)}
\\
&
\overset{\scriptscriptstyle{\eqref{teams}}}{=} 
&
f_{(-1)} \otimes_\ahha  \big(f'_{(0)}(n_{(0)}) \gve(f''_{(0)}(n_{(0)})) \big) n_{(-1)}
\\
&
\overset{\scriptscriptstyle{\eqref{michigan1}}}{=} 
&
f_{(-1)} \otimes_\ahha  \tau^{-1}_\enne f_{(0)}(n)
\\
&
\overset{\scriptscriptstyle{}}{=} 
&
(\id \otimes_\ahha  \tau^{-1}_\enne) \gl^\ell f (n),
\end{eqnarray*}
  \end{small}
as desired, where we used the left $U$-colinearity of $f$ in the fifth and in the eighth step again, along with the aYD condition \eqref{funkydesertbreaks1} in the third.

We still need to prove the hexagon axiom \eqref{tarrega7}. For better readability, let us write down what this means on the level of $k$-modules:
\begin{small}
\begin{equation}
\label{tarrega8}
\xymatrix@C=.37cm{
  \Hom_\Aopp(P, \Hom^\uhhu(N, M \otimes_\ahha U))
  \ar@{<-}[rr]_{\gvt_{\scriptscriptstyle P,M, N}}
\ar@{<-}[d]_{\Hom_\Aopp(P, \tau_\enne) \, }
& &
\Hom^\uhhu(N, \Hom_\Aopp(P,M) \otimes_\ahha U)
\ar[d]^{\, \Hom^\uhhu(N, \tau_\pehhe \, \otimes_\ahha U)}
&
  \\
  \Hom_\Aopp(P, \Hom_\Aopp(N, M))
  \ar@{<-}[d]_{\phi_{\scriptscriptstyle P, N, M}}
  & &
  \Hom^\uhhu(N, \Hom^\uhhu(P,M \otimes_\ahha U) \otimes_\ahha U)
  \ar@{<-}[d]^{\quad \psi_{\scriptscriptstyle M, P, N}}
  \\
  \Hom_\Aopp(P \otimes_\ahha N, M)
  \ar[rr]_{\tau_{P \otimes N}}
  & &
    \Hom^\uhhu(P \otimes_\ahha N, M \otimes_\ahha U)
  }
\end{equation}
\end{small}
Verifying that the above diagram \eqref{tarrega8} commutes with respect to the central structure \eqref{michigan2} is essentially straightforward: by abuse of notation, let us abbreviate 
$\gvt = \gvt_{\scriptscriptstyle P,N,M}$ and likewise for $\phi$ and $\psi$,
along with
$\tau_\enne = \Hom_\Aopp(P, \tau_\enne)$, and $\tau_\pehhe = \Hom^\uhhu(N, \tau_\pehhe \otimes_\ahha U)$. We then have for
$f \in \Hom_\Aopp(P \otimes_\ahha N, M)$:
\begin{small}
\begin{eqnarray*}
  &&
  (\psi^{-1} \circ \tau_\pehhe  \circ \gvt^{-1} \circ \tau_\enne \circ \phi \circ f)(p \otimes_\ahha n)
\\
&\overset{\scriptscriptstyle{\eqref{bademser}}}{=}&
(\tau_\pehhe  \circ \gvt^{-1} \circ \tau_\enne \circ \phi \circ f)'(n)'(p)
\\
&& \qquad \qquad 
\otimes_\ahha
\gve\big((\tau_\pehhe  \circ \gvt^{-1} \circ \tau_\enne \circ \phi \circ f)'(n)''(p)\big) \lact
  (\tau_\pehhe  \circ \gvt^{-1} \circ \tau_\enne \circ \phi \circ f)''(n)
   \\
  &
\overset{\scriptscriptstyle{\eqref{michigan2}, \eqref{wszw1}}}{=} 
&
(\gvt^{-1} \circ \tau_\enne \circ \phi \circ f)'(n)
(p_{[0]})_{[0]} p_{[1][+]}
  \bract \gve\big(p_{[1][-]}(\gvt^{-1} \circ \tau_\enne \circ \phi \circ
f)'(n)(p_{[0]})_{[1]}\big) 
\\
 && \qquad \qquad 
 \otimes_\ahha
 (\gvt^{-1} \circ \tau_\enne \circ \phi \circ f)''(n)
    \\
  &
\overset{\scriptscriptstyle{\eqref{Tch8}}}{=} 
&
(\gvt^{-1} \circ \tau_\enne \circ \phi \circ f)'(n)
(p_{[0]}) p_{[1]}
 \otimes_\ahha
 (\gvt^{-1} \circ \tau_\enne \circ \phi \circ f)''(n)
    \\
  &
\overset{\scriptscriptstyle{\eqref{leitz2a}}}{=} 
&
(\tau_\enne \circ \phi \circ f)(p_{[0]})'(n) \gve(p_{[1][+]}) p_{[2]}
 \otimes_\ahha
p_{[1][-]} (\tau_\enne \circ \phi \circ f)(p_{[0]})''(n)
    \\
  &
\overset{\scriptscriptstyle{\eqref{Tch4}, \eqref{wszw1}}}{=} 
&
\pig((\tau_\enne \circ \phi \circ f)(p_{[0]})'(n) \otimes_\ahha
(\tau_\enne \circ \phi \circ f)(p_{[0]})''(n) \pig) \mpsqact p_{[1]}
\\
&
\overset{\scriptscriptstyle{\eqref{michigan2}}}{=} 
&
\pig((\phi \circ f)(p_{[0]})(n_{[0]})_{[0]} \otimes_\ahha
(\phi \circ f)(p_{[0]})(n_{[0]}) \pig) \mpsqact (n_{[1]} p_{[1]})
\\
&
\overset{\scriptscriptstyle{\eqref{leitz1}}}{=} 
&
\pig(f(p_{[0]} \otimes_\ahha n_{[0]})_{[0]} \otimes_\ahha
f(p_{[0]} \otimes_\ahha n_{[0]})_{[1]}
\pig) \mpsqact (n_{[1]} p_{[1]})
\\
&
\overset{\scriptscriptstyle{\eqref{cosmicflux}, \eqref{michigan2}}}{=} 
&
\tau_{\scriptscriptstyle P \otimes_\ahha N} f (p \otimes_\ahha n)
\end{eqnarray*}
\end{small}
for any $p \otimes_\ahha n \in P \otimes_\ahha N$, which proves the commutativity of diagram \eqref{tarrega8} and concludes the proof of this part.

(ii):
Let  $(M, \tau) \in \cZ_{\ucomod^\op}(\ucomod)$
be an object in the bimodule centre.
For any $m \in M$, define $f_m \in \Hom^\uhhu(U, M \otimes_\ahha \due U \lact {})$ by
\begin{equation}
  \label{pennstate}
f_m(u)
 \coloneqq  m_{[0]} \otimes_\ahha m_{[1]}u,
\end{equation}
where the right coaction on the left comodule $M$ is (as always) the induced one \eqref{soestrene}.
The left $U$-colinearity of $f_m$ is a simple check.
However, $f_m$ is not an element in $\HOM^\ell(U,M)$:
formally applying $\gl^\ell$ to $f_m$ produces
$\gl^\ell f_m(u) = u_- m_{[1]-} \otimes_\ahha (m_{[0]} \otimes_\ahha m_{[1]+} u_+)
 = u_- m_{(-1)} \otimes_\ahha f_{m_{(0)}}(u_+)
$
with the help of \eqref{blumare2}, 
which is an element in 
$
\Hom^\uhhu(N, U_\ract \otimes_\ahha \due {(M \otimes_\ahha \due U \lact {})} \blact {} \!),
$
rather than in its subspace
$
U_\ract \otimes_\ahha  \Hom^\uhhu(N,M \otimes_\ahha \due U \lact {}).
$
This can be made even more explicit by considering the special case when $U$  is a Hopf algebra over $A=k$: here, simplifying the $k$-module underlying the left internal Homs by means of the isomorphism $\nu$ from Remark \ref{relajante2}, the map \eqref{pennstate} becomes $f_m(u) = \gve(u)m$ and using the formula for $\gl^\ell$ for this case as given in Remark \ref{castelnuovo}, one obtains $\gl^\ell f_m(u) = S(u) m_{(-1)} \otimes_k m_{(0)}$, hence $\gl^\ell f_m$ manifestly lives in $\Hom_k(U, U \otimes_k M)$ but not in $U \otimes_k \Hom_k(U, M)$. 
%
%

Nevertheless, to $f_m$ we can still formally apply 
$
\tau^{-1}_\uhhu
$
if seen as a map 
$
\Hom^\uhhu(U, M \otimes_\ahha U) \to \Hom_\Aopp(U,M),
$
but not if seen as a map
$
\HOM^\ell(U,M)
     \to
     \HOM^r(U,M),
$
and hence we cannot directly benefit from the properties of central structures
as in  Definition \ref{wisconsin}, which will be needed to prove that Eq.~\eqref{ghettokaisers1} induces a right $U$-action on $M$.
By a standard argument, as in \cite[p.~479]{Sha:OTAYDMCC}, 
this problem is circumvented as follows:
     in general, if $N$ were a finitely $A$-generated comodule, then obviously $\HOM^\ell(N,M) = \Hom^\uhhu(N, M \otimes_\ahha \due U \lact {})$ as comodules. By what is sometimes called the {\em Fundamental Theorem of Comodules} \cite[Thm.~2.1.7]{DasNasRai:HAAI}, every element of a comodule over a $k$-coalgebra (where $k$ is a field) is contained in a finite-dimensional subcomodule. This result can be extended to bialgebroids (or general $A$-corings for that matter) as soon as $U_\ract$ is $A$-projective, which follows from \cite[Cor.~2.7 \& Prop.~2.8]{KaoGomLob:SC}. Hence,
$$
\Hom^\uhhu(N, M \otimes_\ahha \due U \lact {}) = \varprojlim
\HOM^\ell(N_\iota, M),
$$
where the $N_\iota$ are finitely generated left $U$-subcomodules, and similarly
$$
\Hom_\Aopp(N, M) = \varprojlim
\HOM^r(N_\iota, M).
$$
This induces a map $\Hom^\uhhu(N, M \otimes_\ahha \due U \lact {}) \to \Hom_\Aopp(N, M)$ with the same properties as $\tau^{-1}_\enne$ and will therefore, by slight abuse of notation, be denoted by the same symbol. In case $N = U$, this is the map used to define the right $U$-action \eqref{ghettokaisers1} on $M$,
that is,
\begin{equation}
\label{ghettokaisers2}
mu  \coloneqq  (\tau^{-1}_\uhhu f_m)(u), \qquad u \in U, \ m \in M.
\end{equation}
Let us show that this, in fact, defines an action with respect to which the left $U$-comodule $M$ becomes an aYD module: more precisely, we will {\em first} prove that the what-is-going-to-be action \eqref{ghettokaisers2} 
is compatible with the left $U$-coaction on $M$ in the sense of the aYD condition \eqref{funkydesertbreaks1}, or, equivalently, \eqref{funkydesertbreaks1a}. To this end, note that considering $U$ as a left $U$-comodule via the coproduct, the corresponding right coaction obtained from Eq.~\eqref{soestrene} reads
\begin{equation}
  \label{basculer}
  u_{[0]} \otimes_\ahha u_{[1]}  \coloneqq  u_{[+]} \otimes_\ahha u_{[-]}.
 \end{equation}
Moreover, if $\tau$ is a central structure, by definition $\tau^{-1}_\uhhu$ is a left $U$-comodule isomorphism $\Hom^\uhhu(U, M \otimes_\ahha \due U \lact {}) \to \Hom_\Aopp(U,M)$, and therefore satisfies
\begin{equation}
  \label{wald1}
  (\tau^{-1}_\uhhu f)(u_{[+]})_{(-1)} u_{[-]} \otimes_\ahha (\tau^{-1}_\uhhu f)(u_{[+]})_{(0)}
  = f_{(-1)}  \otimes_\ahha (\tau^{-1}_\uhhu f_{(0)})(u)
\end{equation}
with respect to the left $U$-coaction \eqref{teams} on $\Hom^\uhhu(U, M \otimes_\ahha \due U \lact {})$.
Applying this to $f_m$ from \eqref{pennstate} and considering that
\begin{equation}
  \label{wald2}
        {(f_m)}_{(-1)}  \otimes_\ahha {(f_m)}_{(0)}(u) = u_- m_{(-1)} \otimes_\ahha (m_{(0)[0]} \otimes_\ahha m_{(0)[1]} u_+),
        \end{equation}
 as derived from \eqref{teams} and \eqref{blumare2}, we have
for the right hand side in \eqref{wald1}
\begin{small}
\begin{equation}
  \label{wald3}
(\tau^{-1}_\uhhu f_m)(u_{[+]})_{(-1)} u_{[-]} \otimes_\ahha (\tau^{-1}_\uhhu f_m)(u_{[+]})_{(0)} = (mu_{[+]})_{(-1)} u_{[-]} \otimes_\ahha (mu_{[+]})_{(0)},
\end{equation}
\end{small}
 whereas for the left hand side in \eqref{wald1}:
 \begin{equation*}
   \begin{split}
(f_m)_{(-1)}  \otimes_\ahha \tau^{-1}_\uhhu {(f_m)}_{(0)}(u)
&= u_- m_{(-1)} \otimes_\ahha m_{(0)[0](0)} m_{(0)[0](-1)}
     m_{(0)[1]} u_+
\\
  &   =  u_- m_{(-1)} \otimes_\ahha m_{(0)} u_+,
   \end{split}
   \end{equation*}
 with the help of Eq.~\eqref{blumare1}. Hence, \eqref{wald1} implies \eqref{funkydesertbreaks1a} and therefore the aYD condition \eqref{funkydesertbreaks1}, as desired.
 
 To conclude this part, let us show that Eq.~\eqref{ghettokaisers1} resp.~\eqref{ghettokaisers2} effectively defines a right $U$-action, {\em i.e.},
that for any $u, v \in U$
 \begin{equation}
   \label{teich}
   (mu)v
   =
\tau^{-1}_\uhhu f_{\tau^{-1}_\uhhu f_m(u)} (v)
   =
   (\tau^{-1}_\uhhu f_m)(uv)   
   =
   m(uv)
 \end{equation}
holds. 
 To this end,  first note that the right $U$-coaction induced by \eqref{soestrene} on the left $U$-comodule $\Hom^\uhhu(U, M \otimes_\ahha U)$ explicitly reads for the element $f_m$ as follows:
 \begin{equation}
   \label{wald4}
         {(f_m)}_{[0]}(u)  \otimes_\ahha {(f_m)}_{[1]} = (m_{[0]} \otimes_\ahha m_{[1](1)} u_{(1)})
\otimes_\ahha m_{[1](2)} u_{(2)}
         ,
 \end{equation}
 as seen directly by Eqs.~\eqref{wald2}, \eqref{soestrene}, \eqref{blumare2}, and \eqref{Sch8}, whereas in the same spirit Eq.~\eqref{wald2} also implies
 $$
 (\tau^{-1}_\uhhu f_m)(u)_{(-1)} \otimes_\ahha (\tau^{-1}_\uhhu f_m)(u)_{(0)} = (f_m)_{(-1)} u_{(1)}  \otimes_\ahha \tau^{-1}_\uhhu (f_m)_{(0)}(u_{(2)}).
 $$
 by Eqs.~\eqref{Tch2} and \eqref{Tch4},
and therefrom the expression for the right coaction
  \begin{equation}
    \label{wald5}
    \begin{split}
   (&\tau^{-1}_\uhhu f_m)_{[0]}(u) \otimes_\ahha (\tau^{-1}_\uhhu f_m)(u)_{[1]}
\\
      &=
      \tau^{-1}_\uhhu (f_m)_{[0]}(u_{[+]}) \otimes_\ahha u_{[-]} (f_m)_{[1]}
      \\
      &=
      \tau^{-1}_\uhhu \big(m_{[0]} \otimes_\ahha m_{[1](1)}\sma\cdot \big)(u_{[+](1)}) \otimes_\ahha
      u_{[-]} m_{[1](2)} u_{[+](2)}, 
\end{split}
    \end{equation}
  on the element $\tau^{-1}_\uhhu f_m$ in the sense of \eqref{soestrene} again, where Eqs.~\eqref{wald4} \& \eqref{Tch4} were used. Proving the associativity \eqref{teich} now essentially hinges on the fact that if $\tau$ is a central structure, it makes the diagram \eqref{tarrega8} (resp.~\eqref{tarrega7}) commute and is natural: the multiplication $\mu \colon  U_\bract \otimes_\ahha \due U \lact {} \to U$, $u \otimes_\ahha v \mapsto uv$ by the bialgebroid properties is a morphism in $\ucomod$, and hence by \eqref{lebkuchen} we have  $\big(\tau^{-1}_{\scriptscriptstyle{U \otimes_A U}} (f \circ \mu)\big)(u \otimes_\ahha v)  = (\tau^{-1}_{\uhhu} f)(uv)$ for any $f \in \Hom^\uhhu(U, M \otimes_\ahha \due U \lact {})$, which we are going to exploit in the penultimate step of the following computation:  
\begin{small}
\begin{eqnarray*}
&&
  (mu)v
  \\
  &
  \overset{\scriptscriptstyle{\eqref{ghettokaisers2}}}{=} 
  &
  \tau^{-1}_\uhhu f_{\tau^{-1}_\uhhu f_m(u)} (v)
  \\
&\overset{\scriptscriptstyle{\eqref{ghettokaisers2}}}{=}&
  \tau^{-1}_\uhhu \big( \tau^{-1}_\uhhu f_m(u)_{[0]} \otimes_\ahha \tau^{-1}_\uhhu f_m(u)_{[1]} \sma\cdot 
  \big) (v)
   \\
  &
\overset{\scriptscriptstyle{\eqref{wald5}}}{=} 
&
\tau^{-1}_\uhhu\big(\tau^{-1}_\uhhu \big(m_{[0]} \otimes_\ahha m_{[1](1)}\sma\cdot \big)(u_{[+](1)}) \otimes_\ahha
      u_{[-]} m_{[1](2)} u_{[+](2)} \big)(v)
    \\
  &
\overset{\scriptscriptstyle{\eqref{below}}}{=} 
&
\tau^{-1}_\uhhu\big((\Hom^\uhhu(U, \tau^{-1}_\uhhu \otimes_\ahha U)
\circ \psi \circ f_m)'\sma\cdot (u_{[+]})
\\
&& \qquad \qquad \qquad \qquad \qquad \qquad
\otimes_\ahha u_{[-]} (\Hom^\uhhu(U, \tau^{-1}_\uhhu \otimes_\ahha U)
\circ \psi \circ f_m)''\sma\cdot 
\big)(v)
    \\
  &
\overset{\scriptscriptstyle{\eqref{leitz2}}}{=} 
&
\tau^{-1}_\uhhu\big((\gvt \circ \Hom^\uhhu(U, \tau^{-1}_\uhhu \otimes_\ahha U) \circ \psi \circ f_m)(u)\big)(v)
\\
&
\overset{\scriptscriptstyle{\eqref{leitz1}}}{=} 
&
(\phi^{-1} \circ
\Hom_\Aopp(U, \tau^{-1}_\uhhu)
\circ \gvt \circ
\Hom^\uhhu(U, \tau^{-1}_\uhhu \otimes_\ahha U)
\circ \psi \circ f_m \circ \mu)(u \otimes_\ahha v)
\\
&
\overset{\scriptscriptstyle{\eqref{tarrega8}}}{=} 
&
   (\tau^{-1}_{\scriptscriptstyle{U \otimes_A U}} (f_m \circ \mu))(u \otimes_\ahha v)   
\\
&
\overset{\scriptscriptstyle{\eqref{lebkuchen}}}{=} 
&
   (\tau^{-1}_{\uhhu} f_m)(\mu(u \otimes_\ahha v))   
\\
&
\overset{\scriptscriptstyle{\eqref{ghettokaisers2}}}{=} 
&
   m(uv),
\end{eqnarray*}
\end{small}
as claimed.
Here, in the fourth step we additionally needed the fact that
\begin{equation}
  \label{below}
  \begin{split}
  (&\psi f_m)'(v)'(u) \otimes_\ahha (\psi f_m)'(v)''(u) \otimes_\ahha (\psi f_m)''(v)
\\
&= f_m'( u \otimes_\ahha v_{[0]}) \otimes_\ahha f_m''( u \otimes_\ahha v_{[0]})_{(1)} v_{[1]}
\otimes_\ahha f_m''( u \otimes_\ahha v_{[0]})_{(2)}
\\
&=
m_{[0]}
\otimes_\ahha m_{[1](1)} u_{(1)} v_{[+](1)} v_{[-]} \otimes_\ahha m_{[1](2)} u_{(2)} v_{[+](2)}
\\
&=
m_{[0]}
\otimes_\ahha m_{[1](1)} u_{(1)} \otimes_\ahha m_{[1](2)} u_{(2)},
\end{split}
  \end{equation}
as results from Eqs.~\eqref{basculer} and \eqref{Tch2}.

The unitality of the so-defined action once again follows from the naturality \eqref{lebkuchen}: for $N= A$, the source map $s\colon A \to U$ is a morphism in $\ucomod$ as well and therefore
$
\tau^{-1}_\ahha(f \circ s)(a) = (\tau^{-1}_\uhhu f) (s(a))
$
for $f \in \HOM^\ell(U, M)$. Hence, 
$$
m1_\uhhu = (\tau^{-1}_\uhhu f_m)(s(1_\ahha)) = (\tau^{-1}_\ahha (f_m \circ s))(1_\ahha) = m1_\ahha = m,
$$
taking into consideration that $\tau^{-1}_\ahha \colon  \HOM^\ell(A, M) \simeq M
\to \HOM^r(A, M) \simeq M$~is the identity map along with the unitality of the source map, plus the fact that $f_m \circ s$ under the isomorphism $\Hom^\uhhu(A, M \otimes_\ahha\due U \lact {}) \simeq \Hom_\Aopp(A, M) \simeq M$ becomes the map $L_m \colon a \mapsto ma$.

 (iii):
Here, we need to verify two things: first,
that any morphism $M \to \tilde M$ of aYD modules induces a morphism $(M, \tau) \to (\tilde M, \tilde\tau)$ between the corresponding objects in the bimodule centre (and vice versa); second, that the two procedures of how to obtain a central structure from a right $U$-action and a right $U$-action from a central structure are mutually inverse.

As for the first issue, if $\gvf \colon M \to \tilde M$ is a morphism of aYD modules,
 we have to show that for any $N \in \ucomod$ the diagram
\begin{small}
\begin{equation}
\label{verite2}
\xymatrix{
\Hom^\uhhu(N, M \otimes_\ahha U)
  \ar@{<-}[r]^-{\tau_N}
\ar[d]_-{\Hom^\uhhu(N, \gvf \, \otimes_\ahha U) \, }
& 
\Hom_\Aopp(N,M) 
\ar[d]^-{\, \Hom_\Aopp(N, \gvf)}
  \\
  \Hom^\uhhu(N, \tilde M \otimes_\ahha U)
  \ar@{<-}[r]_-{\tilde\tau_N}
  & 
  \Hom_\Aopp(N, \tilde M)
  }
\end{equation}
\end{small}
 commutes, which we, for convenience, will show with respect to $\tau^{-1}$.  
 Indeed, let $n \in N$ and $f \in \Hom^\uhhu(N, M \otimes_\ahha U)$. Then
\begin{small}
  \begin{equation*}
    \begin{split}
 \gvf\big(\tau^{-1}_N f(n) \big)
 &= \gvf\big(f'(n)_{(0)}f'(n)_{(-1)}f''(n) \big)
= (\gvf \circ f')(n)_{(0)} (\gvf \circ f')(n)_{(-1)} f''(n)
\\
& = \tilde\tau^{-1}_N\big((\gvf \circ f') \otimes_\ahha f''\big)(n) 
\\
&= \tilde\tau^{-1}_N \big( \Hom^\uhhu(N, \gvf \otimes_\ahha U) \circ f\big)(n)
 \end{split}
 \end{equation*}
\end{small}
 since $\gvf$ is in particular a morphism of right $U$-modules and left $U$-comodules.

 Vice versa, let $\gvf \colon (M, \tau) \to (\tilde M, \tilde\tau)$ be a morphism of objects in the centre $\cZ_{\ucomod^\op}(\ucomod)$; this, in particular, means that $\gvf$ is a left $U$-comodule map and that the diagram \eqref{verite2} commutes. In order to define a morphism of aYD modules, it suffices to show that $\gvf$ is a right $U$-module morphism as well. To start with, observe that if $\gvf$ is a left $U$-comodule map, one has for $m \in M$
\begin{small}
 \begin{equation*}
   \begin{split}
     \gvf(m_{[0]}) \otimes_\ahha m_{[1]}
     &= \gvf(\gve(m_{(-1)[+]}) m_{(0)}) \otimes_\ahha m_{(-1)[-]} 
\\
&= \gve(m_{(-1)[+]}) \gvf(m_{(0)}) \otimes_\ahha m_{(-1)[-]}
\\
&= \gve(\gvf(m)_{(-1)[+]}) \gvf(m)_{(0)} \otimes_\ahha \gvf(m)_{(-1)[-]}
=      \gvf(m)_{[0]} \otimes_\ahha \gvf(m)_{[1]},
   \end{split}
 \end{equation*}
 \end{small}
that is, it is also a right $U$-comodule morphism with respect to the right coaction \eqref{soestrene}. Applying then diagram \eqref{verite2} to the case $N = U$, we obtain 
\begin{small}
 \begin{equation*}
   \begin{split}
     \gvf(mu) &= \gvf\big(\tau^{-1}_\uhhu f_m(u)\big)
    \\
     &
     =
     \tilde\tau^{-1}_\uhhu \big( \Hom^\uhhu(U, \gvf  \otimes_\ahha U) \circ f_m \big)(u)
     \\
     &
     =
     \tilde\tau^{-1}_\uhhu \big( \gvf(m_{[0]}) \otimes_\ahha m_{[1]}\sma\cdot  \big)(u)
      \\
     &
     =
      \tilde\tau^{-1}_\uhhu \big( \gvf(m)_{[0]} \otimes_\ahha \gvf(m)_{[1]}\sma\cdot  \big)(u)
            \\
     &=
            \tilde\tau^{-1}_\uhhu f_{\gvf(m)}(u)
                       \\
     &=
     \gvf(m)u
  \end{split}
 \end{equation*}
\end{small}
for any $u \in U$.
Hence, $\gvf$ is a also a morphism of right $U$-modules. 

Second, and finally, we have to show that obtaining a central structure from a right $U$-action and a right $U$-action from a central structure are mutually inverse procedures. Indeed, if a right $U$-action $m \otimes u \mapsto mu$ on $M \in \ucomod$ is given and a corresponding central structure $\tau$ is defined by means of Eq.~\eqref{michigan2}, which in turn defines a right $U$-action as in Eq.~\eqref{ghettokaisers1}, we have
$$
(\tau^{-1}_\uhhu f_m)(u) = m_{[0](0)} m_{[0](-1)} m_{[1]} u = mu,
$$
with the help of Eq.~\eqref{blumare1} and \eqref{Tch7}, which is just the right $U$-action that we started with.
Vice versa, given a central structure $\tau$ that defines a right $U$-action as in \eqref{ghettokaisers1} that, in turn, defines a central structure as in \eqref{michigan2}, in a similar way reproduces the central structure $\tau$ we started with. To see this, assume that $\gs$ is the central structure defined by the action \eqref{ghettokaisers1}; we will show now that $\gs = \tau$. Indeed, for $g \in \Hom^\uhhu(N, M \otimes_\ahha \due U \lact {})$, one has, using \eqref{michigan2}
\begin{equation}
  \label{fueller0}
  \gs^{-1}_\enne g(n)
  =  \big(g'(n_{(0)})\gve(g''(n_{(0)}))\big) n_{(-1)} 
  =  \tau^{-1}_\uhhu\big(f_{g'(n_{(0)})\gve(g''(n_{(0)}))} \big) (n_{(-1)}), 
   \end{equation}
where $f_m \in \Hom^\uhhu(U, M \otimes_\ahha \due U \lact {})$ was, as before, the element defined in Eq.~\eqref{pennstate}.

Before we continue, note that any left $U$-coaction $\gl\colon N \to U_\ract \otimes_\ahha N$ on $N$ is a morphism in $\ucomod$ if $U_\ract \otimes_\ahha N$ is seen as a free left $U$-comodule, {\em i.e.}, ignoring the coaction on $N$ and only taking the coproduct on $U$ into account.
From the naturality of a central structure we obtain
\begin{equation}
  \label{fueller1}
\tau^{-1}_\enne(\tilde g \circ \gl) = \tau^{-1}_{\scriptscriptstyle{U \otimes_\ahha N}}(\tilde g) \circ \gl
\end{equation}
for any $\tilde g \in \Hom^\uhhu(U_\ract \otimes_\ahha N, M \otimes_\ahha \due U \lact {})$, along with
\begin{equation}
  \label{fueller2}
\tau^{-1}_{\scriptscriptstyle{U \otimes_\ahha N}} \tilde g(u \otimes_\ahha n) = \tau^{-1}_\uhhu f_{\tilde g'(u_{(2)} \otimes_\ahha n) \gve( \tilde g''(u_{(2)} \otimes_\ahha n))} (u_{(1)}).
\end{equation}
Set then $\tilde g  \coloneqq  \gve \otimes g$. Combining \eqref{fueller1} and \eqref{fueller2} and comparing the outcome
with Eq.~\eqref{fueller0}, we obtain:
\begin{equation*}
\begin{split}
  \tau^{-1}_\enne g(n)
  &= \tau^{-1}_\enne ((\gve \otimes g) \circ \gl)(n)
\\
&= \tau^{-1}_{\scriptscriptstyle{U \otimes_\ahha N}}(\gve \otimes g)(n_{(-1)} \otimes_\ahha n_{(0)})
\\
&= \tau^{-1}_\uhhu\big(f_{g'(n_{(0)})\gve(g''(n_{(0)}))} \big) (n_{(-1)}) 
=
\gs_\enne g(n),
  \end{split}
  \end{equation*}
as desired.

(iv):
Here we only have to add to the preceding parts the proof that 
$
(M, \tau) \in \cZ'_{\ucomod^\op}(\ucomod)
$
if $M$ is stable as an aYD module in the sense of Definition \ref{chelabertaschen2}, and also, vice versa, that starting from such a central object
$
(M, \tau) \in \cZ'_{\ucomod^\op}(\ucomod),
$
one obtains
$m_{(0)} m_{(-1)} = m$, with respect to the action defined in \eqref{ghettokaisers2},
or equivalently $m_{[0]} m_{[1]} = m$, as seen in \eqref{stableagain}.

As for the first issue, suppressing the left and right unit constraints for $\mathbb{1} = A$, we have
\begin{equation*}
  \begin{array}{rcl}
&& (\zeta^{-1} \circ \Hom^\uhhu(A, \tau_\emme) \circ \xi \circ \id_\emme)(m)
\\
    & \overset{\scriptscriptstyle{\eqref{coad2a}}}{=} &
    (\Hom^\uhhu(A, \tau_\emme) \circ \xi \circ \id_\emme)(1_\ahha)'(m)
    \gve\pig((\Hom^\uhhu(A, \tau_\emme) \circ \xi \circ \id_\emme)(1_\ahha)''(m) \pig)
\\
    & \overset{\scriptscriptstyle{\eqref{michigan2}, \eqref{wszw1}, \eqref{Tch8}}}{=} &
    (\xi \circ \id_\emme)(1_\ahha)(m_{[0]})m_{[1]}
\\
& \overset{\scriptscriptstyle{\eqref{coad1}}}{=} &
m_{[0]}m_{[1]}
\\
& \overset{\scriptscriptstyle{\eqref{stableagain}}}{=} &
m
  \end{array}
\end{equation*}
for all $m \in M$,
hence $\zeta^{-1} \circ \Hom^\uhhu(A, \tau_\emme) \circ \xi \circ \id_\emme = \id_\emme$,
which was to prove.

Vice versa, note that if $\id_\emme \in \Hom^\uhhu(M,M)$ is mapped to itself by means of \eqref{subcategory}, then, by virtue of the adjunctions \eqref{coad1} and \eqref{coad2}, 
$
\tau^{-1}_\emme(\id_\emme \otimes_\ahha 1) = \id_\emme.
$
We can then argue as above Eq.~\eqref{fueller1}: the left coaction
$\gl \colon M \to U_\ract \otimes_\ahha M$ is a morphism in $\ucomod$ if $U_\ract \otimes_\ahha M$ is seen as a free left comodule. Define
$
\tilde g \in \Hom^\uhhu(U_\ract \otimes_\ahha M,  M \otimes_\ahha \due U \lact {})  
$
by $\tilde g = \gve \otimes_\ahha \id_\emme \otimes_\ahha 1$ and apply 
\eqref{fueller1} and \eqref{fueller2} to it, observing that
$(\tilde g \circ \gl)(m) = m \otimes_\ahha 1$ and $\tilde g'(u \otimes_\ahha n) \gve( \tilde g''(u \otimes_\ahha n)) = \gve(u)m$; that is, for any $m \in M$, we have
  \begin{equation*}
    \begin{array}{rcl}
m
&=& \tau^{-1}_\emme(\id_\emme \otimes_\ahha 1)(m)
= \tau^{-1}_\emme(\tilde g \circ \gl)(m)
\\
&
\stackrel{\scriptscriptstyle{\eqref{fueller1}}}{=}
&
\tau^{-1}_{\uhhu \otimes_\ahha \emme}(\tilde g)(m_{(-1)} \otimes_\ahha m_{(0)})
\\
&
\stackrel{\scriptscriptstyle{\eqref{fueller2}}}{=}
&
\tau^{-1}_{\uhhu}f_{\gve(m_{(-1)}) m_{(0)}}(m_{(-2)})
=
\tau^{-1}_{\uhhu}f_{m_{(0)}}(m_{(-1)})
\stackrel{\scriptscriptstyle{\eqref{ghettokaisers1}}}{=}
m_{(0)}m_{(-1)}, 
    \end{array}
    \end{equation*}
using $f_{\gve(u)m}\sma\cdot  = f_{m}(\sma\cdot  \ract \gve(u))$, which results from \eqref{pennstate} with \eqref{taklinco2}. Hence, the aYD module $M$ defined in (ii) is stable if $(M, \tau) \in \cZ'_{\ucomod^\op}(\ucomod)$, which concludes the proof.
  \end{proof}

\begin{rem}
  Observe that comparing the situation for $\umod$ and aYD contramodules resp.\ $\ucomod$ and aYD modules is less symmetric than expected: whereas $\umod$ was biclosed in presence of one Hopf structure only (or actually none), this is apparently not the case for $\ucomod$, where left and right Hopf structures seem to be needed.
  \end{rem}

\subsection{Traces on $\ucomod$}
\label{internalflight2}

In a spirit analogous to what was done in \S\ref{internalflight1}, we can now state a dual version of 
Theorem \ref{ganzleerheute} for the comodule case. Lemma \ref{funkyfunk} along with Theorem \ref{tegel2} directly imply:
 
\begin{theorem}
\label{reisecke}
Let an $A$-biprojective  left bialgebroid $(U,A)$ be both left and right
Hopf. If
$M$ is an anti Yetter-Drinfel'd module, then $T \coloneqq  \Hom^\uhhu(-, M)$ yields a weak trace functor $\ucomod \to \kmod$,
which is unital if $M$ is stable (and vice versa).
In particular, there are isomorphisms 
$$
\tr_{\scriptscriptstyle{N,P}} \colon \Hom^\uhhu(N \otimes_\ahha P, M) \stackrel\simeq\lra \Hom^\uhhu(P \otimes_\ahha N, M), 
$$
being functorial in  $N, P \in \ucomod$ that explicitly read
\begin{equation}
\label{hungrr1}
(\tr_{\scriptscriptstyle{N,P}} f)(p \otimes_\ahha n) =  f(n \otimes_\ahha p_{[0]}) p_{[1]}
\end{equation}
for $n \in N$ and $p \in P$.
\end{theorem}

\begin{proof}
  Similar to the proof of Theorem \ref{ganzleerheute}, there are  only two things left to show: first, 
 that in this context the general trace maps \eqref{cupper} assume the explicit form \eqref{hungrr1}. Indeed, for $f \in \Hom_\uhhu(N \otimes_\ahha P, M)$, we have
\begin{equation*}
\begin{array}{rcl}
&&
  (\zeta^{-1} \circ \Hom_\uhhu(N,\tau_\pehhe) \circ \xi \circ f)(p \otimes_\ahha n)
\\
  &
 \stackrel{\scriptscriptstyle{\eqref{coad2a}}}{=}
  &
\big(\Hom_\uhhu(N,\tau_\pehhe) \circ \xi \circ f\big)(n)'(p) \gve\pig(\big(\Hom_\uhhu(N,\tau_\pehhe) \circ \xi \circ f\big)(n)''(p) \pig)
 \\
  &
 \stackrel{\scriptscriptstyle{\eqref{michigan2}, \eqref{wszw1}}}{=}
 &
(\xi \circ f)(n)(p_{[0]})_{[0]}p_{[1][+]} \bract \gve\pig(p_{[1][-]} (\xi \circ f)(n)(p_{[0]})_{[1]}\pig)
 \\
  &
 \stackrel{\scriptscriptstyle{\eqref{Tch7}}}{=}
  &
(\xi \circ f)(n)(p_{[0]})p_{[1]}
 \\
  &
 \stackrel{\scriptscriptstyle{\eqref{coad1}}}{=}
  &
f(n \otimes_\ahha p_{[0]})p_{[1]},
\end{array}
\end{equation*}
 which was to prove.
 Unitality of this trace functor, that is, $\tr_{\scriptscriptstyle A, P} = \id$, is then immediate in case $M$ is stable: since for $N = A$ the left $U$-colinearity of an element $f \in \Hom^\uhhu(P,M)$ also implies right $U$-colinearity in the sense of
$
f(p_{[0]}) \otimes_\ahha p_{[1]} = f(p)_{[0]} \otimes_\ahha f(p)_{[1]}, 
$
we have
\begin{equation*}
(\tr_{\scriptscriptstyle A, P} f)(p) = f(p_{[0]}) p_{[1]} = f(p)_{[0]} f(p)_{[1]}  \stackrel{\scriptscriptstyle{\eqref{stableagain}}}{=}
  f(p)
\end{equation*}
for all $p \in P$.
\end{proof}

\begin{rem}
  \label{quelquechose2}
  Dually to Remark \ref{quelquechose}, this trace functor can be analogously enhanced by introducing more coefficients: if $M$ is an aYD module and $Q$ a Yetter-Drinfel'd module, then $\HOM^r(Q, M)$ is again an aYD module as proven in the second
  part of Corollary \ref{wasasesam2}. Hence, if this aYD module is stable (which is not equivalent to $M$ being stable),
  by
  $
\xi \colon   \Hom^\uhhu(P \otimes_\ahha N \otimes_\ahha Q, M)
  \simeq 
\Hom^\uhhu(P \otimes_\ahha N, \HOM^r(Q, M))
  $, 
  it is possible to construct a trace functor
  $$
  T  \coloneqq  \Hom^\uhhu(- \otimes_A Q, M),
$$
  with $M$ and $Q$ as above, and corresponding trace map
  $$
\tr_{\scriptscriptstyle{N,P}}\colon \Hom^\uhhu(N \otimes_\ahha P \otimes_\ahha Q, M) \stackrel\simeq\lra \Hom^\uhhu(P \otimes_\ahha N \otimes_\ahha Q, M), 
  $$
for arbitrary $N, P \in \ucomod$.
\end{rem}


\appendix

\addtocontents{toc}{\protect\setcounter{tocdepth}{1}}

\section{Left and right Hopf algebroids}
\label{bialgebroids}

\subsection{Bialgebroids}
\label{bialgebroids1}
A left bialgebroid $(U, A, \gD, \gve, s, t)$, introduced first in {\cite{Tak:GOAOAA}} and rediscovered a couple of times, 
is a generalisation of a $k$-bialgebra to a bialgebra object over a noncommutative base ring $A$, consisting of a compatible algebra and coalgebra structure over $\Ae$ resp.\ over $A$. In particular, there is a ring homomorphism resp.\ antihomomorphism $s,t \colon  A \to U$ ({\em source} resp.\ {\em target}) that induce four commuting $A$-module structures on $U$, denoted by
\begin{equation}
  \label{pergolesi}
  a \blact b \lact u \ract c \bract d  \coloneqq  t(c)s(b)us(d)t(a)
\end{equation}
  for $u \in U, \, a,b,c,d \in A$, which we abbreviate by
$
\due U {\blact \lact} {\ract \bract}
$, depending on the relevant action(s) in question.
 Moreover, apart from the multiplication, $U$ also carries a 
 comultiplication $\Delta \colon  U \to U \times_\ahha U \subset \due U {} \ract \otimes_\ahha \due U \lact {}, \ u \mapsto u_{(1)} \otimes_\ahha u_{(2)}$ and a counit $\gve \colon  U \to A$ subject to certain identities that
at some points differ from those in the bialgebra case, see \cite{BoeSzl:HAWBAAIAD, Tak:GOAOAA} or elsewhere. To the $\Ae$-ring
$$
U \times_{\scriptscriptstyle A} U    \coloneqq 
   \big\{ {\textstyle \sum_i} u_i \otimes  v_i \in U_{\!\ract}  \otimes_\ahha \!  \due U \lact {} \mid {\textstyle \sum_i} a \blact u_i \otimes v_i = {\textstyle \sum_i} u_i \otimes v_i \bract a,  \ \forall a \in A  \big\}
$$
we usually refer to as {\em Sweedler-Takeuchi product}.
   
\subsection{Left and right Hopf algebroids} 
\label{bialgebroids2}
Generalising Hopf algebras ({\em i.e.}, bialgebras with an antipode) to noncommutative base rings is a much more challenging task.
If one wants to avoid the abundance of structure maps that accompany the notion of a {\em full} Hopf algebroid as in \cite{BoeSzl:HAWBAAIAD}, that is, {\em two} bialgebroid structures (meaning two coproducts, two counits, eight $A$-actions on the total space, etc.) and an antipode map as sort of intertwiner between these, one renounces on the idea of an antipode and rather requires a certain Hopf-Galois map to be invertible {\cite{Schau:DADOQGHA}}, which even leads to a more general concept than that of full Hopf algebroids. 
More precisely, if $(U, A)$ is a left bialgebroid, consider the maps
\begin{equation}
  \label{nochmehrRegen}
\begin{array}{rclrcl}
\ga_\ell  \colon  \due U \blact {} \otimes_{\Aopp} U_\ract &\to& U_\ract  \otimes_\ahha  \due U \lact,
& u \otimes_\Aopp v  &\mapsto&  u_{(1)} \otimes_\ahha u_{(2)}  v, \\
\ga_r  \colon  U_{\!\bract}  \otimes_\ahha \! \due U \lact {}  &\to& U_{\!\ract}  \otimes_\ahha  \due U \lact,
&  u \otimes_\ahha v  &\mapsto&  u_{(1)}  v \otimes_\ahha u_{(2)},
\end{array}
\end{equation}
of left $U$-modules.
Then the left bialgebroid $(U,A)$ is called a {\em left Hopf algebroid} or simply {\em left Hopf} if $ \alpha_\ell $ is invertible and {\em right Hopf algebroid} or {\em right Hopf}
if this is the case for $\ga_r$. Adopting kind of Sweedler notations
\begin{equation*}
  \begin{array}{rcl}
u_+ \otimes_\Aopp u_-  & \coloneqq &  \alpha_\ell^{-1}(u \otimes_\ahha 1)
\\
   u_{[+]} \otimes_\ahha u_{[-]}  & \coloneqq &  \alpha_r^{-1}(1 \otimes_\ahha u),
\end{array}
  \end{equation*}
with, as usual, summation understood, one proves that for a left Hopf algebroid
\begin{eqnarray}
\label{Sch1}
u_+ \otimes_\Aopp  u_- & \in
& U \times_\Aopp U,  \\
\label{Sch2}
u_{+(1)} \otimes_\ahha u_{+(2)} u_- &=& u \otimes_\ahha 1 \quad \in U_{\!\ract} \! \otimes_\ahha \! {}_\lact U,  \\
\label{Sch3}
u_{(1)+} \otimes_\Aopp u_{(1)-} u_{(2)}  &=& u \otimes_\Aopp  1 \quad \in  {}_\blact U \! \otimes_\Aopp \! U_\ract,  \\
\label{Sch4}
u_{+(1)} \otimes_\ahha u_{+(2)} \otimes_\Aopp  u_{-} &=& u_{(1)} \otimes_\ahha u_{(2)+} \otimes_\Aopp u_{(2)-},  \\
\label{Sch5}
u_+ \otimes_\Aopp  u_{-(1)} \otimes_\ahha u_{-(2)} &=&
u_{++} \otimes_\Aopp u_- \otimes_\ahha u_{+-},  \\
\label{Sch6}
(uv)_+ \otimes_\Aopp  (uv)_- &=& u_+v_+ \otimes_\Aopp v_-u_-,
\\
\label{Sch7}
u_+u_- &=& s (\varepsilon (u)),  \\
\label{Sch8}
\varepsilon(u_-) \blact u_+  &=& u,  \\
\label{Sch9}
(s (a) t (b))_+ \otimes_\Aopp  (s (a) t (b) )_-
&=& s (a) \otimes_\Aopp s (b)
\end{eqnarray}
are true \cite{Schau:DADOQGHA}, 
where in  \eqref{Sch1}  we mean the Takeuchi-Sweedler product
\begin{equation*}
\label{petrarca}
   U \! \times_\Aopp \! U    \coloneqq 
   \big\{ {\textstyle \sum_i} u_i \otimes v_i \in {}_\blact U  \otimes_\Aopp  U_{\!\ract} \mid {\textstyle \sum_i} u_i \ract a \otimes v_i = {\textstyle \sum_i} u_i \otimes a \blact v_i, \ \forall a \in A \big\},
\end{equation*}
and if the left bialgebroid $(U,A)$ is right Hopf, in the same spirit one verifies  
\begin{eqnarray}
\label{Tch1}
u_{[+]} \otimes_\ahha  u_{[-]} & \in
& U \times^{\scriptscriptstyle A} U,  \\
\label{Tch2}
u_{[+](1)} u_{[-]} \otimes_\ahha u_{[+](2)}  &=& 1 \otimes_\ahha u \quad \in U_{\!\ract} \! \otimes_\ahha \! {}_\lact U,  \\
\label{Tch3}
u_{(2)[-]}u_{(1)} \otimes_\ahha u_{(2)[+]}  &=& 1 \otimes_\ahha u \quad \in U_{\!\bract} \!
\otimes_\ahha \! \due U \lact {},  \\
\label{Tch4}
u_{[+](1)} \otimes_\ahha u_{[-]} \otimes_\ahha u_{[+](2)} &=& u_{(1)[+]} \otimes_\ahha
u_{(1)[-]} \otimes_\ahha  u_{(2)},  \\
\label{Tch5}
u_{[+][+]} \otimes_\ahha  u_{[+][-]} \otimes_\ahha u_{[-]} &=&
u_{[+]} \otimes_\ahha u_{[-](1)} \otimes_\ahha u_{[-](2)},  \\
\label{Tch6}
(uv)_{[+]} \otimes_\ahha (uv)_{[-]} &=& u_{[+]}v_{[+]}
\otimes_\ahha v_{[-]}u_{[-]},  \\
\label{Tch7}
u_{[+]}u_{[-]} &=& t (\varepsilon (u)),  \\
\label{Tch8}
u_{[+]} \bract \varepsilon(u_{[-]})  &=&  u,  \\
\label{Tch9}
(s (a) t (b))_{[+]} \otimes_\ahha (s (a) t (b) )_{[-]}
&=& t(b) \otimes_\ahha t(a),
\end{eqnarray}
see \cite[Prop.~4.2]{BoeSzl:HAWBAAIAD},
where in  \eqref{Tch1} we denoted
\begin{equation*}  \label{petrarca2}
   U \times^{\scriptscriptstyle A} U    \coloneqq 
   \big\{ {\textstyle \sum_i} u_i \otimes  v_i \in U_{\!\bract}  \otimes_\ahha \!  \due U \lact {} \mid {\textstyle \sum_i} a \lact u_i \otimes v_i = {\textstyle \sum_i} u_i \otimes v_i \bract a,  \ \forall a \in A  \big\}.
\end{equation*}
If the left bialgebroid $(U,A)$ is simultaneously left and right Hopf, the compatibility between the two (inverses of the) Hopf-Galois maps comes out as:
\begin{eqnarray}
\label{mampf1}
u_{+[+]} \otimes_\Aopp u_{-} \otimes_\ahha u_{+[-]} &=& u_{[+]+} \otimes_\Aopp u_{[+]-} \otimes_\ahha u_{[-]}, \\
\label{mampf2}
u_+ \otimes_\Aopp u_{-[+]} \otimes_\ahha u_{-[-]} &=& u_{(1)+} \otimes_\Aopp u_{(1)-} \otimes_\ahha u_{(2)}, \\
\label{mampf3}
u_{\smap} \otimes_\ahha u_{{\smam}+} \otimes_\Aopp u_{[-]-} &=& u_{(2)[+]} \otimes_\ahha u_{(2)[-]} \otimes_\Aopp u_{(1)},
\end{eqnarray}
see \cite[Lem.~2.3.4]{CheGavKow:DFOLHA}.
A simultaneous left and right Hopf structure on a left bialgebroid still does not imply the existence of an antipode required in the definition of a full Hopf algebroid.
For example, the universal enveloping algebra $V\!L$ of a Lie-Rinehart algebra $(A,L)$ constitutes a left bialgebroid that
is both left and right Hopf but still does not admit an antipode in general.

However, in case $(U, A) = (H, k)$ is actually a Hopf algebra over a field $k$, the invertibility of $\ga_\ell$ guarantees the existence of the antipode $S$ and the invertibility of $\ga_r$ the existence of $S^{-1}$. More precisely, in these cases 
for any $h \in H$ we had
\begin{equation}
  \label{sesam}
  \begin{array}{rcl}
    h_+ \otimes_k h_- &=& h_{(1)} \otimes S(h_{(2)})
    \\[2pt]
    h_{[+]} \otimes_k h_{[-]} &=& h_{(2)} \otimes S^{-1}(h_{(1)}).
  \end{array}
  \end{equation}

\addtocontents{toc}{\SkipTocEntry}
\section*{Acknowledgements}
The author would like to thank the referee for valuable comments that helped to improve both content and presentation.

\providecommand{\bysame}{\leavevmode\hbox to3em{\hrulefill}\thinspace}
\providecommand{\MR}{\relax\ifhmode\unskip\space\fi M`R }
\providecommand{\MRhref}[2]{%
  \href{http://www.ams.org/mathscinet-getitem?mr=#1}{#2}}
\providecommand{\href}[2]{#2}

\end{document}